%% file: main.tex
\documentclass[11pt]{article}

\include{mycommands_ams}

\author{Simon-Christian Klein\footnote{simon-christian.klein@tu-bs.de}
			}

\date{\today}

\title{Stabilizing Discontinuous Galerkin Methods Using Dafermos' Entropy Rate Criterion: II -- Systems of Conservation Laws and Entropy Inequality Predictors}

\begin{document}

	\maketitle
	\abstract{A novel approach for the stabilization of the Discontinuous Galerkin method based on the Dafermos entropy rate crition is presented. First, estimates for the maximal possible entropy dissipation rate of a weak solution are derived. Second, families of conservative Hilbert-Schmidt operators are identified to dissipate entropy. Steering these operators using the bounds on the entropy dissipation results in high-order accurate shock-capturing DG schemes for the Euler equations, satisfying the entropy rate criterion and an entropy inequality.
}
	\section{Introduction} \label{sec:intro}
	\input{intro}
	\section{Entropy inequality predictors} \label{sec:theory}
	\input{theory}
	\section{Suitable dissipation directions and filtering} \label{sec:dispd}
	\input{dispd}
	\section{Numerical tests} \label{sec:NT}
	\input{tests}
	\section{Conclusion} \label{sec:concl}
	\input{concl}
		
	\section{Competing Interests}
		The author has no relevant financial or non-financial interests to disclose.
	\section{Data Availability}
		The commented implementation of the schemes is available under  \newline \href{https://github.com/simonius/dgdafermos}{https://github.com/simonius/dgdafermos}.
	\section{Bibliography}
	\bibliographystyle{plainnat}
	\bibliography{lit}
\end{document}

%% file: mycommands_ams.tex
\usepackage{amsthm}
\usepackage{amsmath}

\usepackage{graphicx}
\usepackage[english]{babel}
\usepackage[a4paper]{geometry}

\usepackage{caption}
\usepackage{comment}
\usepackage{subcaption}
\usepackage{amssymb}
\usepackage{mathtools}
\usepackage{mathptmx} 
\usepackage{xcolor}
\usepackage{tikz}
\usetikzlibrary{positioning}
\usepackage[square,numbers,sort]{natbib}
\usepackage{dsfont}
\usepackage{hyperref}

\newcommand{\derive}[2] {{\frac{\partial {#1} }{\partial {#2}}}}
\newcommand{\derd}[2] {{\frac{\vd #1 }{\vd #2}}}

\newcommand{\T}{\mathcal T}

\newcommand{\R}[0]{\mathbb{R}}

\newcommand{\Z}{\mathbb{Z}}
\newcommand{\C}{\mathbb{C}}

\newcommand{\abs}[1]{\left | #1\right| }

\newcommand{\norm}[1]{\left \|#1  \right \|}
\newcommand{\CC}{\mathcal{C}}
\newcommand{\set}[2]{ \left \{ #1 ~\middle|~  #2 \right\}}
\newcommand{\sset}[1]{ \{ #1\}}

\newcommand{\Leb}{\mathrm{L}}

\DeclareMathOperator{\im}{im}
\newcommand{\vd}{ \mathrm{d}}
\newcommand{\intd}{\,\mathrm{d}}

\newcommand{\e}{\mathrm{e}}
\renewcommand{\epsilon}{\varepsilon}
\renewcommand{\phi}{\varphi}
\newcommand{\skp}[2]{\left \langle {#1}, {#2} \right \rangle}
\newcommand{\bskp}[2]{\left [ {#1}, {#2} \right ]}

\DeclareMathOperator{\argmin}{argmin}
\DeclareMathOperator{\bigO}{\mathcal{O}}

\DeclareMathOperator{\lspan}{span}
\DeclareMathOperator{\op}{\mathbb{P}}
\DeclareMathOperator{\ip}{\mathbb{I}}
\DeclareMathOperator{\ev}{\derive{U}{u}}
\DeclareMathOperator{\Id}{\mathrm{I}}

\newcommand{\rskp}[2]{\left(#1\middle)\cdot\middle(#2\right)}
\newcommand{\drskp}[2]{#1 \cdot #2}
\DeclareMathOperator{\ch}{ch}

\theoremstyle{definition}
\newtheorem{definition}{Definition}
\theoremstyle{theorem}
\newtheorem{lemma}{Lemma}
\theoremstyle{remark}

\theoremstyle{theorem}
\newtheorem{theorem}{Theorem}

\newtheorem{corollary}{corollary}


%% file: intro.tex
Discontinuous Galerkin methods \cite{CS2001RKDG} are a popular tool to design numerical schemes for hyperbolic systems of conservation laws \cite{Dafermos2016Hyper}
\begin{equation} \label{eq:HCL}
	\derive{f(u)}{x} + \derive{u}{t} = 0 \quad \text{for} \quad  u(x, t): \R \times \R \to \R^m, \quad f:\R^m \to \R^m.
\end{equation}
\begin{table}
	\begin{tabular}{c | c}
		A cell of the subdivision $\T$ of the domain $\Omega$ & $T$ \\
		The left and right boundaries of cell $T$ & $T_l$, $T_r$\\
		The set of test functions & $\mathcal{D}$ \\
		The space of ansatz functions for an cell $T$& $V^T = \lspan\sset{\phi_1, \phi_2, \dots, \phi_N}$ \\
		Polynomials of degree of $p$ in cell $T$ & $V^{T, p}$\\
		$\Leb^2$ projection of $u$ onto $V$ & $\op_V u$.\\
		Interpolation of $u$ on $V$ w.r.t. the collocation points $(\xi_k)_{k=1}^N$& $\ip_V u$ \\ 
		Vector of nodal values in cell $T$ at time $t$& $u^T(t)$ \\
		ansatz function in cell $T$ at position $x$ and time $t$ & $u^T(x, t)$ \\
		Inner product on cell $T$ & $\skp{u}{v}_T = \int_T \rskp{u(x)} {v(x)} \intd x$\\
		Surface inner product on cell $T$ & $\bskp{v}{f}_T = \int_{\partial T} \rskp v f \intd O$\\
		Gramian Matrices on cell $T$ &  $	M_{k, l}^T = \skp{\phi_k}{\phi_l}_T,\, S_{k, l}^T = \skp{\derive{\phi_k}{x}}{\phi_l}_T $\\
		The total entropy in cell $T$ & $E_{u, T}(t)$ \\
		The discrete total entropy in cell $T$ & $E^T(t)$ \\
		The inner product on $T$ discretised using $\omega_k$ & $\skp{u}{v}_{T, \omega} = \sum_k \omega_k u_k v_k$\\
		Entropy variables in cell $T$ & $\derive{U}{u}(u^T(x, t))$ \\
		Vector of nodal values of the entropy variables in cell $T$ at $t$ & $\ev^T(t)$ \\
		Interpolation of the entropy variables in cell $T$ on $V$& $\ev^T(x, t)$ \\
		The canonical inner product between $a, b \in \R^n$ & $\drskp{a}{b}$ or $\rskp{a}{b}$ \\
		The inner product between $u$ and $v$ on cell $T$ & $\skp{u}{v}_T$ \\
		The $p$ norm of $u$ in cell $T$ & $\norm{u}_{T, p}$ \\
		Exact solution to the initial condition $u(x, t_0)$ after $t-t_0$ & $H(u(\cdot, t_0), t-t_0)$. \\
		Mean value of subcell $k$ of $N$, $u^T$ as initial condition & $u^{T, N}_k$ \\
		The convex hull of a set $A$ & $\ch A$ \\
	\end{tabular}
	\caption{Notation used. As a general rule, quantities with only $t$ as an argument are vectors of nodal values at a certain time. Values with $x$ and $t$ in their argument list are functions that were evaluated at these values. Objects with $T$ added as exponent are approximations of the quantity in the cell $T$.}
	\label{tab:notation}
\end{table}
An intriguing feature of DG methods is their ability to transfer the definition of a weak solution to a hyperbolic conservation law \cite{GR91}
\begin{equation}
		\label{eq:weaksol}
	\begin{aligned}
\forall \phi \in \mathcal{D}:	\int_0^\infty \int_\R \rskp{u(x, t)}{ \derive{\phi(x, t)}{t}} &+ \rskp{ f \circ u(x, t)}{ \derive{\phi(x, t)}{x}}  \intd x \intd t \\
	&+ \int_\R \rskp{u(x, 0)}{\phi(x, 0)} \intd x = 0. 
	\end{aligned}
\end{equation}
to the semidiscrete level \cite{Cockburn1989DGI,CockburnShu1989DGI,ShuDGReview}. Using a method of lines approach this leads to the set of equations

\[
\forall T \in \T, \phi \in \mathcal{D}: \skp{\derive{\phi}{x}}{f}_T - \bskp{\phi}{f}_T - \skp{\phi}{\derive{u^T}{t}}_T = 0
\]
for every cell $T \in \T$ of a subdivision $\T$ of the domain into cells. The solution $u(x, t)$ is approximated in every cell by $u^T(x, t) \in V^T$ out of a finite dimensional space of ansatz functions $V^T$.
Using an approximation of the inner products as point evaluations results in the matrix vector form
\begin{equation} \label{eq:DGMatForm}
M^T \derd{u^T}{t} = S^T f(u^T(t)) - \begin{pmatrix} \phi^T_1(x_r)f^*_r-\phi^T_1(x_l) f^*_l \\ \vdots \\  \phi^T_N(x_r) f^*_r  - \phi^T_N(x_l) f^*_r\end{pmatrix}.
\end{equation}
Sadly, the constructed schemes lack robustness and stability in the high order case and some stabilization measures and robustness enhancements are needed. Popular are overintegration, flux-differencing, modal filtering, sub-cells and (W)ENO recoveries \cite{GOS2018Modal,RGOS2018Stab,Gassner,ShuDGReview,ZHU20114353,LUO2007686,MS2014FVS}. In this publication the procedure first presented in \cite{klein2023stabilizing} will be refined, connections to some other stabilization techniques will be shown, and the technique will be tested on a catalog of problems for the Euler system of conservation laws. The method in \cite{klein2023stabilizing} is based on the entropy rate admissibility criterion  \cite{Dafermos72,Feireisl2014MD}. An entropy \cite{Lax71} is a convex functional $U:\R^m \to \R$ satisfying
\[
	\derd f u \derd U u = \derd F u
\]
in conjunction with a entropy flux function $F:\R^m \to \R$.
One can show that for this pair $(F,U)$ holds
\begin{equation} \label{eq:ceieq}
	\derive {U(u(x, t))}{t} + \derive{F}{x} \leq 0
\end{equation}
in the sense of distributions \cite{Lax71}. If the solution is smooth one can even show
\[
		\derive {U(u(x, t))}{t} + \derive{F}{x} = 0.
\] 
The entropy rate criterion states that the total entropy 
\[
E_u(t) = \int U(u(x, t)) \intd x
\]
of the selected weak solution $u$ should reduce faster than the entropy of any other existing weak solution $\tilde u$
\[
\forall t > 0: \quad \derd{E_u}{t} \leq \derd{E_{\tilde u}}{t}. 
\]
A numerical approximation of this total entropy can be defined as 
\begin{equation} \label{eq:dEdef}
	E_{u, T}(t) = \int_T U(u^T(x,  t)) \intd x \approx \sum_{k} \omega^T_k U(u^T(x_k, t)), \quad E_{u} = \sum_{T \in \T} E_{u, T}
\end{equation}
via a (positive) quadrature rule $\omega^T_k$ on each cell $T \in \T$.
The numerical enforcement of the criterion with respect to such a definition of the discrete entropy happened in \cite{klein2023stabilizing} in three steps
\begin{itemize}
	\item Calculate the time derivative of the ansatz function $\derd {\tilde u^T} t$ using a DG scheme
	\item Calculate an error prediction $\delta^T$ for $\derd {\tilde u^T} t$ on $T$, i.e. $\norm{\derd{\tilde u^T} t - \derive{u} t}_T \leq \delta^T$.
	\item Correct the time derivative into the direction of the steepest entropy descent 
	\[
	\derd {u^T} t = \derd {\tilde u^T} t - \frac{\delta^T}{\norm{h^T}_T} h^T,
	\]
		where $h$ shall be the steepest descent direction that does not change the average value in cell $T$.
\end{itemize}
	While the approach above is successful for scalar conservation laws
 \cite{klein2023stabilizing} significant improvements can be made by introducing two refinements. The first one concerns the usage of an error indicator to estimate the entropy correction needed. We will instead show that it is possible to directly give bounds on how dissipative a weak solution can be. This will eliminate the need for the error indicator while allowing a faster convergence, because the derived bounds converge to zero significantly faster in the smooth case. A second refinement concerns the direction used for the entropy correction. DG methods can make use of modal filtering to remove unwanted high frequency modes from the solution \cite{RGOS2018Stab}. These filters can be sometimes expressed as viscosity, and we will devise correction directions that at the same time dissipate entropy and filter the solution from unwanted oscillations and thereby combine the dissipation and filtering.

Our schemes will therefore follow the slightly different general layout of
\begin{itemize}
	\item Calculate a time derivative for the ansatz function $\derd {\tilde u^T} t$
	\item Estimate the highest possible entropy dissipation speed $\sigma^T$ in cell $T$
	\item Calculate the correction direction $\upsilon^T$
	\item Calculate the size $\lambda^T$  of the correction needed to achieve that 
	\[
		\derive{u^T}{t} = \derive{\tilde u^T}{t} + \lambda^T \upsilon^{T}
	\]
	satisfies
	\begin{equation}
		\label{eq:lambdadef}
 \derd {E_{u, T}} t = \skp{\derd U u}{\derive{\tilde u^T}{t} + \lambda^T \upsilon^T}_T \leq \sigma^T + F^*_l - F^*_r.
	\end{equation}
	the dissipation mandated by the estimate. Here $F^*$ shall be a numerical entropy flux \cite{Tadmor1984I, Tadmor1984II, Tadmor1987}.
	\end{itemize}
The procedure makes only use of the fact that in our cells there exist local ansatz functions and is therefore also applicable to similar schemes like the spectral volume (SV) method \cite{Wang2002SV}. The only difference would lie in the evaluation of a different scheme for the uncorrected derivative $\derd{\tilde u} {t}$. One complication is brought in by the fact that entropy dissipation implies the non-smoothness of the solution, as otherwise the entropy equality applies. Therefore, dissipation can't happen in cells in the continuous setting, as polynomials are smooth. Instead, dissipation is a process taking place at the cell edges were our different ansatz functions transition.  As we are not correcting the numerical fluxes used between cells dissipation will be centered in cells and not at cell edges, and we will show in section \ref{sec:CorIP} how to work around this problem. 

%% file: theory.tex
\subsection{Bounds for entropy and entropy dissipation}
Our main tool to approximate the most dissipative weak solution using a DG method will be a bound on the derivative of the total entropy. We will derive a lower bound for the entropy dissipation
\[
	s^\theta_u(t_1, t_2) = \int_{t_1}^{t_2}\int_{\theta} \derive{U}{t} + \derive{F}{x} \intd x \intd t \leq 0.
\]
Here $\theta \subset \Omega$ shall be an arbitrary open subdomain of the complete domain. This value has to be smaller than zero for a solution that is admissible with respect to the classical entropy inequality \eqref{eq:ceieq}. Further, we are interested in the entropy dissipation speed
\[
	\sigma^\theta_u(t) = \derive{s^\theta_u(t_0, t)}{t} =  \int_{\theta} \derive{U}{t} + \derive{F}{x} \intd x  \leq 0.
\] If this value is known one can estimate the total entropy's $E_u(t)$ derivative as
\[
	E_u(t) \geq \sum_{\theta \in \Theta} s^\theta_u(0, t), \quad	\derd{E_u}{t} \geq \sum_{\theta \in \Theta} \sigma^\theta_u,
\]
when $\Theta = \{ \theta_1, \theta_2, \theta_3, \dots, \theta_L\}$ is overlapping $\Omega$ in the sense of 
\[
\Omega \subset \bigcup_{\theta \in \Theta} \theta.  
\]

To achieve our goal of estimating $s^\theta$ we will view the problem in the setting of classical Finite-Volueme schemes \cite{SONAR201655} and go over to the limit $\Delta x \to 0$. In \cite{Dafermos2009MDR} it was shown that for scalar conservation laws the flux $f$ of the solution to the Riemann problem $u_\text{R}(u_l, u_r; x, t)$ is given by 
\[
	f\left(\argmin_{u \in \ch(u_l, u_r)} \skp{\derd U u (u_r) - \derd U u (u_l)}{f(u)} \right),
\]
i.e. by entering the value of $u$ into $f$ that when entered into the flux yields the fastest entropy dissipation. In \cite{KS2023EAR} it was shown that some approximate Riemann solvers, for example the local Lax-Friedrichs flux, can be also interpreted as approximate solutions to such variational descriptions of two-point fluxes. While the aforementioned results hold for semidiscrete schemes the new results below are new and aim at three point first order Finite-Difference/Finite-Volume schemes for systems of conservation laws. As one assumes piecewise constant functions in those first order methods any quadrature exact for constants will yield the same result in equation \eqref{eq:dEdef}. As we only look at discrete time values in this part of the publication we will write $E^n_u = E_u(t_n)$ for the discrete total entropy at time level $n$.
\begin{lemma} \label{lem:LFspeed}
	Let a system of hyperbolic conservation laws in conservation form and a strictly convex entropy pair $(U, F)$ be given that is approximated by a Finite-Volume scheme with grid constant $\lambda = \frac{\Delta t}{\Delta x}$. Then the original Lax-Friedrichs scheme has the fastest dissipation of the total entropy
	\[
	E^n_u = \sum_{k} U(u^n_k) \Delta x
	\]
	 under all consistent and conservative three-point numerical schemes.
	\end{lemma}
\begin{proof}
	Assume $f(u_k, u_{k+1})$ is a consistent numerical two point flux minimizing the total entropy with maximal rate and let $u_l, u_r$ be arbitrary in the domain of admissible values for the conserved variables. We apply a scheme using this flux to a Riemann problem, i.e. the initial data 
		\[
		u^0_{k} = \begin{cases}
					u_l & k \leq 0 \\
					u_r & k > 0 \\
					\end{cases}
		\] 
		to query the flux value $f(u_l, u_r) = f(u_0, u_1) = f_{\frac 1 2}$ by analyzing the solution. As the flux is consistent it holds 
		\[  
		\forall k < 0: f\left(u_{k}, u_{k+ 1}\right)= f_{k + \frac 1 2} = f(u_l), \quad  \forall  k > 0: f\left(u_{k}, u_{k+ 1}\right) = f_{k + \frac 1 2} = f(u_l).
		\]
		 The scheme
		\[
			u^1_k = u^0_k + \lambda \left(f_{k - \frac 1 2} - f_{k + \frac 1 2}\right)
		\] 
		therefore implies that $u^1_k = u^0_k$ for all $k \not \in \{0, 1\}$. The total entropy 
		\[
		E^1_u = E^0_u - \Delta x \left(U\left(u^0_0\right) + U\left(u^0_1\right)\right) + \Delta x \left(U\left(u^1_0\right) + U\left(u^1_1\right) \right)
		\]
		 is minimized by $u^1_0 = u^1_1$, as $U$ is strictly convex. Entering this into the scheme's definition with $u^0_0 = u_l$ and $u^0_1 = u_r$ implies
		\[
			u_l + \lambda(f(u_l) - f(u_l, u_r)) = u_r + \lambda(f(u_l, u_r) - f(u_r)).
		\]
		Rearranging for $f(u_l, u_r)$ shows
		\[
			f(u_l, u_r) = \frac{f(u_l) + f(u_r) }{2} +  \frac{u_l - u_r}{2 \lambda},
		\] 
		and this is the classical Lax-Friedrichs flux and therefore uniquely determined by demanding maximal entropy rate. \qed
	\end{proof}
This result shows that the classical LF scheme is the most direct realisation of a scheme satisfying Dafermos' entropy rate criterion and therefore justifies the use of the LF scheme in \cite{Klein2022Using} as the most dissipative scheme possible for systems of conservation laws. Similar results are also known for scalar conservation laws. Tadmor showed in \cite{Tadmor1984I, Tadmor1984II} that every monotonicity preserving scheme satisfying classical numerical entropy inequalities for a scalar conservation law has a viscosity coefficient less or equal to that of the LF scheme, and higher or equal than the viscosity coefficient of Godunov's scheme. Our result can be seen as a generalization of the LF part of this result to systems of conservation laws, as it states that the LF flux is the most dissipative flux for a selected time-step size.
Using the scheme above one can derive estimates for the highest possible entropy dissipation in a time-step and using finite-differencing of this result, approximations for the lowest possible derivative of the total entropy with respect to time. 
\begin{corollary} \label{cor:eiep}
	The biggest possible entropy dissipation during a discrete time-step of a Finite-Volume scheme with grid constant $\lambda = \frac{\Delta t}{\Delta x}$ is given by the difference
	\[
		E^{n+1} - E^{n} = \sum_k U\left (\frac{u_{k -1}^n + u_{k +1}^n}{2} + \lambda \left(f\left(u_{k-1}^n\right) + f\left(u_{k+1}^n\right)\right) \right) \Delta x - U\left(u_{k}^n\right) \Delta x
	\]
	and an approximation to the total entropy's minimal derivative by
	\[
		\derd{E}{t} \approx \sum_k \frac {U\left(\frac{u_{k -1} + u_{k +1}}{2} + \lambda (f(u_{k-1}) + f(u_{k+1}))\right) - U(u_{k})}{\lambda}.
	\]
	\end{corollary}
The second estimate above degenerates for $\lambda \to 0$ as the difference in entropy is in general finite between cells $u_{k-1}, u_k, u_{k +1}$. A second, more refined, estimate is given by the following lemma based on the ideas from \cite{HLL1984HLL} and does not have these deficiencies.
\begin{lemma} \label{lem:HLLspeed}
	Given bounds on the fastest signal speed to the left $a_l$ and the highest signal speed to the right $a_r$ let $M \geq \max(\abs{a_l}, \abs{a_r})$. The maximum entropy dissipation of a Riemann problem solution on the interval $\theta = (-M, M)$ is bounded from below by
	\[
		E^{\theta}_u(t) - E^{\theta}_u(0) \geq  t \left((a_r - a_l)U\left ( u_{lr} \right)  +a_l U(u_l)  - a_rU(u_r)  \right)
	\]
	with
	\[
		u_{lr} = \frac{a_r u_r - a_l u_l+ f(u_l) - f(u_r)}{a_r - a_l}.
	\]
	The rate is bounded from below by
		\[
	\begin{aligned}
		\derd{E^{\theta}}{t}|_{t = 0} \geq (a_r - a_l) U(u_{lr}) + a_l U(u_l) - a_r U(u_r) 
	\end{aligned}.
	\]
	The entropy dissipation is bounded by
	\[
		s^\theta \geq t \left((a_r - a_l) U(u_{lr}) + a_l U(u_l) - a_r U(u_r) +F(u_l) - F(U_r) \right),
	\]
	and its rate by
	\[
	\sigma^\theta \geq (a_r - a_l) U(u_{lr}) + a_l U(u_l) - a_r U(u_r) +F(u_l) - F(U_r). 
	\]
	\end{lemma}
	\begin{proof}
		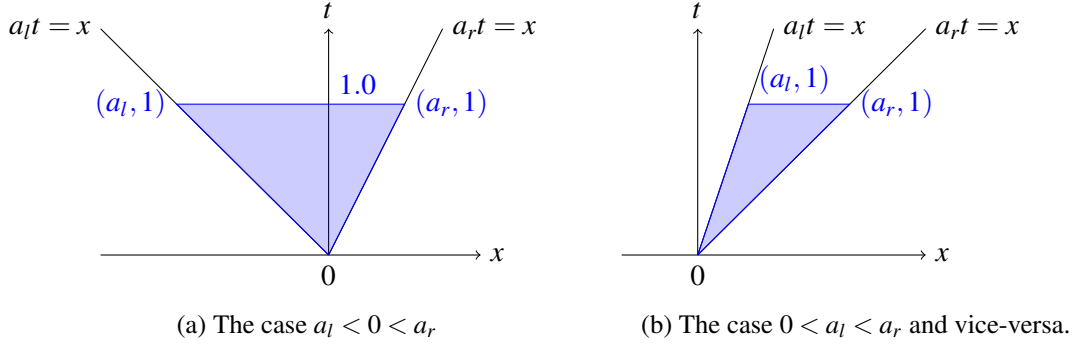
\begin{figure}
			\begin{subfigure}{0.55\textwidth}
			\begin{tikzpicture}
				\draw [->] (0.0, 0.0) node [below] {$0$} -- (0.0, 3.0) node [above] {$t$};
				\draw [->] (-3.0, 0.0) -- (2.0, 0.0) node [right] {$x$};
				\draw (0.0, 0.0) -- (-3.0, 3.0) node [left] {$a_l t = x$};
				\draw (0.0, 0.0) -- (1.5, 3.0) node [right] {$a_r t = x$};
				\filldraw [blue, fill opacity=0.2] (0.0, 0.0) -- (1.0, 2.0) node [right, fill opacity=1.0] {$(a_r, 1)$}
									-- (0.0, 2.0) node [above right, fill opacity=1.0]{$1.0$} 
									-- (-2.0, 2.0) node [left, fill opacity=1.0] {$(a_l, 1)$} 
									-- (0.0, 0.0);
						
			\end{tikzpicture}
			\caption{The case $a_l < 0 < a_r$}
		\end{subfigure}
		\begin{subfigure}{0.42\textwidth}
			\begin{tikzpicture}
				\draw [->] (0.0, 0.0) node [below] {$0$} -- (0.0, 3.0) node [above] {$t$};
				\draw [->] (-1.0, 0.0) -- (3.0, 0.0) node [right] {$x$};
				\draw (0.0, 0.0) -- (1.0, 3.0) node [right] {$a_l t = x$};
				\draw (0.0, 0.0) -- (3.0, 3.0) node [right] {$a_r t = x$};
				\filldraw [blue, fill opacity=0.2] (0.0, 0.0) -- (2.0, 2.0) node [right, fill opacity=1.0] {$(a_r, 1)$}
				
				-- (0.666666, 2.0) node [above right, fill opacity=1.0] {$(a_l, 1)$} 
				-- (0.0, 0.0);
			\end{tikzpicture}
		\caption{The case $0 < a_l < a_r$ and vice-versa.}
		\end{subfigure}
		\caption{Layout of the integration areas in the proof (blue). As originally used in \cite{HLL1984HLL}. To the left and right of the lines $\derd{x}{t} = a_l$ and $\derd{x}{t} = a_r$ the initial condition is unaltered.}
		\label{fig:HLLidea}
		\end{figure}
		The entropy of the initial condition in the interval $[-M, M]$ is given by
		\[
			\int_{-M}^M U(u(x, 0)) \intd x = M U(u_l) + MU(u_r)
		\]
		for any $M > 0$. Integrating over the triangle $T = \ch\{(0, 0), (a_l, 1), (a_r, 1)\}$ in spacetime and using the conservation law yields
		\[
			\begin{aligned}
			0 &= \int_T \derive{u}{t} + \derive{f}{x} \intd V(x, t) = \int_{\partial T} \drskp {\begin{pmatrix} f \\ u \end{pmatrix}}{ n} \intd O(x, t) \\
			 &= \int_{a_l}^{a_r} \drskp {\begin{pmatrix} f(u) \\ u(x, 1) \end{pmatrix}}{ \begin{pmatrix} 0 \\ 1  \end{pmatrix}} \intd x
			 + \int_0^1 \drskp {\begin{pmatrix} f(u) \\ u(t a_l, t) \end{pmatrix}}{ \begin{pmatrix} -1 \\ a_l \end{pmatrix}} \intd t 
			+ \int_0^1 \drskp {\begin{pmatrix} f(u) \\ u(t a_r, t) \end{pmatrix}}{ \begin{pmatrix} 1 \\ -a_r  \end{pmatrix}} \intd t \\
			&= (a_r - a_l) u_{lr} + a_l u_l - a_r u_r + f(u_r) - f(u_l),
			\end{aligned}
		\]  
		in conjunction with the Gauß divergence theorem, cf.~figure \ref{fig:HLLidea}. Here $u_{lr}$ shall denote the mean value of $u(x, 1)$ on $[a_l, a_r]$ and is
		\[
			u_{lr} = \frac{a_r a_l - a_l u_l + f(u_l) - f(u_r)}{a_r - a_l}
		\] 
		as apparent from the calculation above.
		Jensens inequality implies
		\begin{equation} \label{eq:HLLjensen}
			\begin{aligned}
			t(a_r - a_l)U(u_{lr}) &= t(a_r - a_l) U \left( \frac{1}{t(a_r - a_l)} \int_{ta_l}^{ta_r} u(x, t) \intd x \right)\\
								 &\leq \frac{t(a_r - a_l)}{t(a_r - a_l)} \int_{ta_l}^{ta_r} U(u(x, t)) = E^{( ta_l, ta_r)}_u(t).
			\end{aligned}
		\end{equation}
		Therefore it follows
		\begin{equation} \label{eq:HLLest}
			\begin{aligned}
			E^\theta(1) - E^\theta(0) \geq& (a_r - a_l) U(u_{lr}) + (M - a_r) U(u_r) + (M + a_l) U(u_l) \\
			&- M (U(u_l) + U(u_r)) \\
			 	 		=& (a_r - a_l) U(u_{lr}) + a_l U(u_l) - a_r U(u_r)
			\end{aligned}
		\end{equation}
		for the entropy dissipation between $t=0$ and $t=1$ and using the invariance under transformations $(x, t) \mapsto (\mu x, \mu t)$ for $\mu > 0$ yields
		\begin{equation}\label{eq:HLLsemiest}
			\begin{aligned}
			\derd{E}{t}|_{t = 0} \geq (a_r - a_l) U(u_{lr}) + a_l U(u_l) - a_r U(u_r) 
			\end{aligned}
		\end{equation}
		for the rate. To calculate the entropy dissipation $s^\theta$ and its speed $\sigma^\theta$ we just have to account for the entropy flowing in and out of the intervall $\theta$ using the entropy flux $F$. This is possible as $u$ is constant to the left of $(ta_l, t)$ and to the right of $(t a_r, t)$. \qed
		\end{proof}
	
		The estimate above does not depend on any grid constant, and reduces to the previous one for $-a_l = c_\text{max} = a_r$, $\lambda c_\text{max} = 1$, and this is the CFL condition for the classical Lax-Friedrich scheme, i.e. both estimates are compatible.
		A Godunov type scheme using the HLL approximate Riemann solver is also compatible with the estimate above. The discrete total entropy after one time-step is still less or equal than the bound given above.
		\begin{figure}
			\begin{tikzpicture}
			\draw [->] (-5.0, 0.0) -- (5.0, 0.0) node [right] {$x$};
			\draw [->] (-5.0, 0.0) -- (-5.0, 2.0) node [above] {$t$};
			\draw (-3.0, 0.0) -- (-3.0, 2.0) ;
			\draw (-1.0, 0.0) -- (-1.0, 2.0);
			\draw (1.0, 0.0) -- (1.0, 2.0);
			\draw (3.0, 0.0) -- (3.0, 2.0);
			\draw [dotted] (-3.0, 0.0) -- (-4.0, 2.0);
			\draw [dotted] (-3.0, 0.0) -- (-2.0, 2.0);
			
			\draw [dotted] (-1.0, 0.0) -- (-0.5, 2.0);
			\draw [dotted] (-1.0, 0.0) -- (-0.0, 2.0);
			
			\draw [dotted] (1.0, 0.0) -- (0.5, 2.0);
			\draw [dotted] (1.0, 0.0) -- (1.5, 2.0);
			
			\draw [dotted] (3.0, 0.0) -- (2.0, 2.0);
			\draw [dotted](3.0, 0.0) -- (2.5, 2.0);
			
			\draw [->] (-5.0, -3.0) -- (5.0, -3.0) node [right] {$x$};
			\draw [->] (-5.0, -3.0) -- (-5.0, -0.5) node [above] {$u(x, t_{n+1})$};
			\draw (-3.0, -3.0) -- (-3.0, -0.0) ;
			\draw (-1.0, -3.0) -- (-1.0, -0.0);
			\draw (1.0, -3.0) -- (1.0, -0.0);
			\draw (3.0, -3.0) -- (3.0, -0.0);
			\draw 	(-5.0, -2.5) -- (-4.0, -2.5) -- 
					(-4.0, -2.0) -- (-2.0, -2.0) --
					(-2.0, -1.5) -- (-0.5, -1.5) -- 
					(-0.5, -1.0) -- (0.0, -1.0) --
					(0.0, -1.5) -- (0.5, -1.5) --
					(0.5, -2.0) -- (1.5, -2.0) -- 
					(1.5, -2.5) -- (2.0, -2.5) --
					(2.0, -1.5) -- (2.5, -1.5) --
					(2.5, -1.0) -- (5.0, -1.0);
			\end{tikzpicture}
			\caption{A set of noninteracting HLL approximate Riemann solutions.}
			\label{fig:HLLsolv}
		\end{figure}
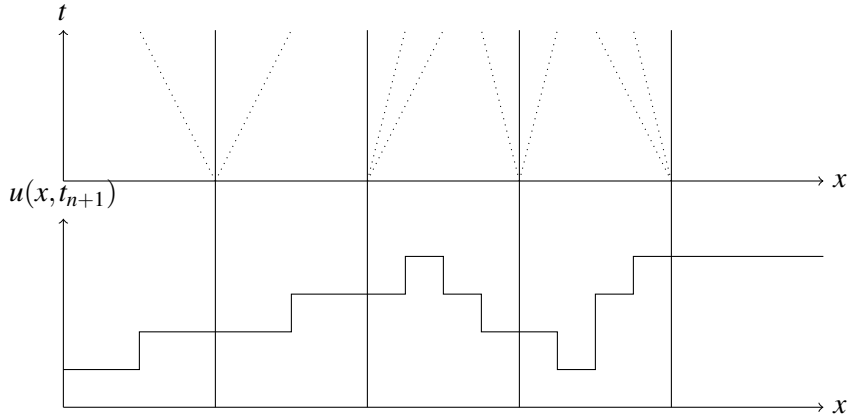
		Let $\lambda c_{\text{max}} \leq \frac 1 2$ hold implying that the Riemann problems do not interact and $u^\mathrm{HLL}(x, t_{n+1})$ be the picewise constant solution of the HLL solver as in figure \ref{fig:HLLsolv}, but not averaged over the cells, while $u^{n+1}_k$ shall be the corresponding cell averages. In this case the total discrete entropy at the next time-step is given by
		\[
			\begin{aligned}
			E^{n+1}_{\mathrm{FV}} =& \sum_k \Delta x U(u_k^{n+1}) \leq \sum_k \frac{\Delta x}{\Delta x} \int_{x_{k-\frac 1 2}}^{x_{k + \frac 1 2}} U(u^{HLL}(x, t)) \intd x\\
			 =& \int_\Omega U(u^{HLL}(x, t)) \intd x \leq E^{n+1}_{u^{\mathrm{HLL}}}.
			\end{aligned}
		\]
		Therefore the discrete entropy of the approximate solution is lower than the entropy of any exact weak solution.
		The next subsection will move beyond first order schemes by generalizing this lower bound to one that also allows smooth solutions instead of piecewise constant ones.

			\subsection{Asymptotic analysis based entropy inequality predictor}
			\begin{figure}
				\begin{subfigure}{0.48\textwidth}
				\begin{tikzpicture}[scale=0.8]
					\draw [->](-3.0, 0) -- (3.0, 0.0) node [right] {$x$};
					\draw [->] (-3.0, 0.0) -- (-3.0, 3.0) node [above] {$u(x, 0)$};
					\draw [dotted] (0.0, 0.0) -- (0.0, 3.0);
					\draw [red] plot [smooth] coordinates {(-3,0) (-2,1) (-1,1) (0,2)};
					\draw [blue] plot [smooth] coordinates {(0,1) (1,0) (2,1) (3,2)};
				\end{tikzpicture}
				\caption{The problem for the generalized entropy inequality predictor. Two piecewise smooth solutions are spliced together and we are interested in the local residual of the entropy equality around the interface.}
				\label{fig:eieq}
					\end{subfigure}
					\begin{subfigure}{0.48\textwidth}
					\begin{tikzpicture}[scale=0.8]
						\draw [->] (-3.0, 0) -- (3.0, 0.0) node [right] {$x$};
						\draw [->] (0.0, 0.0) -- (0.0, 3.0) node [above] {$t$};
						\draw plot [smooth] coordinates {(0,0) (-1, 1.5) (-2.5, 3.0)};
						\draw plot [smooth] coordinates {(0,0) (1, 1.5) (1.5, 3)};
						\draw [dashed] (0.0, 0.0) -- (-3.0, 3.0) node [left] {$a_l$};
						\draw [dashed] (0.0, 0.0) -- (2.5, 3.0) node [right]{$a_r$};
							
					\end{tikzpicture}
					\caption{Solutions of generalized Riemann problems lack scaling invariance. Still we assume that there exist bounds $a_l$ and $a_r$ that the waves from the interaction of $u_l(x)$ and $u_r(x)$ do not leave the cone $[ta_l, ta_r]$ for small t.}
					\label{fig:grHLL}
				\end{subfigure}
					\begin{subfigure}{0.48\textwidth}
					\begin{tikzpicture}[scale=0.8]
					\draw [->](-3.0, 0) -- (3.0, 0.0) node [right] {$x$};
					\draw [->] (-3.0, 0.0) -- (-3.0, 3.0) node [above] {$u(x, t)$};
					\draw [dotted] (-1.0, 0.0) -- (-1.0, 3.0) node [above] {$t a_l$};
					\draw [dotted] (0.5, 0.0) -- (0.5, 3.0) node [above] {$t a_r$};
					\draw [red] plot [smooth] coordinates {(-3,0) (-2,1) (-1,1)};
					\draw [blue] plot [smooth] coordinates {(0.5, 1.0) (1,0.25) (2,1) (3,2)};
					\draw (-1.0, 1.25) -- (0.5, 1.25);
					\end{tikzpicture}
					\caption{The assumed solution used in the generalized entropy inequality predictor. Note that this solution follows the HLL idea of assuming a constant function in the wedge formed by $t a_l$ and $t a_r$.}
					\label{fig:GHLL}
				\end{subfigure}
				\begin{subfigure}{0.48\textwidth}
					\begin{tikzpicture}[scale=0.8]
					\draw [->](-3.0, 0) -- (3.0, 0.0) node [right] {$x$};
					\draw [->] (-3.0, 0.0) -- (-3.0, 3.0) node [above] {$u(x, t)$};
					
					\draw [blue] plot [smooth] coordinates {(-2.5, 1.0) (-2.0,0.5) (-1.5,1.5)};
					\draw [blue] plot [smooth] coordinates {(-1.5, 1.25) (-1.0,1.0) (-0.5,0.75)};
					\draw [blue] plot [smooth] coordinates {(-0.5, 1.0) (-0.0,1.5) (0.5,1.0)};
					\draw [blue] plot [smooth] coordinates {(0.5, 1.0) (1.0,0.75) (1.5,1.25)};
					\draw [blue] plot [smooth] coordinates {(1.5, 0.5) (2.0,0.75) (2.5,1.5)};
					
					\draw [dotted] (-2.5, 0.0) -- (-2.5, 2.5) node [above] {$x_{k - \frac 5 2}$};
					\draw [dotted] (-1.5, 0.0) -- (-1.5, 2.5) node [above] {$x_{k-\frac 3 2}$};
					\draw [dotted] (-0.5, 0.0) -- (-0.5, 2.5) node [above] {$x_{k - \frac 1 2}$};
					\draw [dotted] (0.5, 0.0) -- (0.5, 2.5) node [above] {$x_{k+\frac 1 2}$};
					\draw [dotted] (1.5, 0.0) -- (1.5, 2.5) node [above] {$x_{k + \frac 3 2}$};
					\draw [dotted] (2.5, 0.0) -- (2.5, 2.5) node [above] {$x_{k + \frac 5 2}$};
					
					\draw (-2.0, 0.15) -- (-2.0, -0.15) node [below] {$x_{k-2}$};
					\draw (-1.0, 0.15) -- (-1.0, -0.15) node [below] {$x_{k-1}$};
					\draw (0.0, 0.15) -- (0.0, -0.15) node [below] {$x_{k}$};
					\draw (1.0, 0.15) -- (1.0, -0.15) node [below] {$x_{k-1}$};
					\draw (2.0, 0.15) -- (2.0, -0.15) node [below] {$x_{k-2}$};
					
					\draw (-2.1, -0.7) -- (-2.1, -1.0) -- (-1.5, -1.0) node [below] {$\theta_{k-\frac 3 2}$} -- (-0.9, -1.0) -- (-0.9, -0.7);
					\draw (-1.1, -1.7) -- (-1.1, -2.0) -- (-0.5, -2.0) node [below] {$\theta_{k-\frac 1 2}$} -- (0.1, -2.0) -- (0.1, -1.7);
					\draw (-0.1, -0.7) -- (-0.1, -1.0) -- (0.5, -1.0) node [below] {$\theta_{k+\frac 1 2}$} -- (1.1, -1.0) -- (1.1, -0.7);
					\draw (0.9, -1.7) -- (0.9, -2.0) -- (1.5, -2.0) node [below] {$\theta_{k+\frac 3 2}$} -- (2.1, -2.0) -- (2.1, -1.7);
					\end{tikzpicture}
					\caption{Application of the entropy inequality predictor to an open overlap $\theta_{k + \frac 1 2}$ of the domain - centered on the discontinuities.}
					\label{fig:eieqApp}
				\end{subfigure}
				\caption{Construction and application of the generalized HLL dissipation estimate.}
			\end{figure}
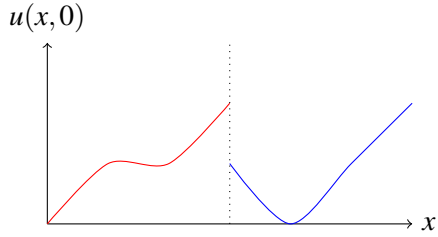
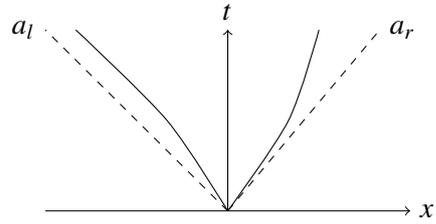
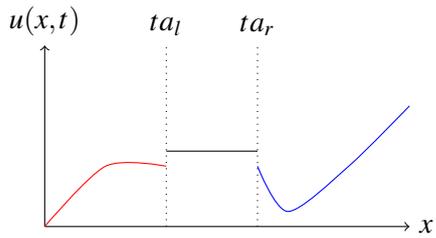
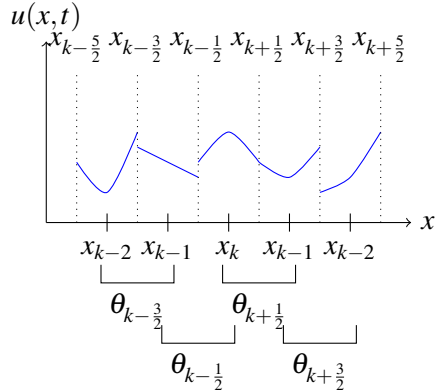
			The entropy inequality predictor in this section will be based on an asymptotic analysis of the problem described in figure \ref{fig:eieq}, i.e. two smooth solutions splined together at an interface. An obstacle lies in the missing self-similarity. This is a difference to the previous part where the self-similarity of the initial condition and assumed self-similarity of the solution induced the existence of a self-similar, i.e. constant, speed of the entropy dissipation. We will therefore try to approximate
			\[
					s^\theta(t_1, t_2) =  \int_{t_1}^{t_2} \int_\theta \derive F x + \derive U t \intd x \intd t
			\]
			for reasonably small $\abs{t_2 - t_1}$ and the discontinuity at the interface in the interior of $\theta$. The schemes in which we  will use these entropy inequality predictors should converge with high orders for smooth solutions, necessitating a convergence of the predictor to zero with a high order for smooth solutions. This convergence is also dictated by the entropy equality for smooth solutions. 
			If $u_l(x)$ and $u_r(x)$ are piecewise constant this problem is already solved by the methods described in the last subsection. We will therefore now reiterate through the proof of lemma \ref{lem:HLLspeed} assuming that $u_l(x)$ and $u_r(x)$ are smooth functions. The missing self-similarity of Generalized Riemann problems \cite{BA2011GRP}, cf.~figure \ref{fig:grHLL},
			defies the existence of the speed estimates $a_l$ and $a_r$, and we therefore just assume that these speed estimates exist for small times. Further we assume that for small times the solutions left of $(ta_l, t)$ and right of $(ta_r, t)$ remain smooth, as no waves from the interaction arrive there and $u_l(x), u_r(x)$ have bounded derivatives.
		
			The average value $u_{lr}$ shall be determined by applying the conservation law to the triangle $T = \ch\{(0, 0), (t a_l, t), (t a_r, t)\}$
				\[
			\begin{aligned}
				0 =& \int_T \derive{u}{t} + \derive{f}{x} \intd V(x, t) = \int_{\partial T} \drskp {\begin{pmatrix} f \\ u \end{pmatrix}}{ n} \intd O(x, t) \\
				=& \int_{t a_l}^{t a_r} \drskp {\begin{pmatrix} f(u) \\ u(x, t) \end{pmatrix}}{ \begin{pmatrix} 0 \\ 1  \end{pmatrix}} \intd x
				+ \int_0^t \drskp {\begin{pmatrix} f(u) \\ u(\tau a_l, \tau) \end{pmatrix}}{ \begin{pmatrix} -1 \\ a_l \end{pmatrix}} \intd \tau \\
				&+ \int_0^t \drskp {\begin{pmatrix} f(u) \\ u(\tau a_r, \tau) \end{pmatrix}}{ \begin{pmatrix} 1 \\ -a_r  \end{pmatrix}} \intd \tau \\
				=& \underbrace{\int_{ta_l}^{ta_r} u(x, t) \intd x}_{t(a_r - a_l)u_{lr}} 
				+ \int_0^t f(u_r(\tau a_r, \tau)) - f(u_l(\tau a_l, \tau)) \intd \tau \\
				&- \int_{t a_l}^0 U(u_l(x, 0)) \intd x - \int_{0}^{t a_r} U(u_r(x, 0)) \intd x.
			\end{aligned}
			\] 
			Dividing this equation by $t$ and going over to the limit $t \to 0$ results in 
			\[
			\begin{aligned}
					\frac{\int_0^t f(u_r(\tau a_r, \tau)) - f(u_l(\tau a_l, \tau)) \intd \tau} {t} &\xrightarrow{t \to 0} f(u_r(0, 0)) - f(u_l(0, 0)), \\
					 \frac{\int_{t a_l}^0 U(u_l(x, 0)) \intd x - \int_{0}^{t a_r} U(u_r(x, 0)) \intd x}{t}& \xrightarrow{t \to 0} a_l U(u_l(0, 0)) - a_r U(u_r(0, 0))
			\end{aligned}
			\]
			using the continuity of the integrands and the mean value theorem of integration \cite{Lax1976Calculus}. Therefore it follows
			\[
				u_{lr}(t) \xrightarrow{t \to 0} \frac{a_r u_r(0) - a_l u_l(0) + f(u_l( 0)) - f(u_r( 0))}{a_r - a_l}
			\]
			for vanishing $t$. Equation \eqref{eq:HLLjensen} stays also valid in the case of piecewise polynomial functions as initial conditions and for small $t > 0$. We can therefore conclude that a generalization of equation \eqref{eq:HLLest} holds in the form
			\[
			\begin{aligned}
				E^\theta(t) - E^\theta(0) \geq& t(a_r - a_l) U(u_{lr}) + \int_{-M}^{t a_l} U(u(x, t)) \intd x + \int_{t a_r}^M U(u(x, t)) \intd x\\
				 &- \int_{-M}^M U(u(x, 0)) \intd x.
			\end{aligned}
			\]
			Accounting for the entropy flowing in and out of $[-M, M]$ yields
			\[
			\begin{aligned}
				s^\theta(0, t) \geq& t(a_r - a_l) U(u_{lr}) + \int_{-M}^{t a_l} U(u(x, t)) \intd x + \int_{t a_r}^M U(u(x, t)) \intd x\\
				&- \int_{-M}^M U(u(x, t)) \intd x + \int_0^t F(u(-M, \tau)) - F(u(M, \tau)) \intd \tau.
				\end{aligned}
			\]	
			Applying the entropy equality to the subdomains $[-M, ta_l] \times [0, t]$ and $[ta_r, M] \times [0, t]$
			\[
			\int_{-M}^{t a_l} U(u(x, t)) \intd x - \int_{-M}^{t a_l} U(u(x, 0)) \intd x = \int_0^{t} F(u(-M, \tau)) - F(u(t a_l, \tau)) \intd \tau,
			\] 
			that holds for small $t > 0$ because the solution stays smooth in the subdomains, allows us to restate this as
			\[
			\begin{aligned}
			s^\theta(0, t) \geq t(a_r - a_l) U(u_{lr}) - \int_{ta_l}^{ta_r} U(u(x, t)) \intd x 
				 - \int_0^t F(u(t a_l, \tau)) - F(u(t a_r, \tau)) \intd \tau.
			\end{aligned}
			\]		
			Dividing by $t$ and going over to the limit, using the limit of $u_{lr}$ and once more the mean value theorem, shows in this case also
			\begin{equation} \label{eq:mepHLL}
			\sigma^\theta \geq (a_r - a_l) U(u_{lr}) - a_r U(u_r(0)) + a_l U(u_l(0)) - F(u_l(0)) - F(u_r(0)).
			\end{equation}
			
			A significant problem of the derivation above lies in the fact that one can only estimate the entropy dissipation speed in the interval $\theta = (-M, M)$, but not in $(-M, 0)$ as the true dissipation can be located anywhere in the cone $[ta_l, ta_r]$. As the cells in our numerical tests will be layed out as in figure \ref{fig:eieqApp}
			\[
				\T = \set{T_k = \left[x_{k-\frac 1 2}, x_{k + \frac 1 2}\right]}{k \in \Z}, \quad x_{k - \frac 1 2} < x_{k+ \frac 1 2}
			\]
			is a suitable set of overlapping open intervals
			\[
				\Theta = \set{\theta_{k + \frac 1 2} = (x_{k} - \epsilon, x_{k+1} + \epsilon)}{k \in \Z}, \quad x_{k} = \frac{x_{k-\frac 1 2} + x_{k + \frac 1 2}}{2}.
			\]
			We are therefore left with the problem of how to split this dissipation onto the two neighboring cells that have overlap with $\theta_{k + \frac 1 2}$.
			This problem will be handled below in section \ref{sec:CorIP}.
			
			\subsection{Accounting for aliasing errors}
			In \cite{klein2023stabilizing} one of the findings in the numerical tests section was that the entropy dissipation of the numerical solutions started already shortly prior a real entropy dissipating discontinuity formed. This was attributed to the fact that while the entropy of the exact solution is still constant as long as the solution is smooth this exact solution will in general not be representable in our ansatz space. It is therefore wise to dissipate entropy to arrive at a function that still lies in our space, and certainly better than selecting an ansatz function that has more entropy than the true solution. A similar issue could be the fact that in the $\Leb^p$ norms, for $p < \infty$, near each piecewise continuous solution $u^T$ lies a $\CC^\infty$ function that can be constructed via mollification. Therefore an infinitely small perturbation of $u^T$ in the usual norms leads to a vanishing entropy dissipation. Or, put differently, the dissipation bound as a functional is discontinuous in the $L^p$ spaces. While unsatisfactory let us remark that the functional is better behaved with respect to the $BV$ semi norms. The discontinuity of the entropy dissipation bound is problematic with under-resolved solutions where a lucky, or in this case better to be considered unlucky, too smooth approximation of the solution in our piecewise polynomial spaces induces wrong, i.e. too conservative entropy dissipation predictions.
			
				 We are therefore interested in allowing our entropy inequality predictor to be also greedy, or one could say pessimistic, with respect to an under-resolved solution. The key to this strengthening is the following lemma.
			\begin{lemma}[Order of the entropy dissipation bound] \label{lem:eds}
				The maximal entropy dissipation prediction \eqref{eq:mepHLL} of a Riemann problem for a smooth flux function with smooth entropy-entropy flux pair vanishes quadratically with the jump of $u$ at the interface
				\[
					\abs{\sigma^\theta} \in \bigO (\norm {u_l - u_r}^2).
				\]
				
				\end{lemma}
			\begin{proof}
				As the entropy inequality holds it is clear that the entropy dissipation is non-positive in the sense of distributions. As we only allow entropy dissipative solutions the entropy dissipation on $\theta$ is a non-positive constant for a fixed jump. On the contrary \eqref{eq:mepHLL} has to be zero for $u_l = u_r$ and is smooth, implying that the line $u_l = u_r$ consists of local maxima. Therefore, a power expansion of \eqref{eq:mepHLL} around $u_l$ in $u_r$ has to be of the form
				\[
					\sigma^\theta(u_l, u_r) = \rskp{u_l - u_r} {H (u_l - u_r)} + \bigO{\norm {u_l - u_r}^3}
				\]
				with a negative semi-definite Hessian $H \in \R^{m \times m}$. This proofs the claim. \qed
				\end{proof}
				We are therefore in the relaxing position that even if our approximations of $u_l, u_r$ only satisfy $\norm{u_l - u_r} \in \bigO ((\Delta x)^p)$ the corresponding estimate will converge significantly faster with $\sigma^\theta(u_l, u_r) \in \bigO ((\Delta x)^{2p})$. Our basic DG method predicts values for our solution $u^T$ in a Hilbertspace that is spanned by polynomials on every cell. In this case a suitable orthonormal basis is spanned by Legendre polynomials and the limits of these basis representations are $\Leb^2$ functions. But as explained before our functional is not continuous on $\Leb^2$ and our ansatz $u^T$ is only an approximation of a projection of the true solution onto our ansatz space. We can therefore try to exploit different projections of our ansatz, especially projections that assume less regularity of $u^T$, and estimate our entropy dissipation with the strongest one encountered in all of these different approximations of $u_l$ and $u_r$. A natural choice for projections on spaces assuming less regularity are projections on lower order polynomials. As the Legendre polynomials on each cell when truncated up to polynomial $p$ are an orthogonal basis of the polynomials with degree less than or equal to $p$, is the orthogonal projection onto these spaces given by discarding the higher order coefficients in the Legendre expansion of $u^T$. We can truncate down to order $p-1$ by discarding the highest coefficient  and still achieve a convergence order of at least $q = 2p-2 > p$ of our entropy inequality predictor for $p>2$.  This can be summed up in the following procedure used above order $p = 2$.
				\begin{itemize}
					\item Assign $\sigma^{\theta_{k + \frac 1 2}}_p = \sigma^{\theta_{k + \frac 1 2}}_{u(x, t)}$.
							\item Project the ansatz $u^T$ in every cell onto $V^{T, p-1}$ using an orthonormal projection \[u^{T, p-1} = \op_{V^{T, p-1}} u(\cdot, t).\]
							\item  Assign $\sigma^{\theta_{k + \frac 1 2}}_{p-1} = \sigma^{\theta_{k + \frac 1 2}} _{ u^{p-1}(\cdot, t) }$.
						
					\item Use $\min\left(\sigma^{\theta_{k + \frac 1 2}}_p,\sigma^{\theta_{k + \frac 1 2}}_{p-1}\right)$ as entropy inequality prediction.
				\end{itemize}

%% file: dispd.tex
\begin{figure}
	\begin{tikzpicture}
	\filldraw (0.0, 0.0) circle (0.1) node [left] {$u^T$};
	
	\draw [->](0.0, 0.0) -- (1.0, 1.0) node [right] {$- \derd {U} {u}$};
	\draw (-3.0, 3.0) -- (1.5, -1.5) node [right] {$E_u = \text{const.}$};
	
	\draw [->] (0.0, 0.0) -- (-2.0, 3.0) node [right]{$-\derive f x$};
	\draw [->] (0.0, 0.0) -- (-3.0, 2.5) node [left] {$\derd {u^T}{t}$};
	\draw [->] (0.0, 0.0) -- (1.5, 0.7) node [right] {$\nabla^2 u^T$};
	\draw [->] (0.0, 0.0) -- (0.7, -1.0) node [left] {$\nabla^8 u^T $};
	\end{tikzpicture}
	\caption{Use of alternative dissipation directions. The direction $-\derive{f}{x}$ shall denote the $\Leb^2$ projection of the exact solutions' derivative $\derive u t$ onto $V$. Our approximation $\derd  {u^T} t$ is not entropy dissipative in this example and should be corrected into the entropy dissipative half-space characterized by the normal $-\derd U u$. The direction $\nabla^2 u$ that stems from a discretization of the heat kernel is suitable for this correction and has additional benefits compared to $-\derd U u$. While the diffusion also has a smoothing effect the addition of $-\derd U u$ can even result in a sharpening effect. Higher even derivatives like $\nabla^8 u$ will smooth the solution but will not result in a dissipation for all entropies.}
	\end{figure}
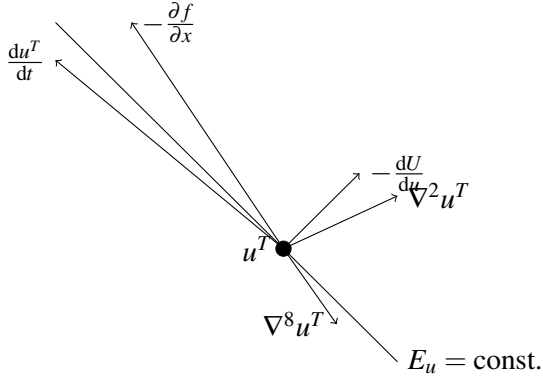

After deriving approximations for the entropy dissipation needed we will now determine how to correct the time derivative of the DG scheme to dissipate the amount of entropy needed. At the same time the resulting scheme is hopefully still high order accurate for entropy conservative solutions. In the scalar case the direction of the steepest descent of the entropy, corrected for conservation, was used for this purpose. This approach incurs several problems:
\begin{itemize}
	\item The direction of steepest entropy descent has in general no smoothing/filtering effect. 
	\item Proving in the previous publication that a correction in the steepest descent direction with the length taken from the error indicator results in enough entropy dissipation was possible but resulted in highly technical arguments\cite{klein2023stabilizing}. 
	\item Dissipation stems from the viscous and parabolic history of hyperbolic conservation laws. A viscous flux
		\[
				f_{\epsilon}\left (u, \derive{u}{x}\right ) = f(u) - \epsilon \derive{u}{x}
		\]
		associated with a viscous regularization of a hyperbolic conservation law is proportional to the gradient of the solution for fixed viscosity and proportional to the gradient of the solution for constant viscosity. If a component of $u$ is smooth with a low magnitude of the first and second derivative the viscous flux of this component will also only differ from the hyperbolic flux by a small margin. If our scheme is corrected with the steepest entropy descent direction one can ask if this correction can be expressed using some viscosity distribution $\epsilon(x)$ in the domain. This will be false in general. Even worse, the steepest gradient descend of the entropy can't be bounded using the first derivative of the respective component of the vector valued function $u(x, t)$, incurring an infinitely large viscosity.
\end{itemize}
All of the above reasons motivate us to devise alternative directions for the entropy correction. These alternative directions should have the following properties
\begin{itemize}
	\item The dissipation direction should have a filtering effect, i.e. when the direction only is used high order modes should be dissipated.
	\item The direction should dissipate entropy. 
	
	\item The dissipation should stem from a viscosity added to the hyperbolic flux.
\end{itemize}

Our new correction directions will be based on the construction of filters, i.e. operators that can regularize a solution $u$. A filter will in our case be a special Hilbert-Schmidt operator $K$ \cite{LaxFun}. 

\begin{definition}[Filter]
	An operator $K: \Leb^2(\Omega) \to \Leb^2(\Omega)$ is said to be a filter if it is an integral operator whose pointwise evaluation results in a weighted average, i.e. 
	\[
	[K u] (x) = \int_\Omega  k(x, y) u(y) \intd y, \quad \text{with } \forall x \in \Omega: \int_\Omega k(x, y) \intd y = 1.
	\] 
	is satisfied and the kernel $k$ is of bounded Hilbert-Schmidt norm.
\end{definition}
We are especially interested in conservative filters as they do not destroy the conservation of our basic schemes when they are applied on a per cell basis.
\begin{lemma}[Conservative filter]
	A filter $K: \Leb^2(\Omega) \mapsto \Leb^2(\Omega)$ is conservative
	\[
		\int_\Omega [K u](x) \intd x = \int_\Omega u(x) \intd x
	\]
	 if it can be written as an integral operator with a kernel with mass one, i.e.
	\[
		[K u] (x) = \int_\Omega u(y) k(x, y) \intd y, \quad \text{with } \forall y\in \Omega: \int_\Omega k(x, y) \intd x = 1
	\] 	
	\end{lemma}
\begin{proof}
	Using Fubini's theorem shows
	\[
		\int_\Omega [K u](x) \intd x = \int_\Omega \int_\Omega k(x, y) u(y) \intd y \intd x = \int_\Omega \underbrace{\int_\Omega k(x, y) \intd x}_{=1} \, u(y) \intd y = \int_\Omega u(y) \intd y
	\]
	in this case. \qed
	\end{proof}
Please note that the weighted average property is stated using the integration w.r.t. the second variable while the conservation results from the unit measure in the first variable.
Obviously a convolution with a convolution kernel satisfying
\[
	\int_\R k(y) \intd y = 1
\]
satisfies both as $k(x, y) = k(x -y)$ holds in this case, but not every operator satisfying these properties is a convolution. Especially when one is interested in bounded domains convolutions are not an option, but there still exist suitable smoothing operators.

\begin{theorem}[Universally dissipative filters]
	A conservative filter $K$ is dissipative for all convex entropies $U$,
	\[
	E_{K u} = \int_\Omega U([K u](x)) \intd x \leq \int_\Omega U(u(x)) \intd x = E_u,
	\]
	 if it can be written as a conservative filter with a positive kernel, i.e.
	\[
		[K u] (x) = \int_\Omega  k(x, y) u(y) \intd y 
	\]
	with $\forall x, y: k (x, y) \geq 0$.
\end{theorem}
\begin{proof}
	Using Jensens inequality \cite{Rudin1966RCA}, the positivity and conservation of the filter allows us to show
	\[
	\begin{aligned}
	\int_\Omega U([K u](x)) \intd x =& \int_\Omega U\left (\int_\Omega  k(x, y) u(y) \intd y \right) \intd x 
	\leq \int_\Omega \int_\Omega  k(x, y)U(u(y)) \intd y \intd x  \\
	=& 
	\int_\Omega \underbrace{\int_\Omega k(x, y) \intd x}_{ = 1} U(u(y))  \intd y = \int_\Omega U(u(y)) \intd y
	\end{aligned}.
	\]
	\qed
	\end{proof}
These theorem shows that the first and second bullet above can be satisfied by an integral operator with a suitable kernel. An example of a dissipation that can be identified with a positive conservative filter is the filtering by the time evolution of
\[
	\derive u t = \epsilon \nabla^2_x u 
\]
on the entire domain as the assorted filter has the heat kernel as kernel function \cite{Evans2010PDE},
\[
	k^t(x, y) = h(x-y, t), \quad	h(x, t) = \frac{\e^{-\frac{\abs{x}^2}{4t}}}{(4 \pi t)^{n/2}}, \quad  t > 0.
\]
Further, this filtering obviously stems from viscosity and has therefore a direct physical interpretation. It is known that while a positive integral operator always dissipates entropy a high order finite-difference implementation will not dissipate all entropies \cite{ranocha2019mimetic} and similar theorems hold for higher even derivatives even in the analytic case. We will therefore outline how to construct a filter that is dissipative in the semidiscrete and fully discrete setting and can therefore be used as a descent direction.
We begin by stating some discrete equivalents of the theorems above and will analyze if usual dissipations/filters satisfy this property. We will assume that $\omega_k \geq 0$ is a positive quadrature rule on the cell $T$ for the rest of the chapter and all notions of conservation for our filters will be centered around being conservative with respect to this quadrature rule. For a general DG method with dense mass matrix a quadrature can be calculated via $\sum_l M_{lk} = \omega_k$, i.e. by entering the constant one into the discretised inner product, but positivity is not guaranteed in general. 
A general view of our plan could be to not discretise the second derivative, but its action as the generator of a Hilbert-Schmidt operator. We will therefore, when given a discrete filter, consider also its (discrete) generator. 

\begin{definition}[Conservative and positive filter generator]
	Let $G \in \R^{(p+1) \times (p+1)}$ be a square matrix. We call this matrix a filter generator if
	\[
		\forall k \in \sset{ 1, \dots, p+1}: \quad \sum_{l=1}^{p+1} G_{kl} = 0 
	\]
	holds. It will be conservative if 
	\[
		\forall l \in \sset{1, \dots, p+1}: \quad \sum_{k=1}^{p+1} \omega_k G_{kl} = 0
	\]
	is satisfied. Further, we call it positive, if
	\[
		\forall l \in  \sset{1, \dots, p+1},\quad \forall k \in \sset{1, \dots, l-1, l+1, \dots, p+1} : \quad G_{kl} \geq 0
	\]
	holds.
	\end{definition}
\begin{definition}[Discrete conservative and positive filter]
	We call a matrix $\Upsilon \in \R^{(p+1) \times (p+1)}$ a filter, if
	\[
			\forall k \in \sset{1, \dots, p+1}: \quad \sum_{l=1}^{p+1} \Upsilon_{kl} = 1
	\]
	holds. It is termed conservative, if
	\[
		\forall l \in \sset{1, \dots, p+1}: \quad \sum_{k=1}^{p+1} \omega_k \Upsilon_{kl} = \omega_l
	\]
	is satisfied. Further, we call it positive, if
	\[
		\forall k, l \in \sset{1, \dots, p+1}:\quad \Upsilon_{kl} \geq 0.
	\]
	\end{definition}
Obviously, the definition of the conservative positive discrete filter mirrors the definition of such a filter in the continuous case using the quadrature rule. The definition of the averaging property on the other hand is not based on the quadrature rule, as this rule is not used when applying the filter pointwise
\[
\upsilon_k = \sum_{l = 1}^{p+1} \Upsilon_{kl} u_l
.\]

Forward Euler steps connect the generators defined above with the filters, as we will see in the lemma below.
\begin{lemma}[Connecting generators and filters] \label{lem:TiFilter}
	It holds 
	\[
		G \text{ conservative as generator } \implies \Upsilon = \Id + \Delta t G \text{ conservative as filter.}
	\]
	Let further $\Delta t \max_l \abs{G_{ll}} \leq 1$. Then it follows
	\[
		G \text{ positive as generator} \implies \Upsilon \text{ positive as filter}.
	\]

	\end{lemma}
	\begin{proof}
		We begin by showing the conservativity and filter property. It holds
		\[
			\sum_{l=1}^{p+1} \Id_{kl} = 1 \implies \sum_{l=1}^{p+1} \Upsilon_{kl} = \sum_{l=1}^{p+1} (\Id + \Delta t G)_{kl} = 1.
		\]
		As the identity is conservative follows
		\[
			\sum_{k=1}^{p+1} \omega_k \Id_{kl} = \omega_l \implies \sum_{k=1}^{p+1} \omega_k \Upsilon_{kl} = \sum_{k=1}^{p+1} \omega_k (\Id + \Delta t G)_{kl} = \omega_l.
		\]
		The positivity follows as for non-diagonal elements,
		\[
		k \neq l \implies 	(I + \Delta tG)_{kl} = \Delta t G_{kl} \geq 0
		\]
		is satisfied for any positive timestep size while the given restriction is needed to enforce
		\[
			(I + \Delta t G)_{ll} \geq 1 - \Delta t \abs{\Upsilon_{ll}} \geq 0.
		\]
		\qed
		\end{proof}
	It is clear that a discrete filter that is positive and conservative is also dissipative by reiterating through the arguments given above for the continuous case. Sadly, it is also true that while in the continuous case the filter which is generated by the second derivative, i.e. the heat kernel, is positive, the second derivative discretised in our DG method is not a positive generator and also does not generate a positive filter directly. We will therefore show how to design a generator generating an approximation of the heat kernel for forward Euler steps, thereby even allowing to prove the dissipativity of the entropy dissipation operator for finite time steps. The basis will be the heat equation with varying heat conductivity $\alpha(x)$ \cite{Johnson2009FEM}
	\[
		\derive u t = \sum_{k=1}^n \derive{~}{x_k} \alpha(x) \derive{u}{x_k}, \quad \alpha \derive u n\Bigg|_{\partial T} = 0
	\] 
	on the (reference) element in conjunction with Neumann boundary conditions. The Neumann boundary conditions enforce the conservation of the resulting solution operator as any change of the cell mean values must happen through the numerical flux of the basic DG method. Discretising this problem \cite{Johnson2009FEM} with the nodal basis of the basic DG method that is a continuous Galerkin method in this case because a single element is considered, yields
	\begin{equation} \label{eq:HEsol}
		\derd {u^T} t = -M^{-1} Q u^T, \quad Q_{kl} = \int_T \rskp{\derive {\phi_k} x} {\alpha(x) \derive {\phi_l} x} \intd x.
	\end{equation}
	As noted before, in general there exists no $\Delta t > 0$ where $\Id + \Delta t (-M^{-1} Q)$ is a positive operator because the negative elements in $(-M^{-1} Q)$ prohibit it from being a positive generator. Yet the following theorem shows that the exact ODE solution to this problem for a $t > 0$ big enough is in fact eligible as a filter.
	\begin{theorem}
		If the quadrature $\omega$ is exact on $V^T$, the solution of \eqref{eq:HEsol} for a positive initial condition $u_0 \in \R^{p+1}$ satisfies for all $t > 0$
			\begin{itemize}
				\item $\sum_{k = 1}^{p+1} \omega_k u_k(t) = \sum_{k=1}^{p+1} \omega_k u_k(0)$ (Conservation)
				\item $u_k(t) = C_{kl} (t) u_l(0)$ with $\forall k\in\sset{1, \dots, p+1}: \quad \sum_{l = 1}^{p+1} C_{kl} = 1$ (averaging property)
			\end{itemize}
		Further, for a $t > 0$ big enough it follows $\forall k \in \sset{1, \dots, p+1}: \quad u_k(t) \geq 0$.
		\end{theorem}
	\begin{proof}
		Entering $v = 1$ into the weak form results in 
		\[
			\int_T \derive u t \intd x = -\int_T \derive 1 x \derive u x \intd x = 0.	
		\]
		As the quadrature is exact for the basis functions the same follows for the discretisation, and this shows the conservation. The matrix $C$ used to describe the solution has the explicit form \cite[sec. 34]{LaxFun}
		\[
			u(t) = \underbrace{\e^{-t M^{-1}Q}}_{C(t)}u(0).
		\]
		Multiplying this matrix with the vector $v \in V$ representing the function $1$ from the right reveals
		\[
			M^{-1}Q v = 0 \implies \e^{-t M^{-1}Q} v = \Id v = v.
		\]
		This already shows the second result as the nodal representation of $C(t)$ must have unit row sum.
		The matrix $-Q$ is negative semi-definite, the $v$ vector is in its null space. If another linearly independent $u \in V$ would be in its null space it would follow 
		\[
			\rskp{u}{Q u} = \skp{\derive u x}{\alpha \derive u x}_T = \int_T \alpha(x) \abs{\derive{u} x}^2 \intd x = 0
		\]
		and this is a contradiction to $\derive u x \neq 0$, as $u$ was assumed non-constant. Therefore, there exists an orthonormal eigenvalue decomposition of the discretisation whose eigenvalues, apart from the constant eigenfunction $\psi_1 = v$ with eigenvalue $\lambda_1 = 0$, are bounded away from zero,
		\[
		\forall k \in \sset{1, \dots, p+1}: \quad -M^{-1} Q \psi_k = \lambda_k \psi_k
		.\] 
		We assume that the eigenvectors are sorted by increasing absolute value of the corresponding eigenvalues,
		\[
			0 = \lambda_1< \abs{\lambda_2} \leq \abs{\lambda_3} \leq \dots \leq \abs{\lambda_{p+1}}.
		\] The solution
		\[
			u(t) = \sum_{k=1}^{p+1} \e^{- \lambda_k t} \psi_k \skp{\psi_k}{u(0)}
		\]
		therefore converges to the average of $u(0)$, as 
		\[
			\norm {u(t) - \psi_0\skp{\psi_0}{u(0)}}^2 = \norm{ \sum_{k = 2}^{p+1} \e^{- \lambda_k t} \psi_k \skp{\psi_k}{u(0)}}^2 \leq \e^{-2\lambda_1 t} \norm{u_0}^2
		\]
		holds. Because a positive initial condition has a positive average the solution will converge to this positive average. \qed
		\end{proof}
		Using the theorem above we can construct filters $\Upsilon$ simply by calculating the matrix $C(t) = G$ used in the proof above. This matrix which maps an initial state onto the solution at time $t$ is always a conservative filter, and when $t$ is large enough also positive. In the implementation the suitable $t$ was found using a bisection algorithm. Using $\Upsilon = (C(t) - \Id)/t$ the corresponding generator can be found. We note in passing that numerous other possibilites exist to define a positive conservative filter as defined above, but that the method given above defines a filter than can be associated with viscosity.
		
		\begin{lemma}\label{lem:alwdisp}
			Assume the null space of $G$ consists only of constants. Then for a non-constant $u$ and a strictly convex entropy $U$ it holds
			\[
				\skp{\derd U u}{G u}_{T, \omega} < 0.
			\]
			If $U$ is just convex only »$\leq$«  applies in the equation above.
			\end{lemma}
		\begin{proof}
		 The discrete dissipativity 
		 \[
		 \begin{aligned}
			 E_T(u + \Delta t \lambda G u) &= \sum_{k = 1}^{p+1} \omega_k U\left( u_k + \Delta t \lambda \sum_{l=1}^{p+1} G_{kl} u_l \right) \\
			 &< \sum_{k = 1}^{p+1} \omega_k \sum_{l = 1}^{p+1} G_{kl}  U(u_l) = \sum_{l=1}^{p+1} \omega_l U(u_l) = E_{u, T}
			 \end{aligned}
		\]
		follows from the positive conservative filter property of $G$ for $\lambda \Delta t > 0$ small enough as in lemma \ref{lem:TiFilter} in conjunction with the strict convexity and Jensens inequality in the strict sense. Let now $\Delta t$ be fixed and small enough for all  $\lambda \in [0, 1]$, and denote by $\epsilon = E^T(u + \Delta t G u) - E^T_u < 0$ the entropy dissipation for $\lambda = 1$. The convexity of $U$ implies
		\[
			E_T(u + \lambda \Delta t Gu) \leq E_T(u) + \lambda \left(E_T(u + \Delta t Gu)  - E_{T,u}\right) = E_T(u) + \lambda \epsilon 
			\]
			Entering this into the definition of the derivative of $E^T$ with respect to $\lambda$ shows
			\[
			 \derd{E_T(u + \lambda \Delta t Gu)}{\lambda} = \lim_{\lambda \to 0}   \frac{E_T(u+\lambda \Delta t Gu) - E_T(u)}{\lambda} \leq \epsilon,
		\]
			and therefore
			\[
				 \skp{\derd U u}{Gu}_{T, \omega} = \sum_{k = 1}^{p+1} \omega_k \rskp{\derd U u (x_k)} {(G u)_k} =\derd{E_T(u + \lambda \Delta t Gu)}{\lambda} \frac{1}{\Delta t} = \frac{\epsilon}{\Delta t}.
			\]
			If $U$ is not strictly convex the case $\epsilon = 0$ is possible, reducing the result to »$\leq$«. \qed
		\end{proof}
		
		The last step consists of selecting a suitable viscosity distribution $\alpha$, i.e. one that is zero at the endpoints. The standard mollifier
		\[
			\alpha(x) = \begin{cases} \e^{1 - \frac{1}{1-x^2}} & \abs{x} \leq 1 \\ 0 & \abs{x} > 1\end{cases}
		\]
		is smooth and zero at the ends of the reference element. Further, even its derivatives vanish there. It was therefore selected.

	\subsection{Stable computation of the correction size required and timestep restrictions} \label{sec:CorIP}
After we have calculated the entropy dissipation needed and a suitable direction $\upsilon = G u^T$ one would guess we only have to calculate $\lambda$ as in \eqref{eq:lambdadef} via 
\[
\lambda  \geq \frac{\sigma^T - \left(\skp{\derd U u}{\derd u t}_{T, \omega} - \left(F^*_l - F^*_r \right)\right)}{ \skp{\derd U u} {\upsilon}_{T, \omega}}.
\] 
It turns out that this process is significantly more intricate than one would expect as this computation has to be stable with respect to roundoff errors. Further, our estimates on the entropy dissipation can only estimate the entropy dissipation that can take place at the interface between two adjacent cells, but are not able to give an estimate of how this dissipation is split between the two cells. Our method of calculating suitable values of $\lambda^T$ therefore consists of two steps. First,  
\begin{equation} \label{eq:EDcorrection}
\lambda_{\mathrm{ED}}^T = \max\left (0, -\frac{ \skp{\derd U u}{\derd u t}_{T, \omega} - \left(F^*_l - F^*_r \right)}{ \skp{\derd U u} {\upsilon}_{T, \omega}} \right)
\end{equation}
 is calculated to enforce  the per cell entropy dissipativity 
\[
	\skp{\derd U u}{ \derd u t + \lambda_{\mathrm{ED}}^T \upsilon}_{T, \omega} \leq F^*_l - F^*_r.
\]
In a second step a correction to enforce an entropy rate high enough
\begin{equation} \label{eq:ERcorrection}
	\lambda_{\mathrm{ER}}^{\theta} =  \max\left(0, \frac{ \sigma^{\theta} -  \sum_{T \cap \theta \neq \emptyset}\left(\skp{\derd U u}{\derd u t +\lambda_{\mathrm{ED}}^T \upsilon}_{T, \omega} - (F^*_{l, T} - F^*_{r, T})\right)}{ \sum_{\theta \cap T \neq \emptyset}\skp{\derd U u} {\upsilon}_{T, \omega} } \right)
\end{equation}
is determined for all $\theta \in \Theta$. Both corrections are then added together 
\[
	\lambda^T_{\Sigma} =  \lambda^{T}_{\mathrm{ED}}  + \sum_{\theta \cap T \neq \emptyset} \lambda^{\theta}_{\mathrm{ER}}
\]
for all cells $T \in \T$.
Round-off errors tend to influence the calculation out of two reasons. The division by $\skp{\derd U u}{ \upsilon}$ in equation \eqref{eq:EDcorrection} and \eqref{eq:ERcorrection} can approach a division by zero for a solution approaching a constant in the cell, as $\upsilon \to 0$ follows in this case.  Further, we saw in lemma \eqref{lem:LFspeed} that the entropy inequality predictor can vanish with a high order for smooth solutions, and an accurate DG scheme will also have a vanishing entropy error vanishing with a high order. The difference of these two values, i.e. the denominator of the fraction above, will in general not vanish that fast because round-off in the difference becomes important. Therefore $\lambda$ will, for highly resolved smooth solutions, be to big because round-off errors propagate into the calculation. Our solution to this problem is to calculate 
\[
	\lambda_{(\cdot)}^T  = \max\left(\frac{ab}{b^2 + c^2}, 0\right), \text{ instead of } \frac a b 
\]
every time a $\lambda$ is calculated by a division in the procedure above.
Here, $a$ shall be the nominator, $b$ shall be the denominator and $c$ shall be a suitable bound on the round-off error, a constant small with respect to $a, b$ but large with respect to the machine precision. In our implementation this is selected as $c = \sqrt{10^{-16}}$, i.e. the square root of the machine precision for a solution scaled to be of unit magnitude. The addition of $c$ can be seen as the one-dimensional version of Tikhonov regularization \cite{Kress1998Tikhonov}. Clipping the calculation of $\lambda$ at $0$ 
ensures that if $a$ or $b$ become negative from rounding errors $\lambda$ will not become negative, i.e. $\lambda \upsilon$ will not be antidissipative. In a last step,
\[
\lambda^T = \min\left(\lambda_\mathrm{max},\lambda^T_{\Sigma} \right),
\] 
the upper limit $\lambda_{\text{max}}$ is introduced for stability reasons as we want to enforce stability of 
\begin{equation}\label{eq:dispTI}
	\derd {u^T}{t} = \upsilon = \lambda^T G u^T.
\end{equation}
If a Runge-Kutta time integration method can be written as convex combination of forward Euler steps, i.e. is Strong Stability Preserving (SSP) \cite{SSPRK, SO1988, SO1989} and the time-steps satisfy $\Delta t \lambda \leq 1$ during every Euler step, the lemma \ref{lem:TiFilter} allows us to show that the solution is also entropy dissipative in the discrete case. If the time integration method used is just a conditionally stable Runge-Kutta method \cite{Dahlquist1963Special, WHN1978Order} we are interested in limiting the operator norm of $\Delta t \norm{\lambda G} \leq R$ in order to at least avoid a linear instability. The exact size depends on the time integration methods' stability region as we would like to fit the half-circle 
\[
	C = \{z \in \C | \norm{z} \leq R \wedge \im z \leq 0\}
\]
into the stability region of the method.

%% file: tests.tex
Our tests will be carried out for the Euler equations of gas dynamics in conservation form \cite{Harten83b}
\[
u =(\rho, \rho v, E) 
\quad f(\rho, \rho v, E) = \begin{bmatrix} \rho v \\ \rho v^2 + p\\ v(E + p) \end{bmatrix} 
\quad p = (\gamma - 1)\left (E - \frac 1 2 \rho v^2 \right)
\]
in conjunction with the physical entropy \cite{Tadmor2003, Harten83b}
\[
U(\rho, \rho v, E) = - \rho S  \quad F(\rho, \rho v, E) = - \rho v S \quad S = \ln(p \rho^{- \gamma}).
\]
The tests below will focus on the cases $p = 3$ and $p = 7$ as the latter are popular in applications because they amount to $4$ and $8$ nodes, suitable for SIMD processor instructions. Results for values in between are essentially interpolatory to the ones reported for $p = 3$ and $p = 7$ and the source code is available to carry out tests for all values $p > 0$.  
Time integration will be carried out using the SSPRK(4, 3) method for most solutions, while  the convergence analysis for $p = 7$ below will use the Hairer-Wanner DOPRI8 method, to achieve the needed convergence speed of the time integration. In all images below the ansatz functions of all cells are shown without any post-processing.

\begin{table}
	\centering
	\begin{tabular}{c | c | c}
		Property & Tested Solver & Reference \\
		\hline
		Type	&	DDG			&	first order FV \\
		Intercell Flux & local Lax-Friedrichs & Lax-Friedrichs \\
		CFL number &	$0.1/(p^2 +p)$ & $0.5$ \\
		Time Integration & SSPRK(4,3), DoPri8& Forward Euler \\
		Dissipation & $G$ from sec. \ref{sec:dispd}& Built-in \\
		$\lambda_{\mathrm{max}}$ & $\frac 1 {\Delta t}$ & Does not apply \\
		Number of Cells & 13, 25, 50, 100, 200 & $3\cdot 10^4$ \\

	\end{tabular}
	\caption{Used schemes in the numerical tests}
\end{table}
\subsection{Shock tube tests}
First, a series of shock tube tests was done to highlight the effectivity of the entropy correction in shock calculations as this is the primary aim of this publication. The first initial condition \cite[Problem I, Section 4.3.3 and Problem 6a]{Toro2009Riemann, SO1989}) is
\begin{align*}
	\rho_0(x, 0) = \begin{cases}1,  \\ 0.125,  \end{cases} 
	\quad 
	v_0(x, 0) = \begin{cases} 0,  \\ 0,  \end{cases}
	p_0(x, 0) = \begin{cases} 1.0, & x<5,  \\ 0.1, & x \geq 5.\end{cases}
\end{align*}
Our second shock tube is the time-evolution of the following Riemann problem \cite[Problem 6b]{SO1989})
\begin{align*}
	\rho_0(x, 0) = \begin{cases}0.445,  \\ 0.5,  \end{cases} 
	\quad 
	v_0(x, 0) = \begin{cases} 0.698,  \\ 0,  \end{cases}
	p_0(x, 0) = \begin{cases} 3.528, & x < 5.0, \\ 0.571, & x \geq 5.0. \end{cases}
\end{align*}
	\begin{figure}
		\begin{subfigure}{0.49 \textwidth}
			\includegraphics[width=\textwidth]{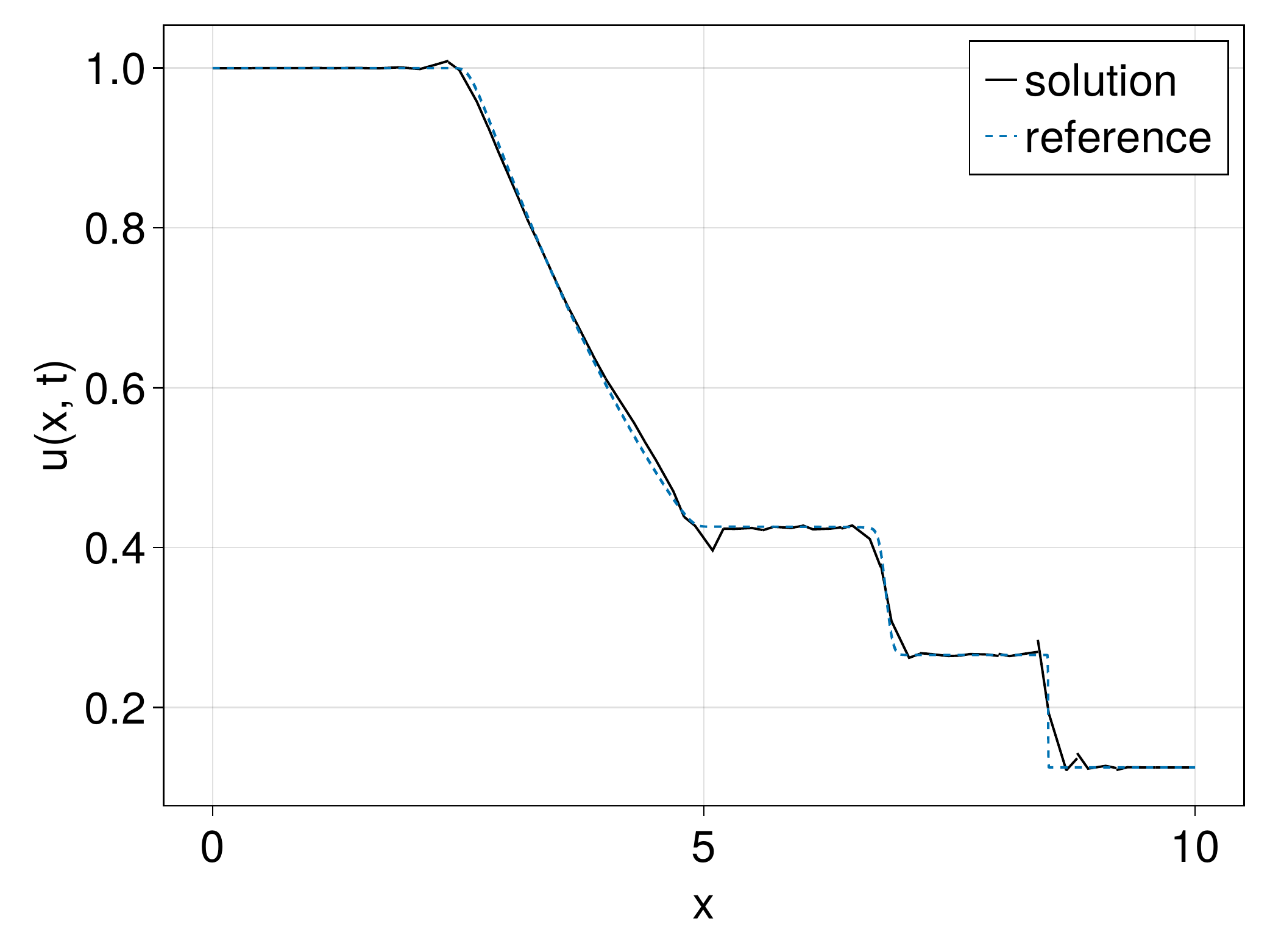}
			\caption{Density, $25$ cells}
		\end{subfigure}
		\begin{subfigure}{0.49 \textwidth}
			\includegraphics[width=\textwidth]{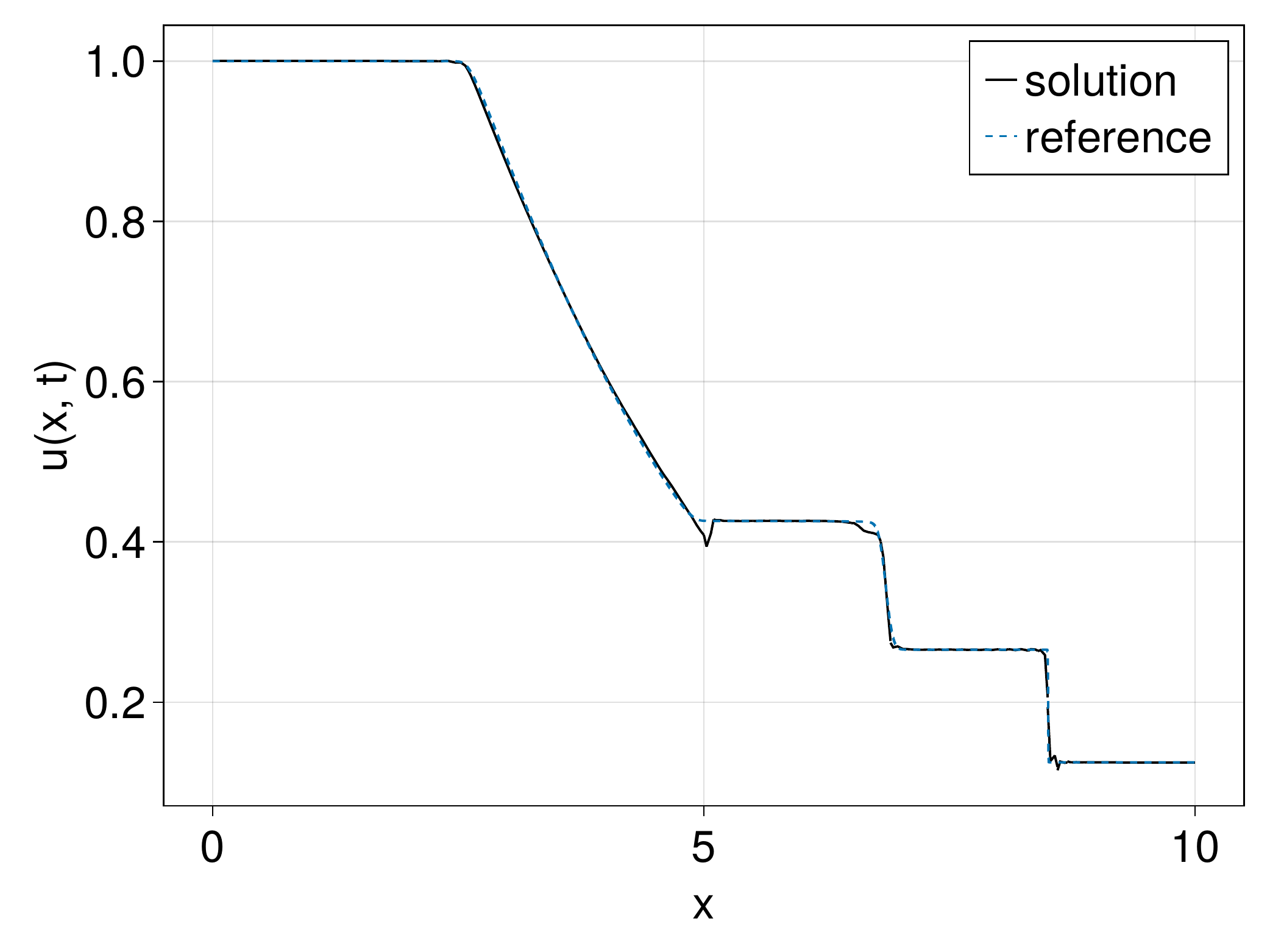}
			 \caption{Density, $100$ cells}
		\end{subfigure}
		\begin{subfigure}{0.49 \textwidth}
			\includegraphics[width=\textwidth]{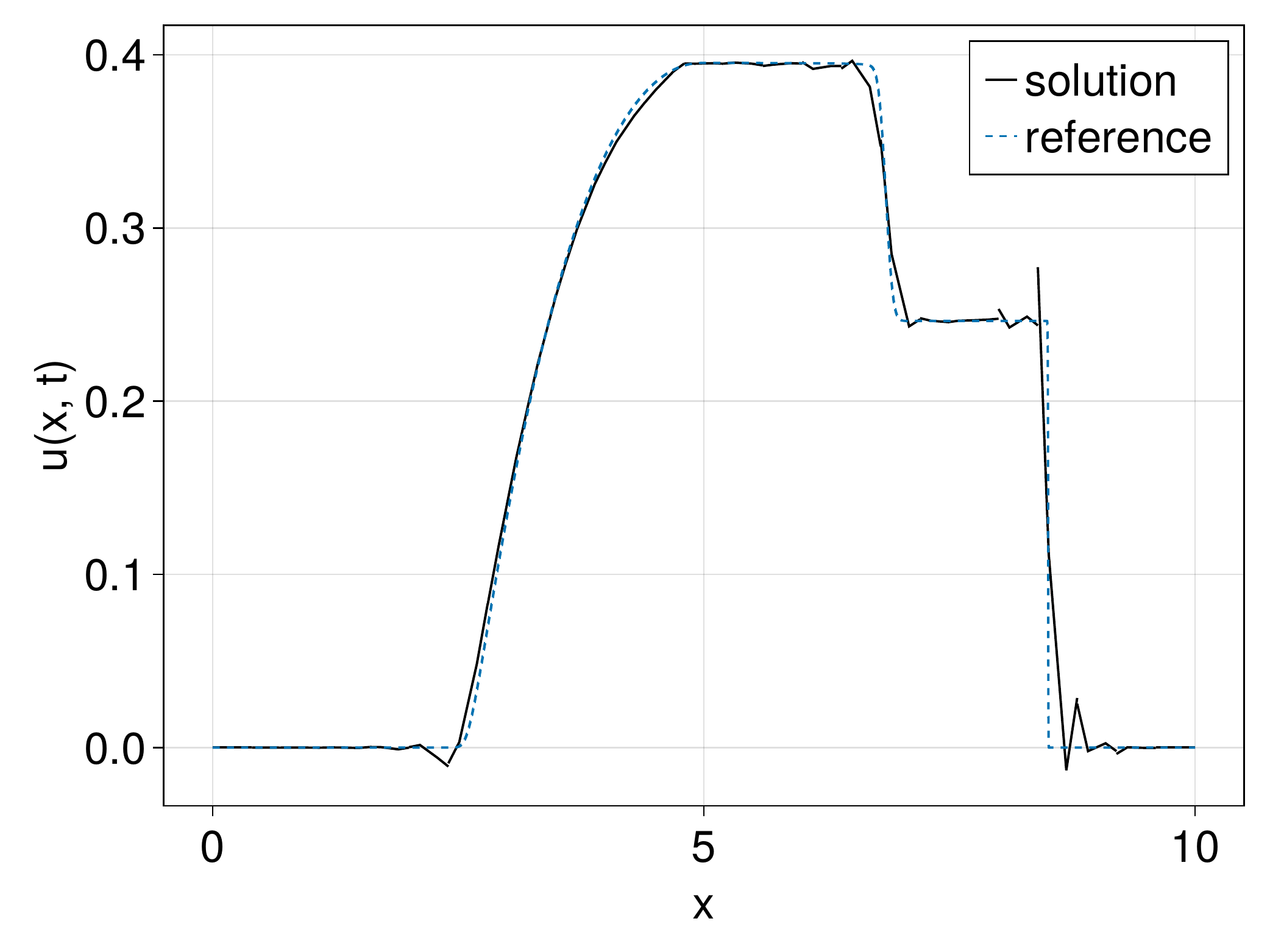}
			 \caption{Moment density, $25$ cells}
		\end{subfigure}
		\begin{subfigure}{0.49 \textwidth}
			\includegraphics[width=\textwidth]{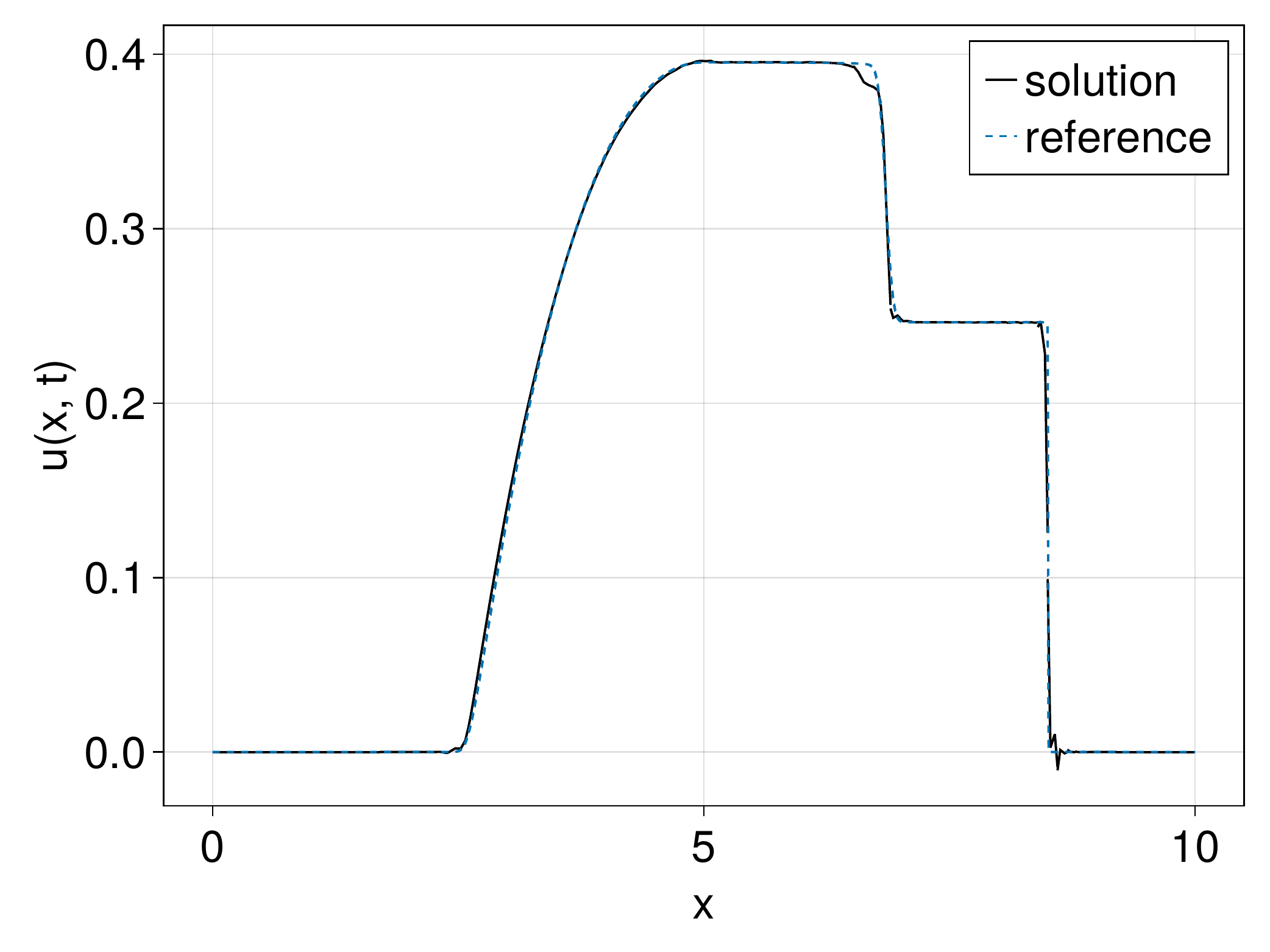}
				\caption{Moment density, $100$ cells}
		\end{subfigure}
		\begin{subfigure}{0.49 \textwidth}
			\includegraphics[width=\textwidth]{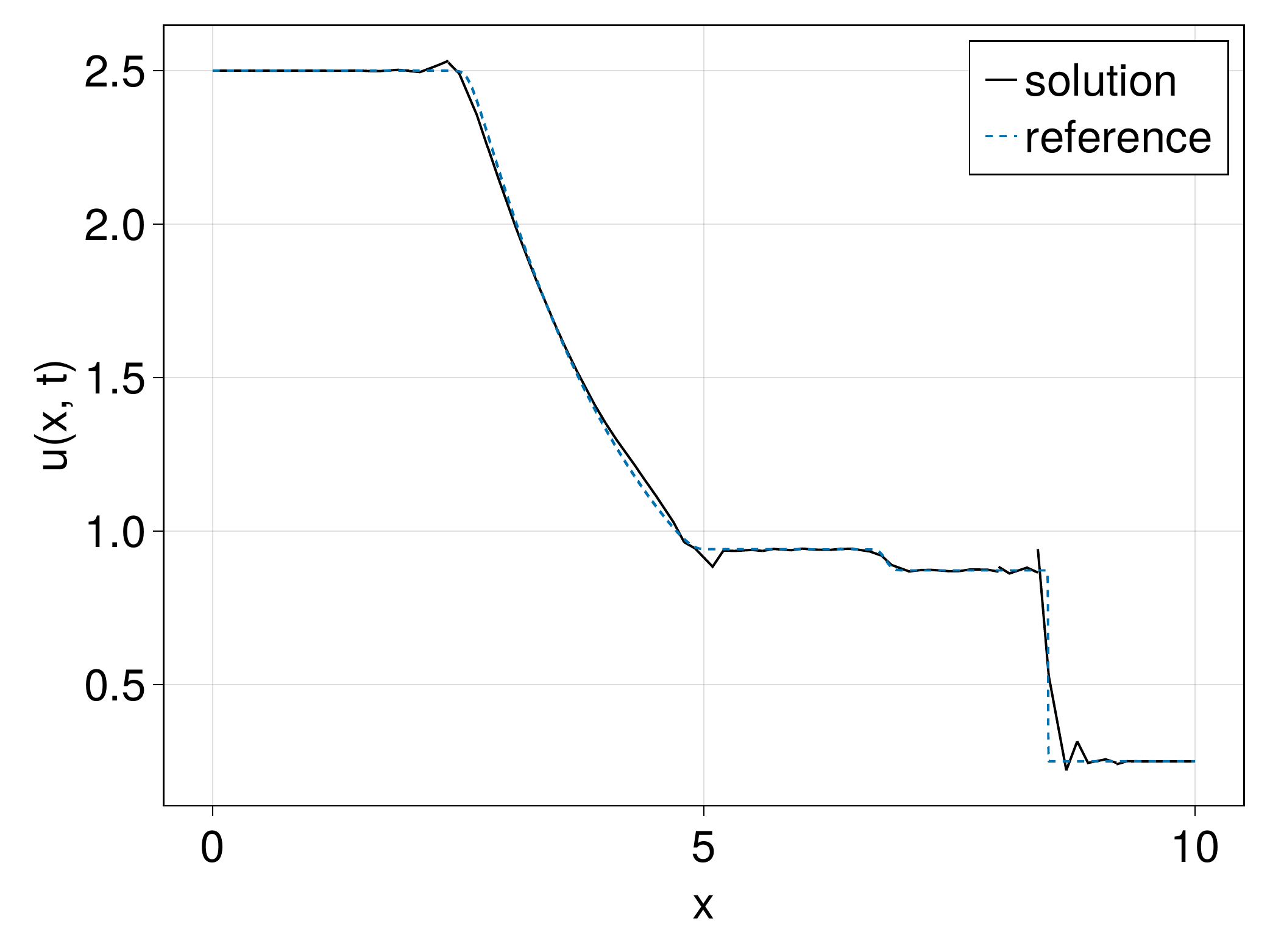}
				\caption{Energy density, $25$ cells}
		\end{subfigure}
		\begin{subfigure}{0.49 \textwidth}
			\includegraphics[width=\textwidth]{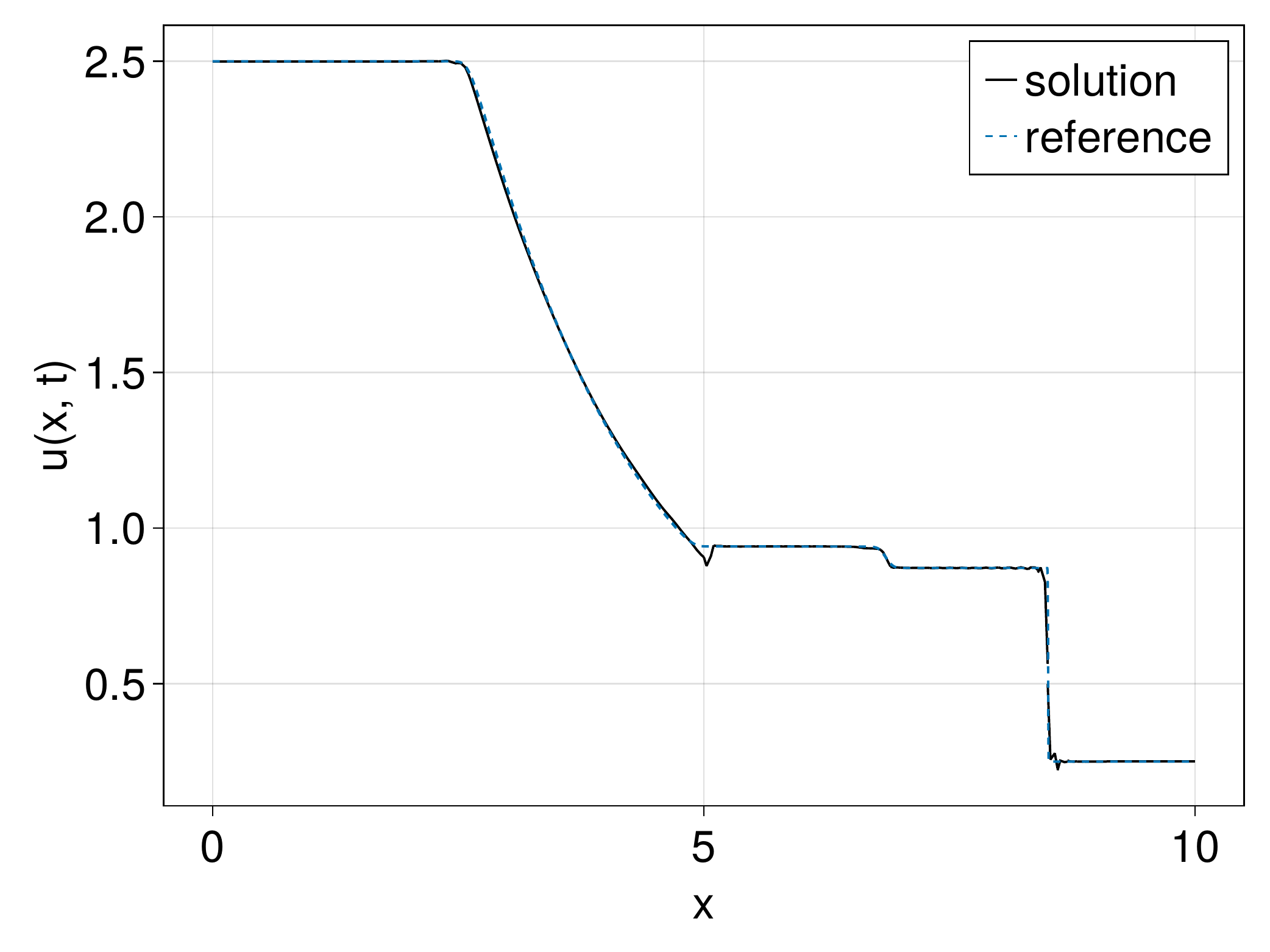}
				\caption{Energy density, $100$ cells}
		\end{subfigure}
		\caption{Shock tube 1 at $t = 1.8$ with $25$ cells corresponding to $100$ degrees of freedom and $100$ cells corresponding to $400$ degrees of freedom (p=3).}
		\label{fig:sod3}
		\end{figure}
		\begin{figure}
		\begin{subfigure}{0.49 \textwidth}
			\includegraphics[width=\textwidth]{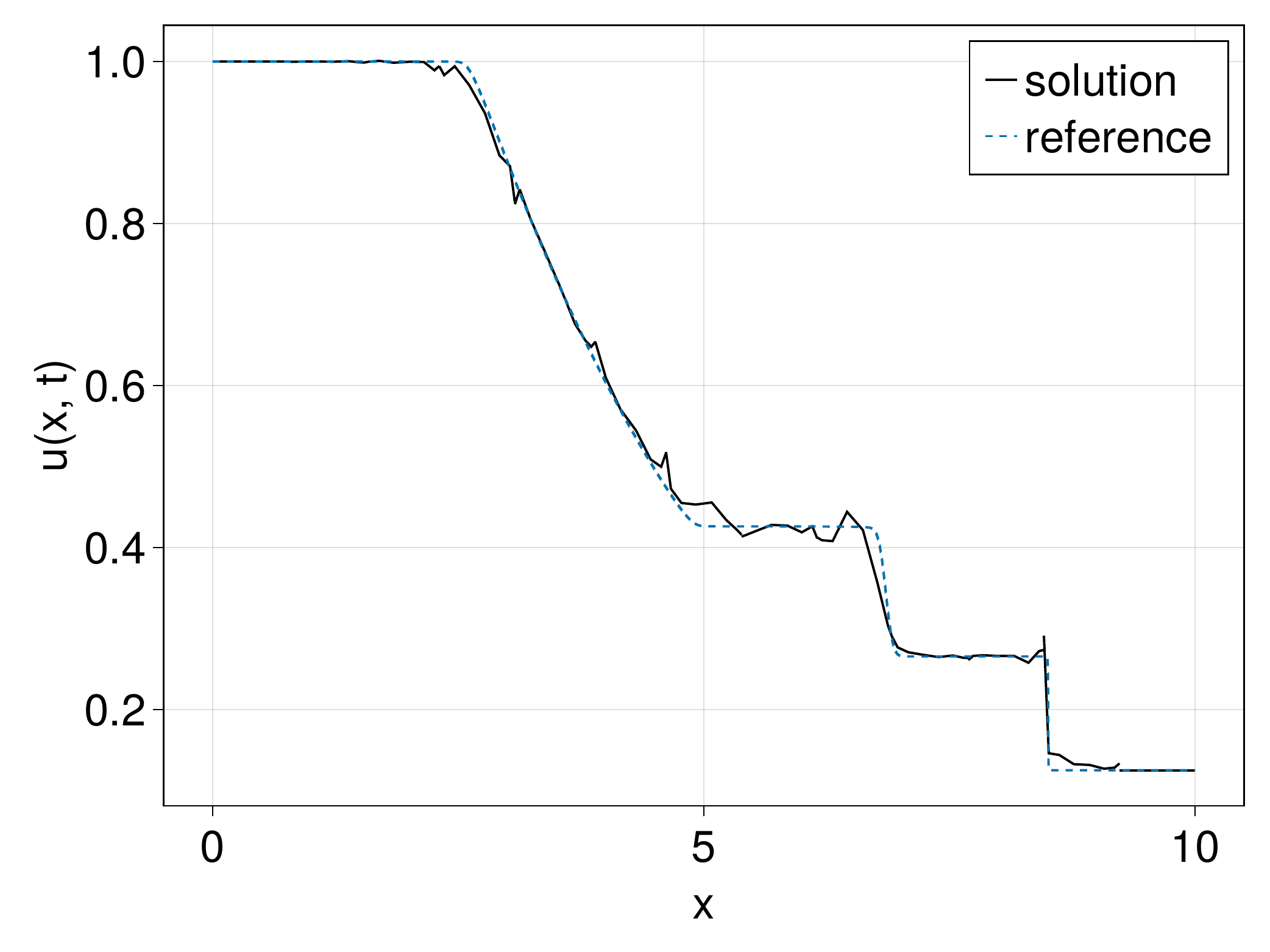}
			\caption{Density, $13$ cells}
		\end{subfigure}
		\begin{subfigure}{0.49 \textwidth}
			\includegraphics[width=\textwidth]{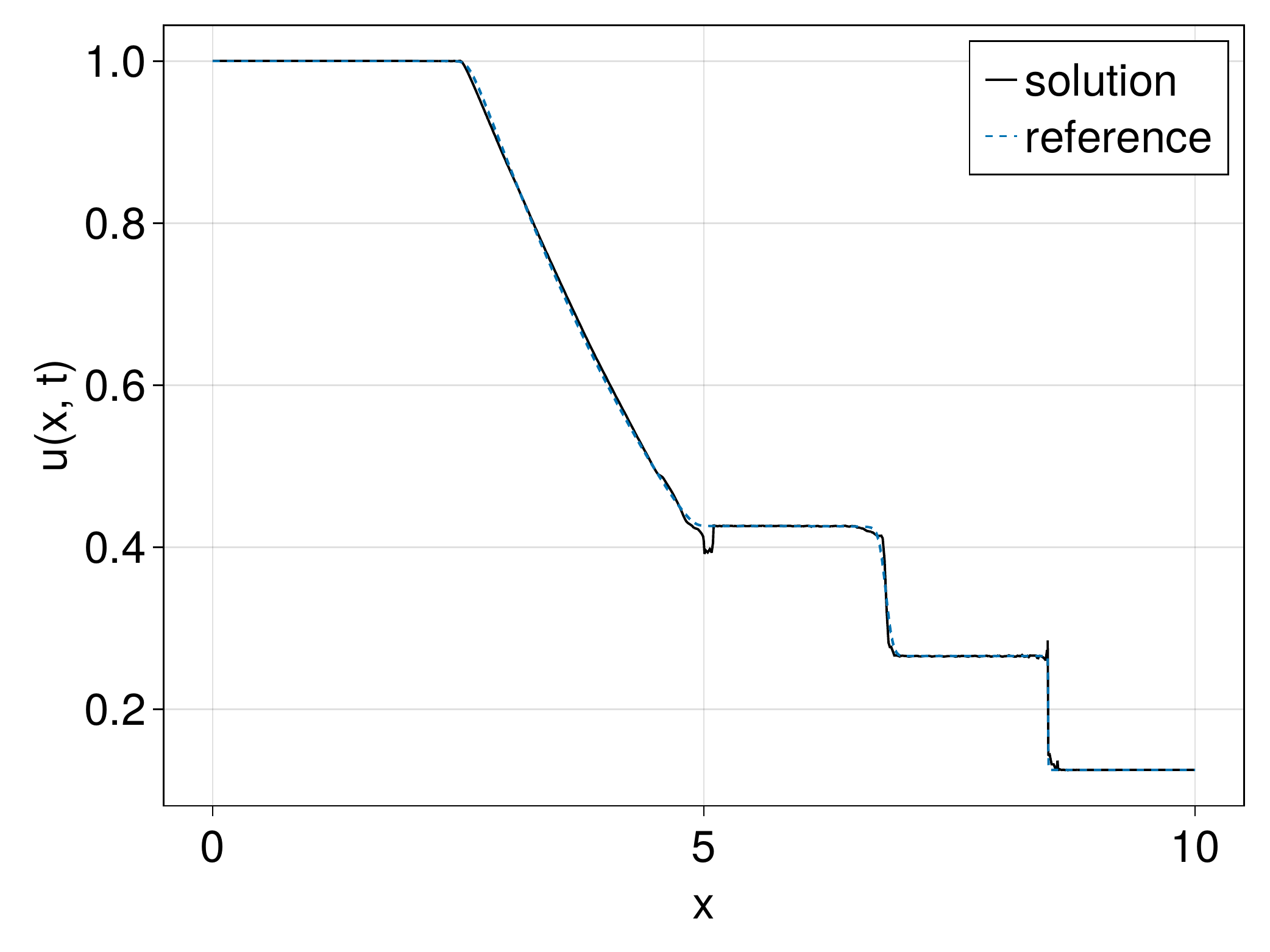}
			\caption{Density, $100$ cells}
		\end{subfigure}
		\begin{subfigure}{0.49 \textwidth}
			\includegraphics[width=\textwidth]{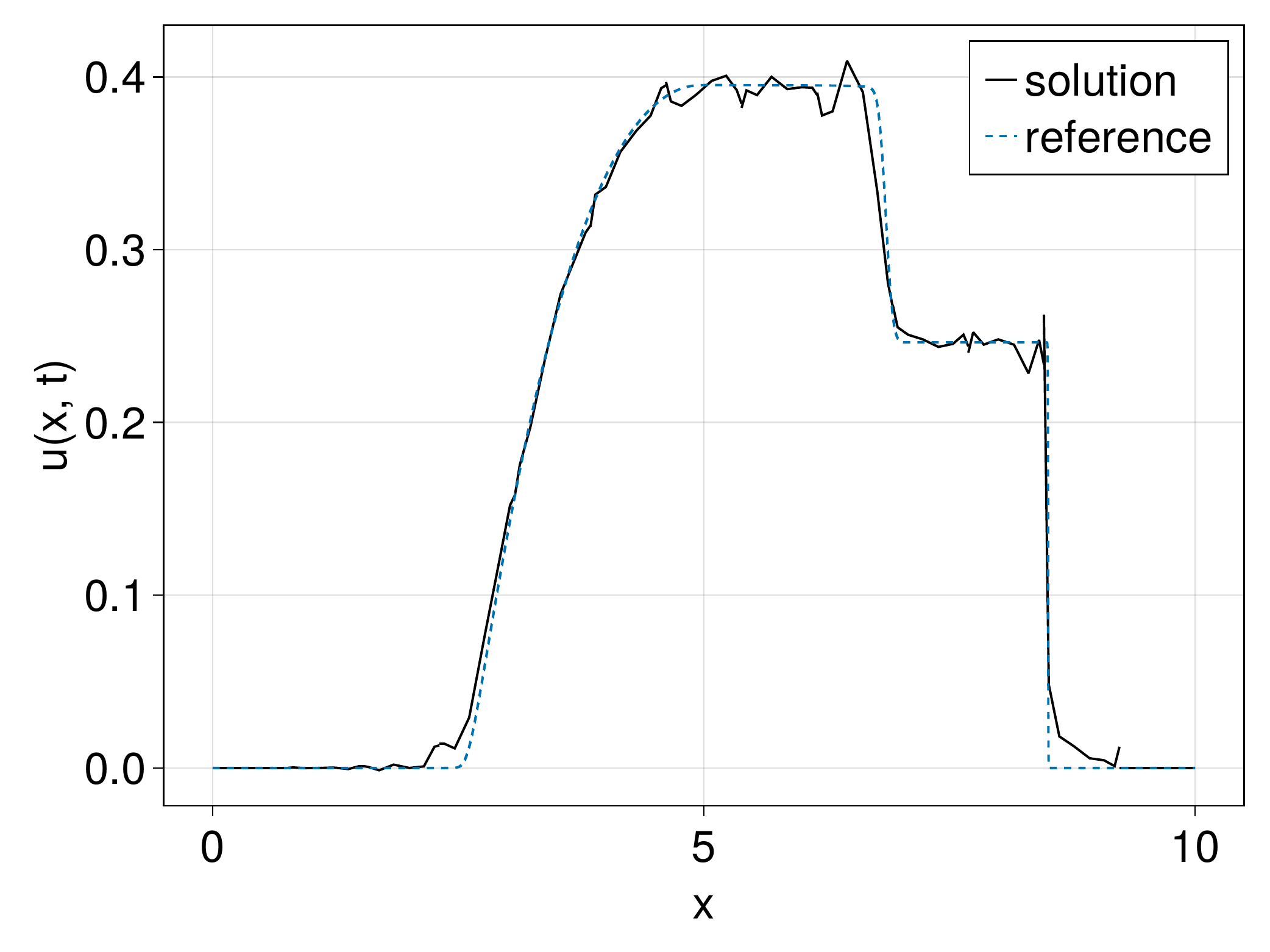}
			\caption{Moment density, $13$ cells}
		\end{subfigure}
		\begin{subfigure}{0.49 \textwidth}
			\includegraphics[width=\textwidth]{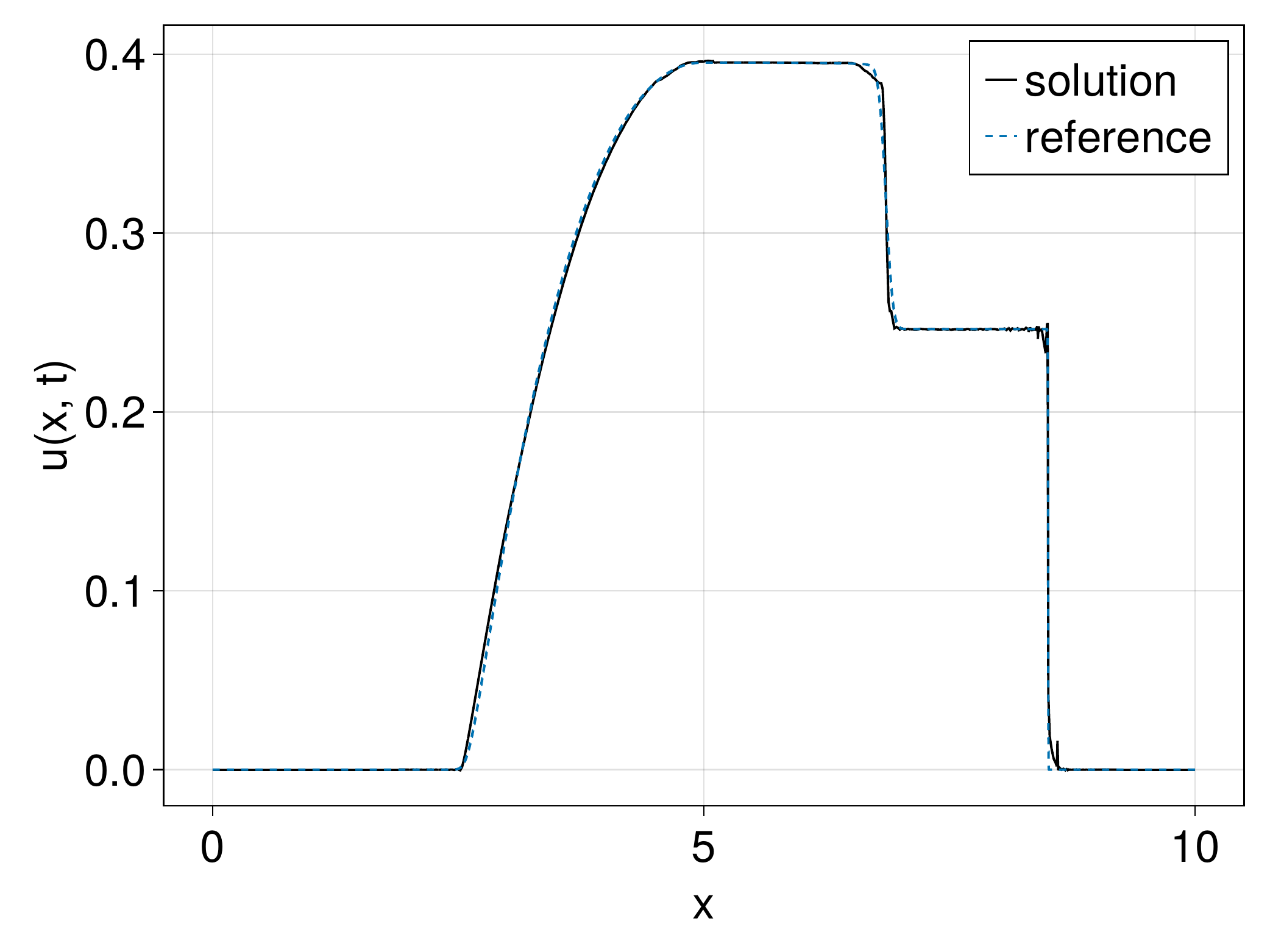}
			\caption{Moment density, $100$ cells}
		\end{subfigure}
		\begin{subfigure}{0.49 \textwidth}
			\includegraphics[width=\textwidth]{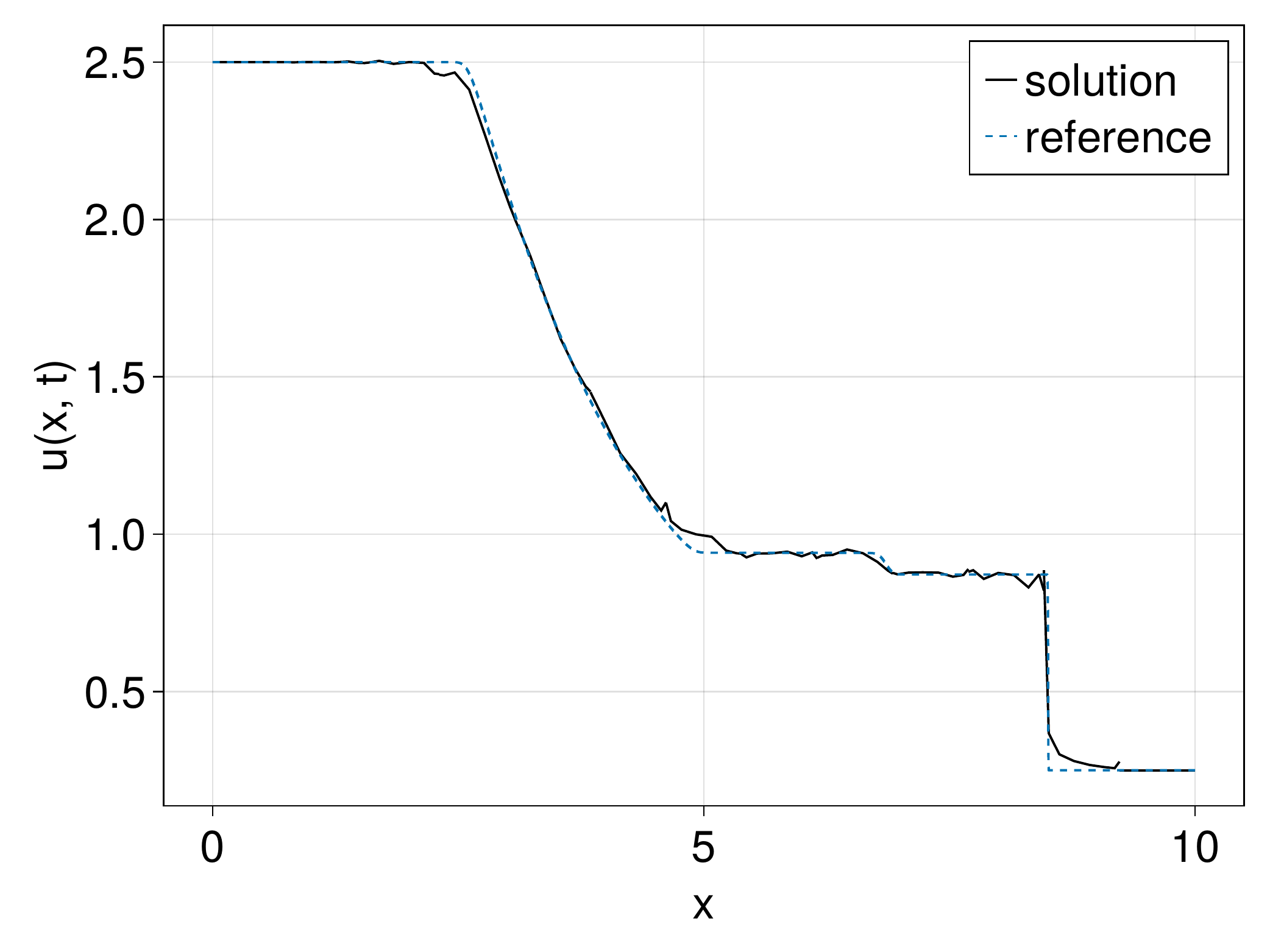}
			\caption{Energy density, $13$ cells}
		\end{subfigure}
		\begin{subfigure}{0.49 \textwidth}
			\includegraphics[width=\textwidth]{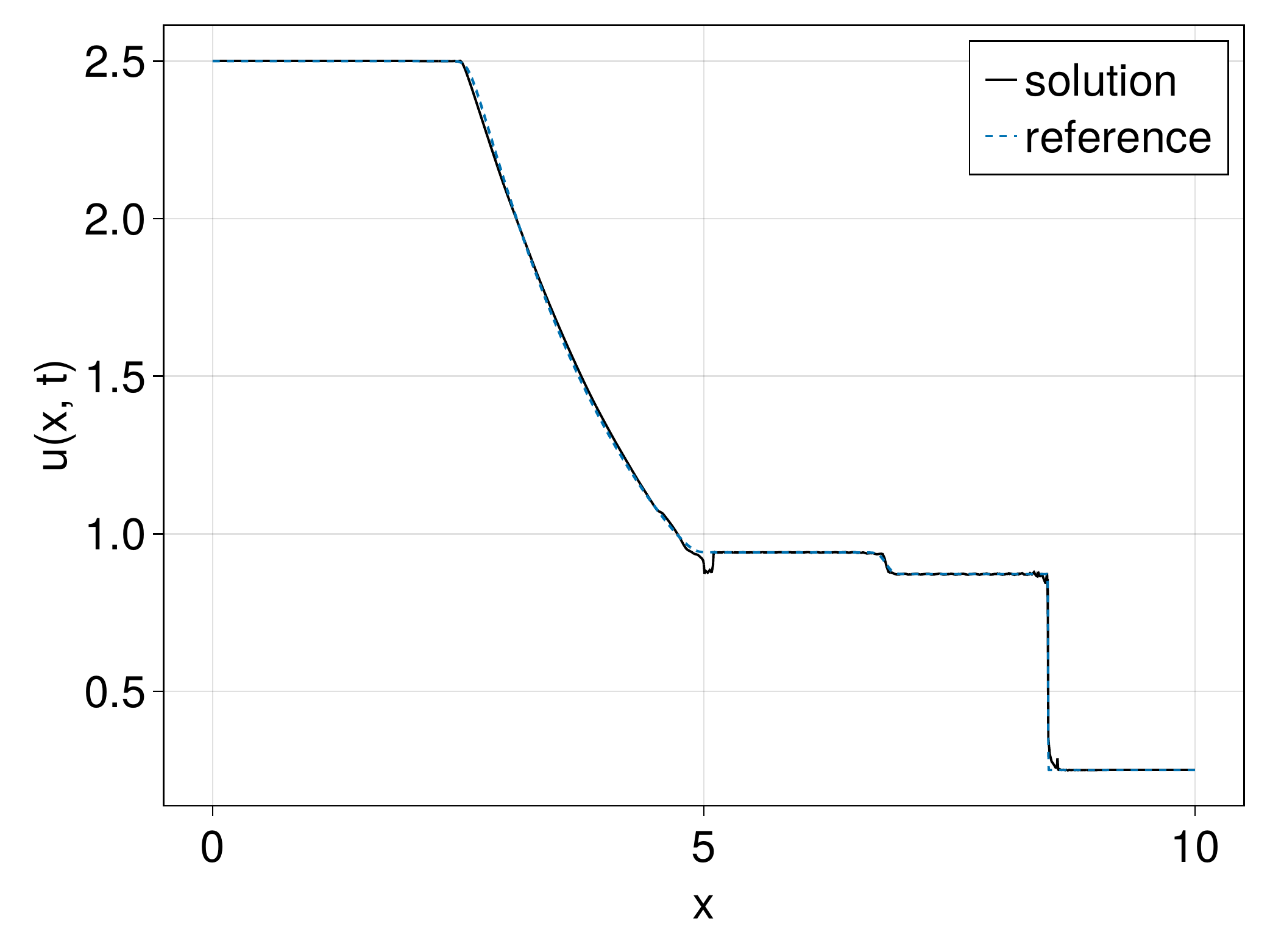}
			\caption{Energy density, $100$ cells}
		\end{subfigure}
		\caption{Shock tube 1 at $t = 1.8$ with $13$ cells corresponding to $104$ degrees of freedom and $100$ cells corresponding to $800$ degrees of freedom (p=7).}
		\label{fig:sod7}
	\end{figure}

	\begin{figure}
		\begin{subfigure}{0.49 \textwidth}
			\includegraphics[width=\textwidth]{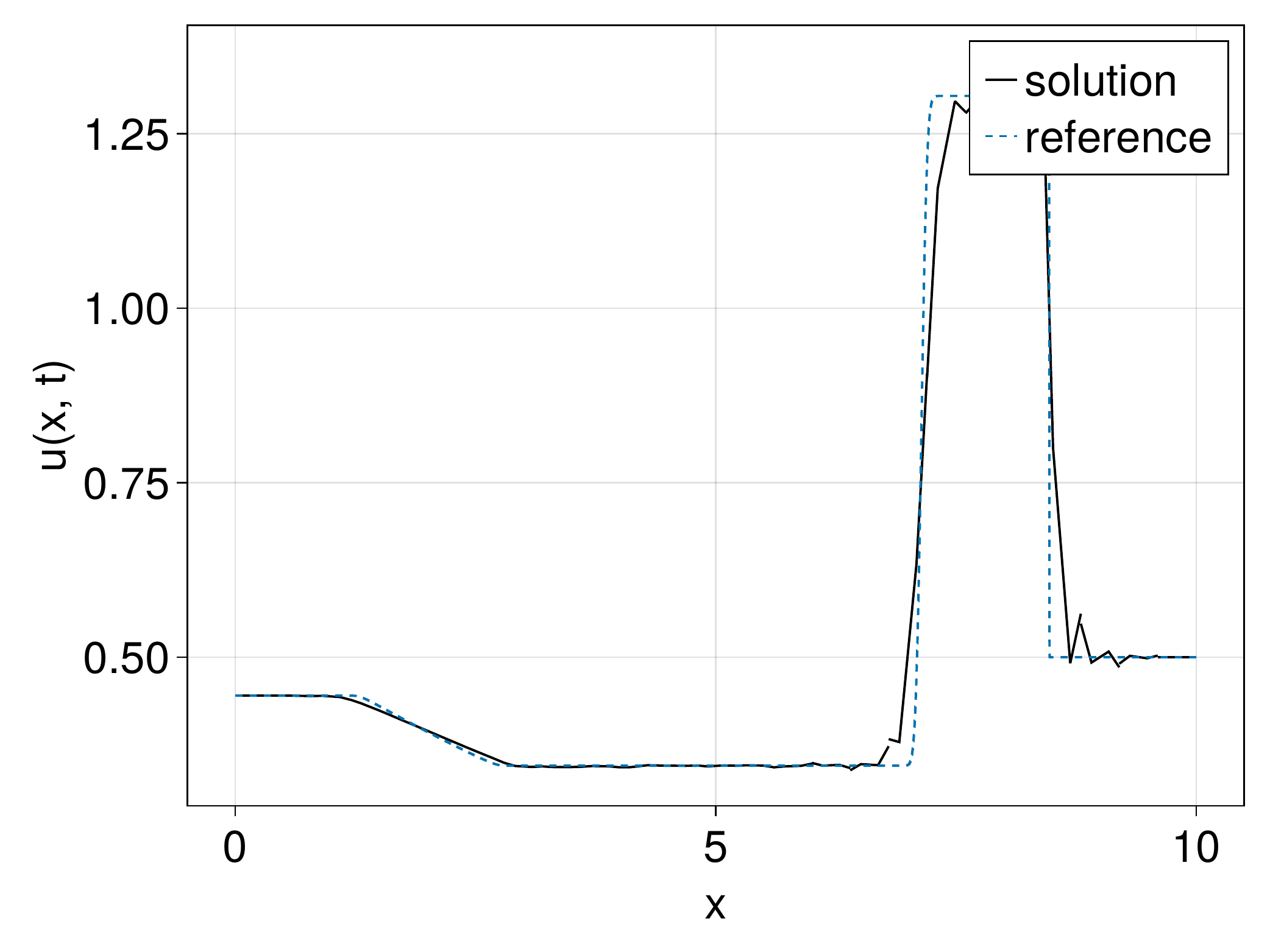}
			\caption{Density, $25$ cells}
		\end{subfigure}
		\begin{subfigure}{0.49 \textwidth}
			\includegraphics[width=\textwidth]{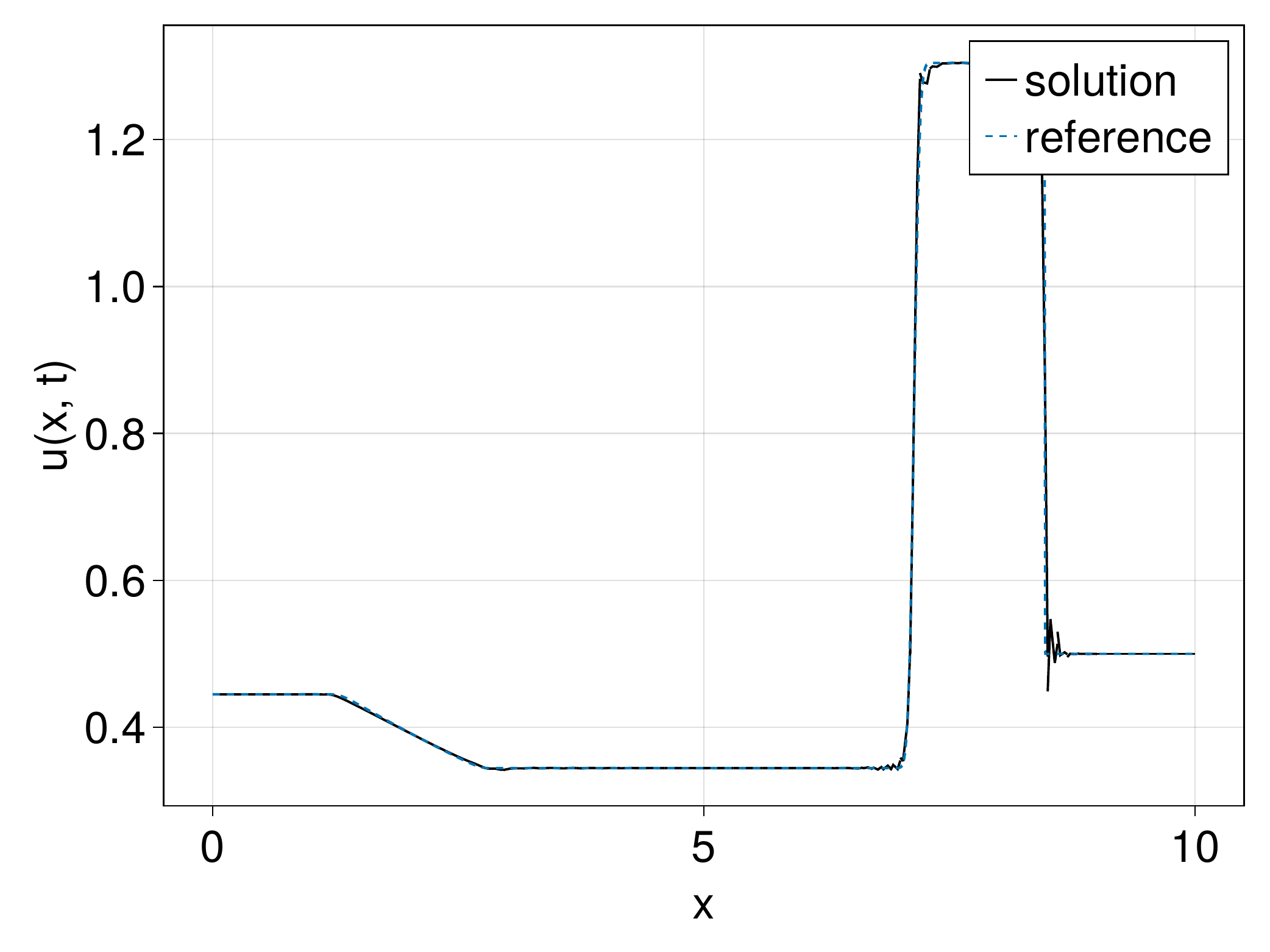}
			\caption{Density, $100$ cells}
		\end{subfigure}
		\begin{subfigure}{0.49 \textwidth}
			\includegraphics[width=\textwidth]{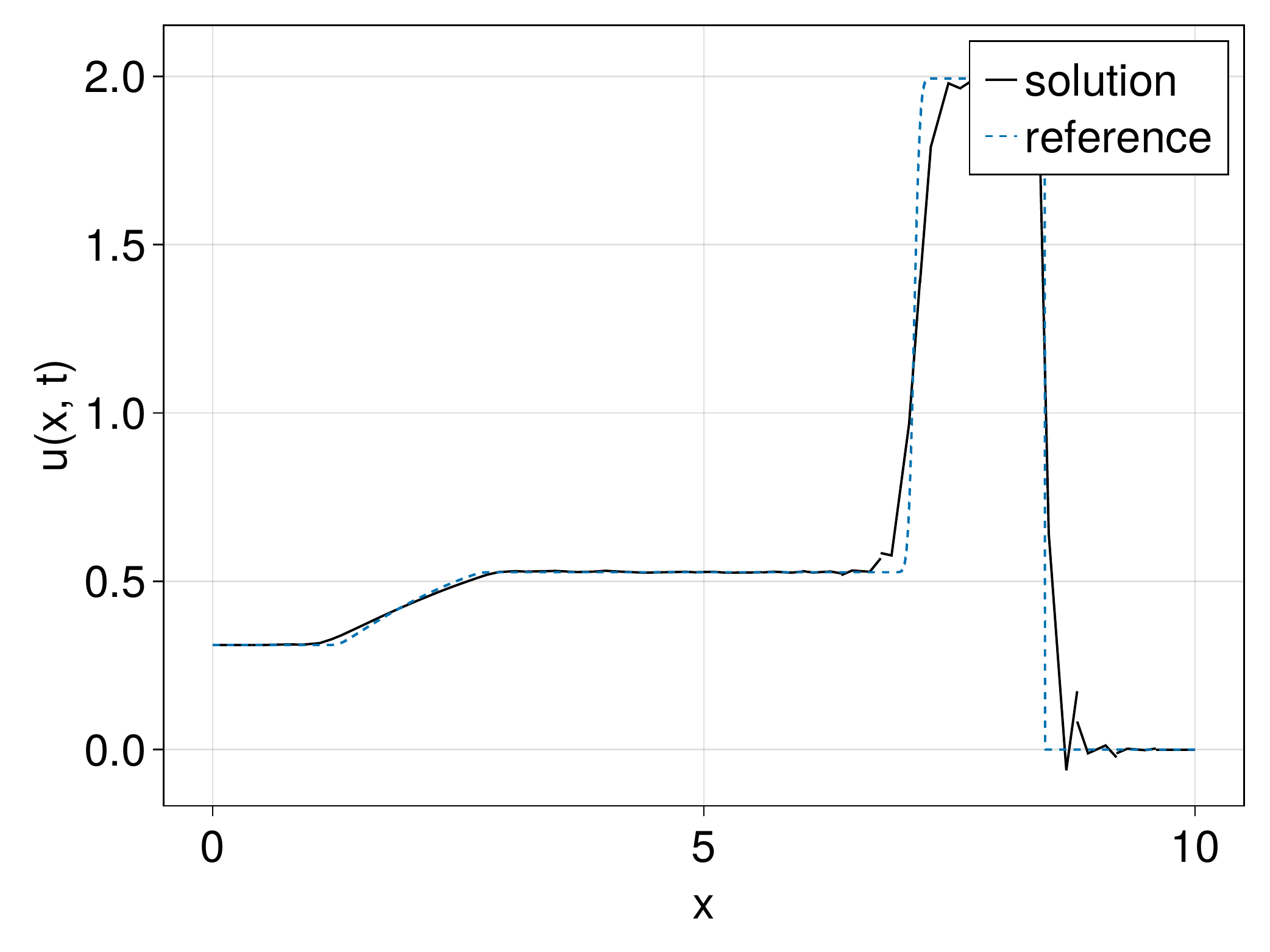}
			\caption{Moment density, $25$ cells}
		\end{subfigure}
		\begin{subfigure}{0.49 \textwidth}
			\includegraphics[width=\textwidth]{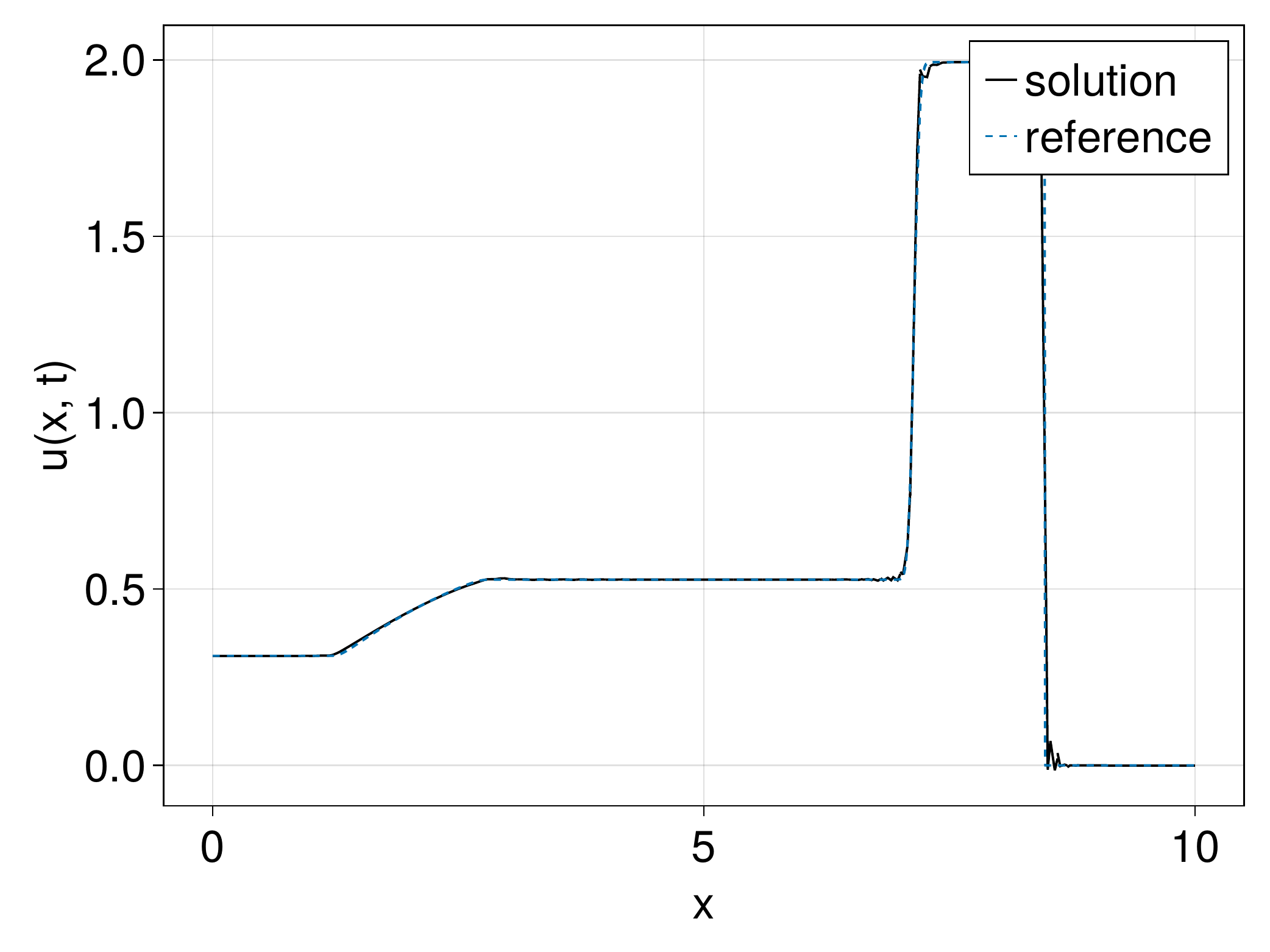}
			\caption{Moment density, $100$ cells}
		\end{subfigure}
		\begin{subfigure}{0.49 \textwidth}
			\includegraphics[width=\textwidth]{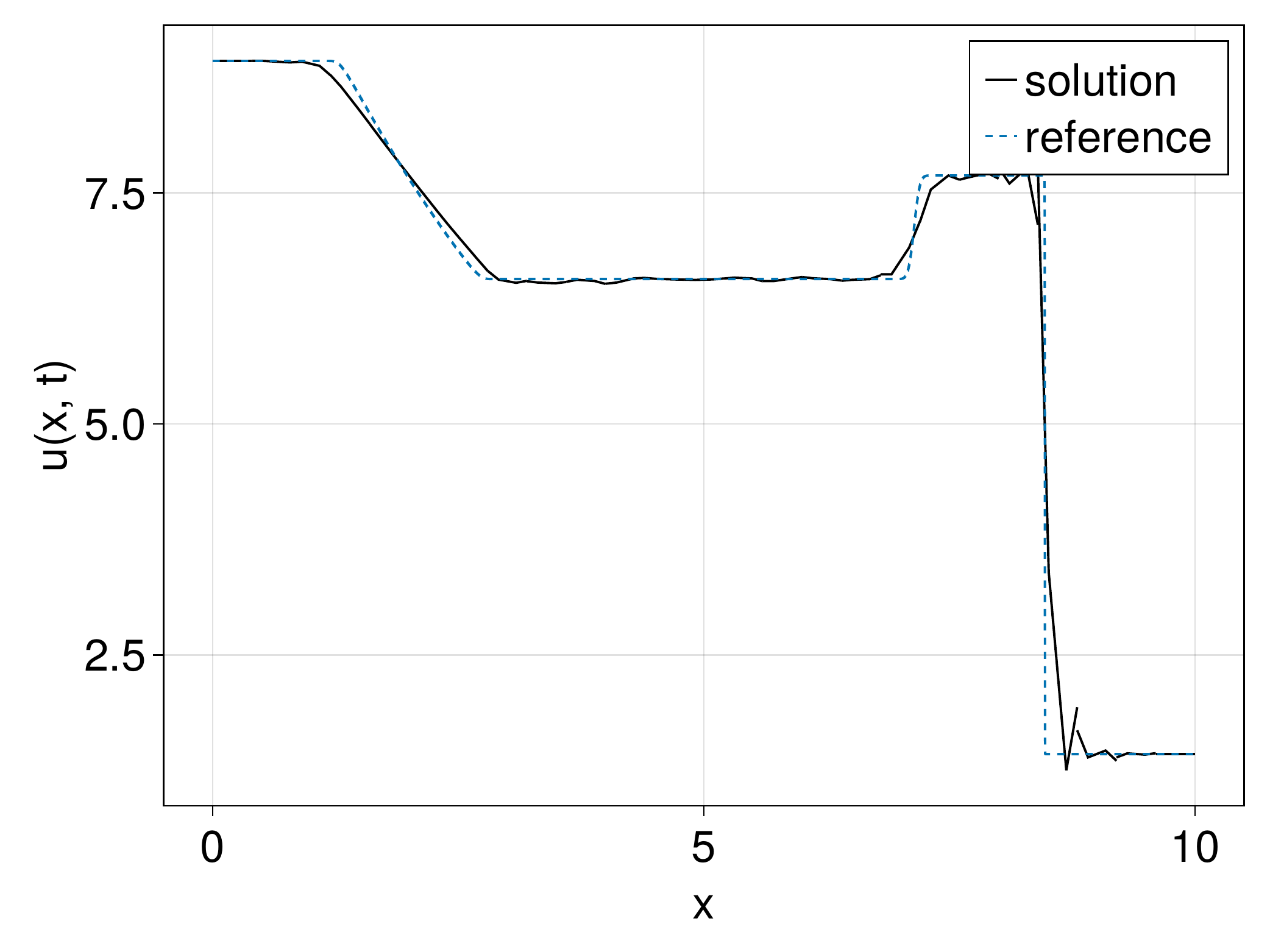}
			\caption{Energy density, $25$ cells}
		\end{subfigure}
		\begin{subfigure}{0.49 \textwidth}
			\includegraphics[width=\textwidth]{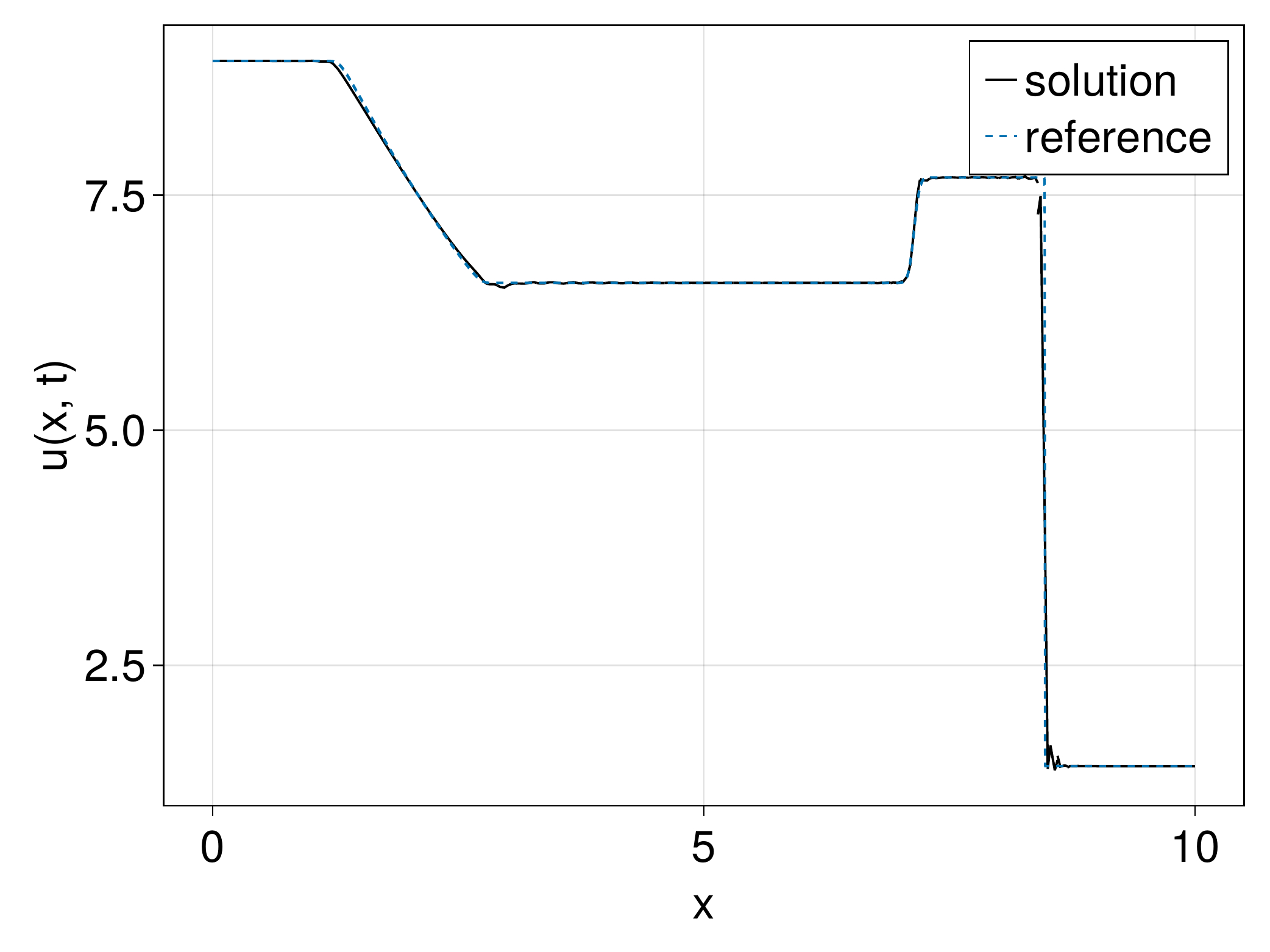}
			\caption{Energy density, $100$ cells}
		\end{subfigure}
		\caption{Shock tube 2 at $t = 1.2$ with $25$ cells corresponding to $100$ degrees of freedom and $100$ cells corresponding to $400$ degrees of freedom (p=3).}
		\label{fig:lax3}
	\end{figure}
	\begin{figure}
	\begin{subfigure}{0.49 \textwidth}
		\includegraphics[width=\textwidth]{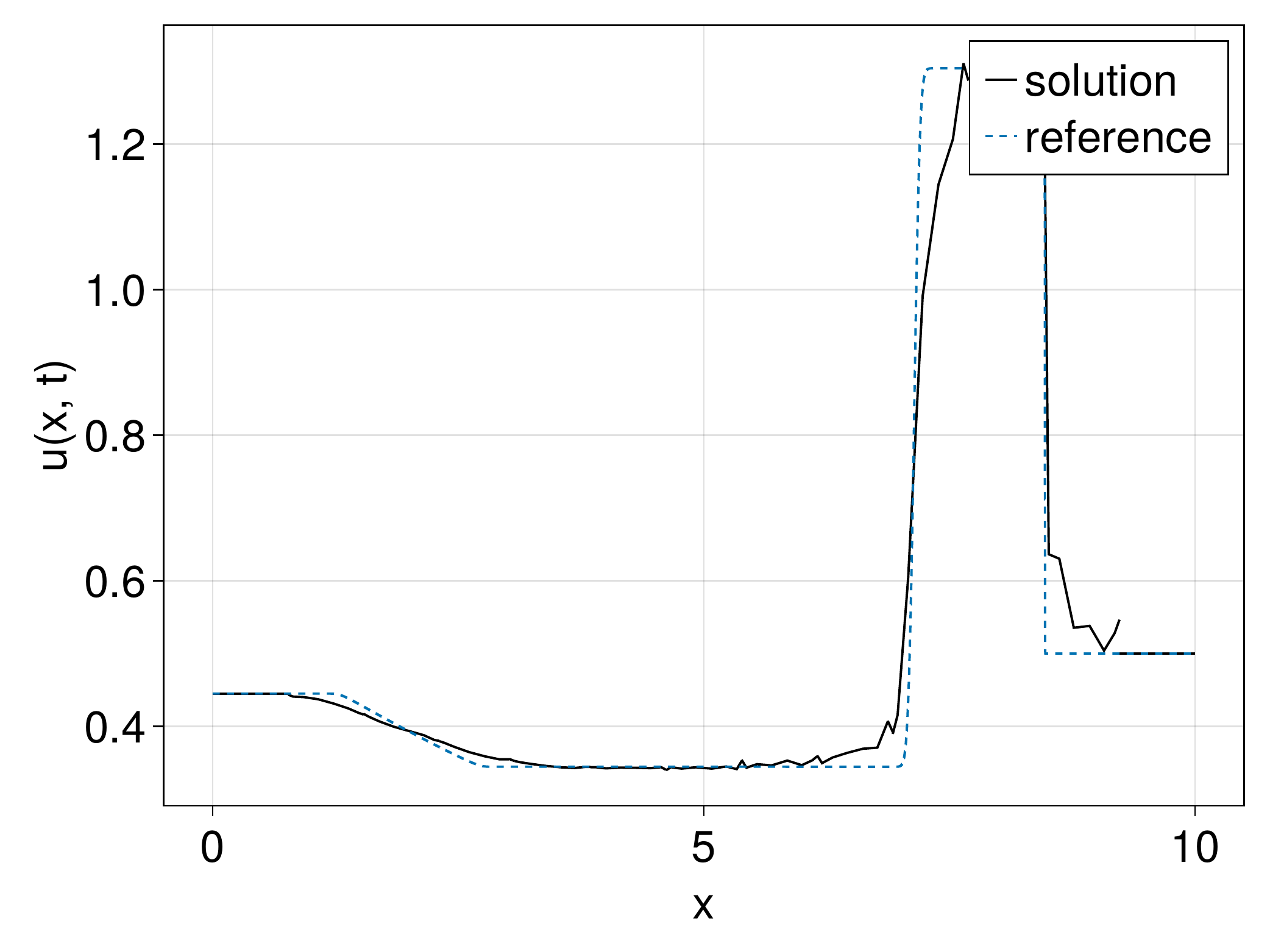}
		\caption{Density, $13$ cells}
	\end{subfigure}
	\begin{subfigure}{0.49 \textwidth}
		\includegraphics[width=\textwidth]{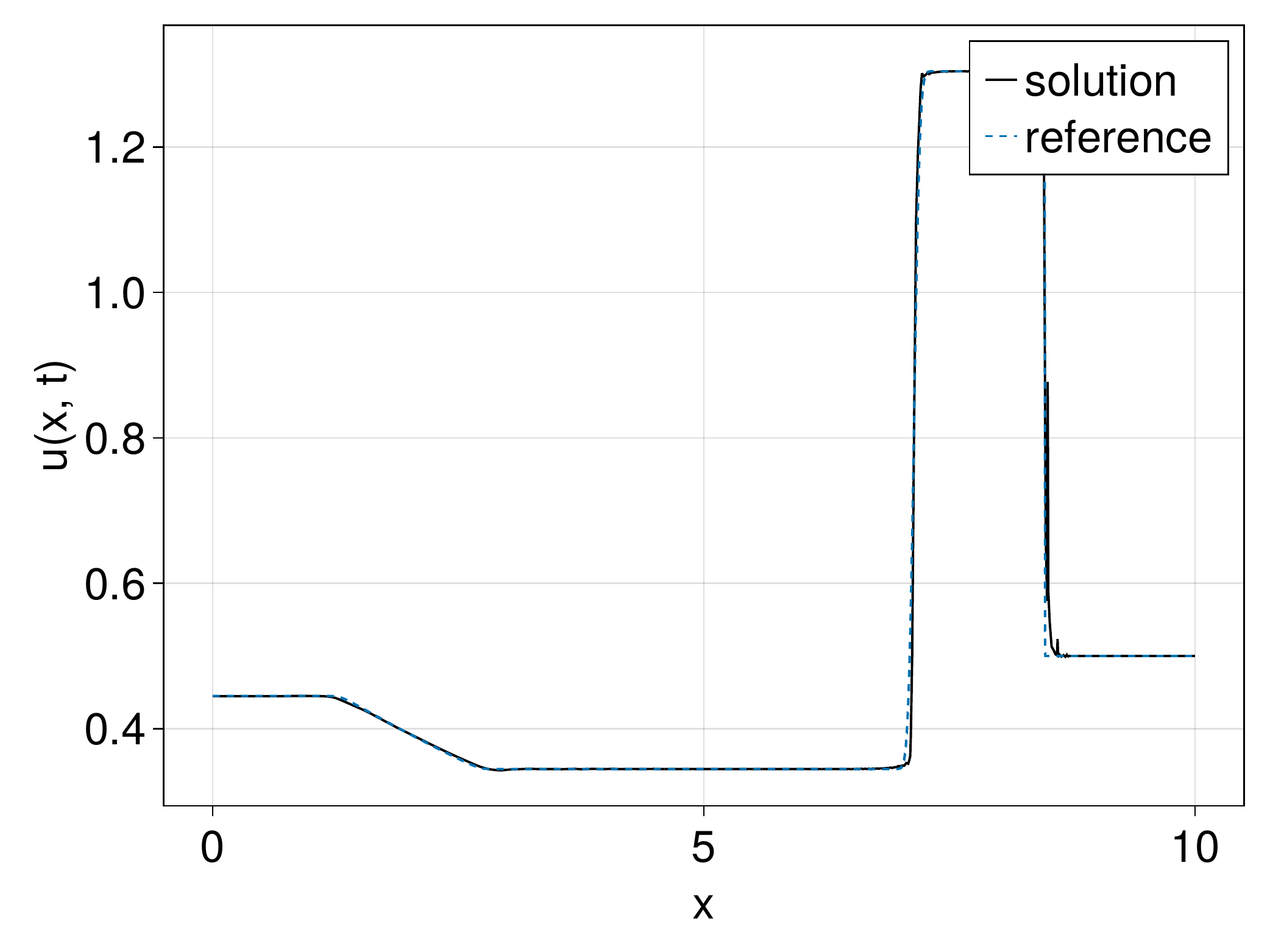}
		\caption{Density, $100$ cells}
	\end{subfigure}
	\begin{subfigure}{0.49 \textwidth}
		\includegraphics[width=\textwidth]{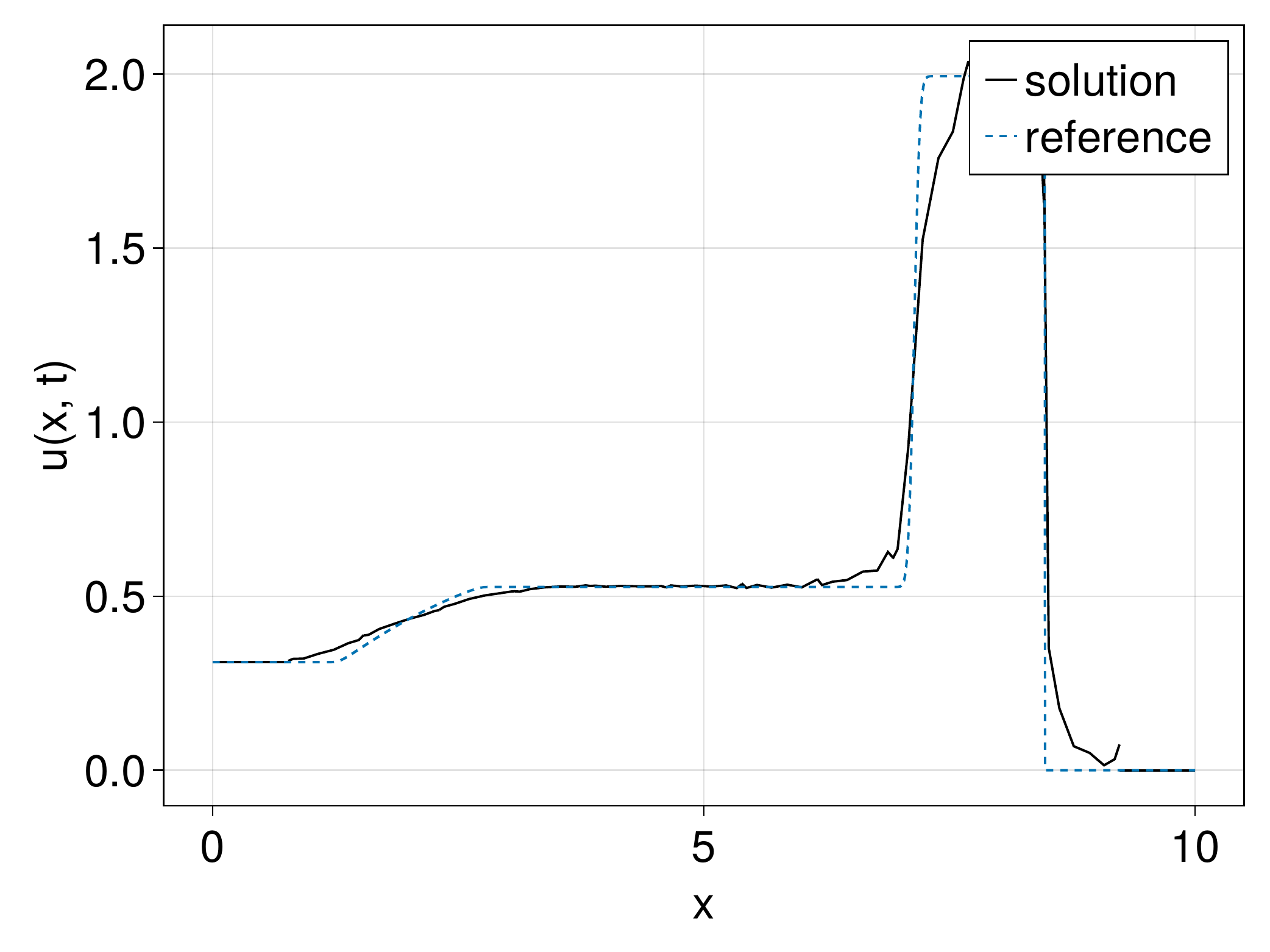}
		\caption{Moment density, $13$ cells}
	\end{subfigure}
	\begin{subfigure}{0.49 \textwidth}
		\includegraphics[width=\textwidth]{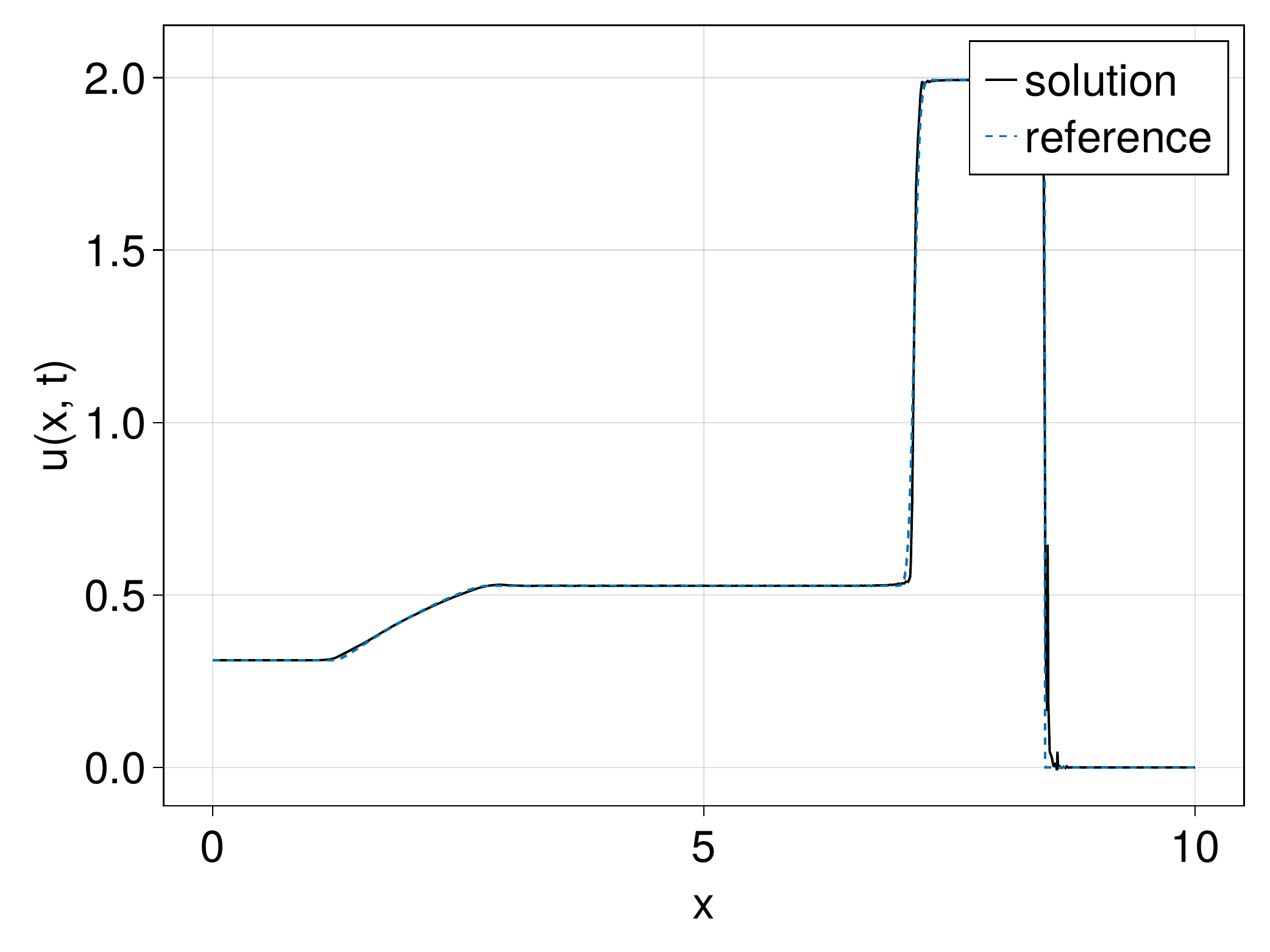}
		\caption{Moment density, $100$ cells}
	\end{subfigure}
	\begin{subfigure}{0.49 \textwidth}
		\includegraphics[width=\textwidth]{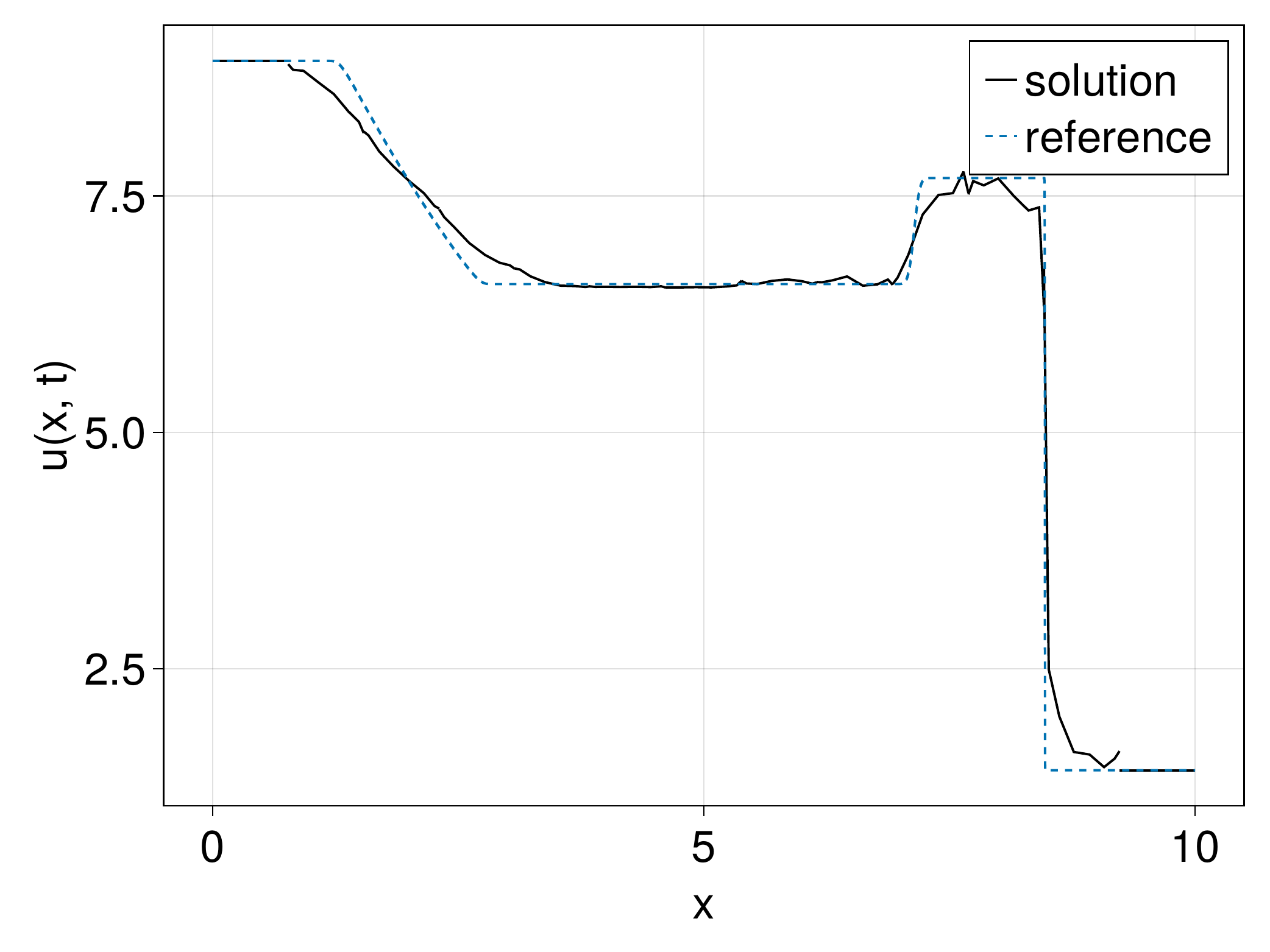}
		\caption{Energy density, $13$ cells}
	\end{subfigure}
	\begin{subfigure}{0.49 \textwidth}
		\includegraphics[width=\textwidth]{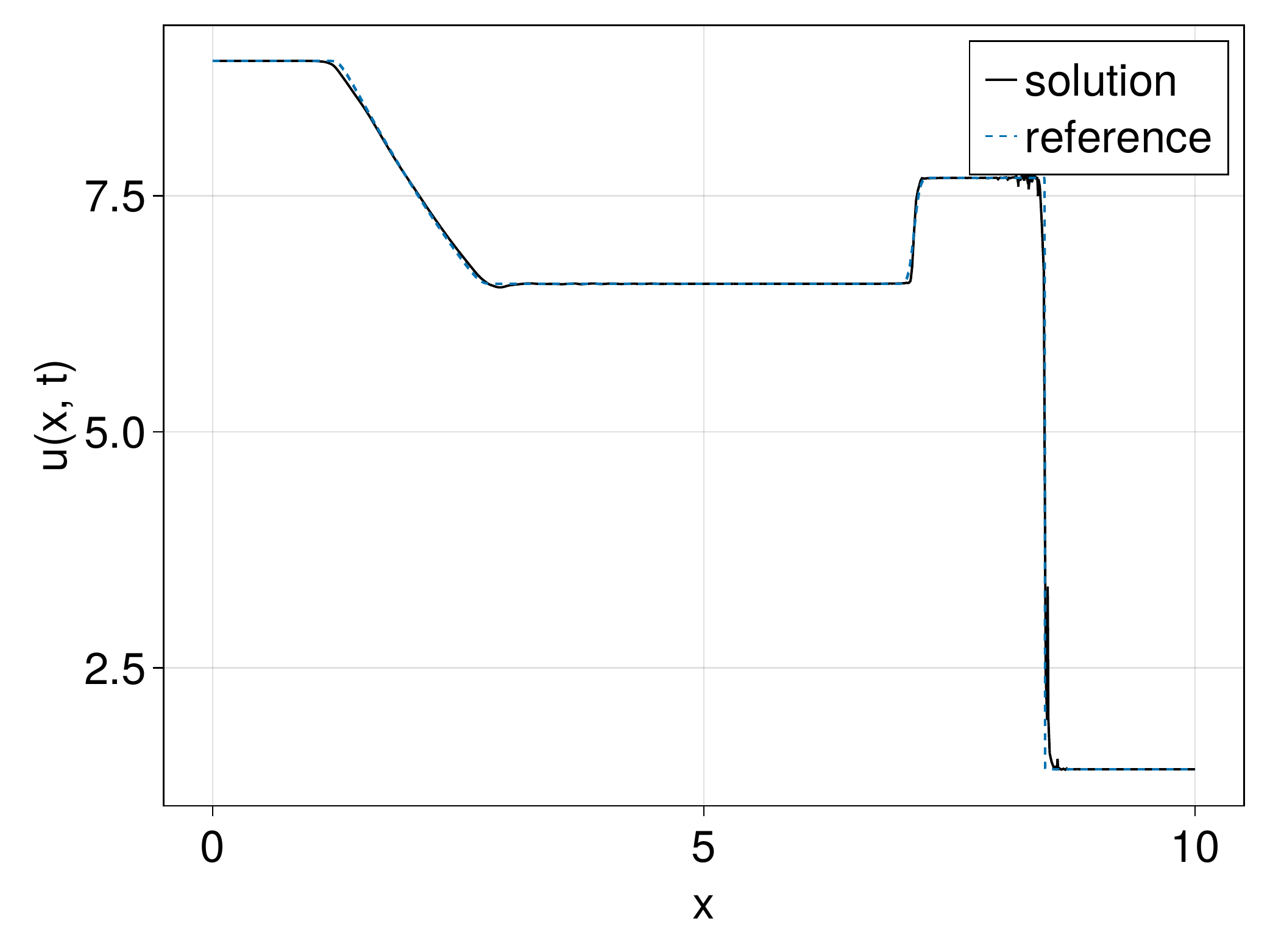}
		\caption{Energy density, $100$ cells}
	\end{subfigure}
	\caption{Shock tube 2 at $t = 1.2$ with $25$ cells corresponding to $100$ degrees of freedom and $100$ cells corresponding to $400$ degrees of freedom (p=7).}
	\label{fig:lax7}
\end{figure}
The shock tube tests were always carried out for two different numbers of cells. First for $N = N_{\text{typ}}/(p +1)$ cells, where $N_{typ} = 100$ is the usual number of cells used in comparisons for Finite-Volume methods. This was done so that the same number of degrees of freedom has to be saved. The results look satisfactory and highlight the effectivity of the method in figures \ref{fig:sod3}, \ref{fig:sod7}, \ref{fig:lax3}, \ref{fig:lax7}. All shocks are sharp and concentrated to less than one cell width. Yet, only slight overshoots and oscillations are visible directly around the shocks. These distortions are confined to the cell directly next to the shock. Contact discontinuities are slightly smeared over one cell, but after they have been smeared to this width no additional smearing takes place. The computational complexity per timestep is still low as no recovery stencil selection has to be carried out and only $1/(p + 1)$ times the number of two-point fluxes need to be evaluated. Because some other publications use 100 cells also for DG methods we carried out the tests once more for $N = 100$ cells, amounting to $400$ and $800$ degrees of freedom for orders $p=3$ and $p=7$. 

\subsection{Numerical validation of the entropy rate criterion}
To verify the entropy rate criterion the total entropy of the solution to the first shock tube above was compared to the solution calculated by a Lax-Friedrichs scheme with $3\cdot 10^4$ cells. Similar comparisons were carried out in \cite{Klein2022Using,klein2023stabilizing,KS2023EAR}. Please note that the Godunov solver used previously was swapped for a LF scheme to evade the need for an exact Riemann solver. This is also supported by our finding in lemma \ref{lem:LFspeed} and corollary \ref{cor:eiep} as a Lax-Friedrichs solution therefore has to comply with the entropy rate criterion. A scheme should in these comparisons have the same entropy dissipation rate (in the limit) as the Lax-Friedrichs scheme in the limit. Comparisons for orders $p=3$ and $p=7$ in figure \ref{fig:Etest} show that this seems to be the case. The DG scheme always has an entropy that lies below the entropy of the LF scheme. As the entropy inequality for vanishing viscosity solutions is also desirable it was also verified on a per-cell basis. We just note that the small positive violations in figure \ref{fig:Etest} are of the same magnitude as the precision achievable during the calculation of $\lambda$ using our procedure with double precision floats.
\begin{figure}
	\begin{subfigure}{0.49 \textwidth}
	\includegraphics[width=\textwidth]{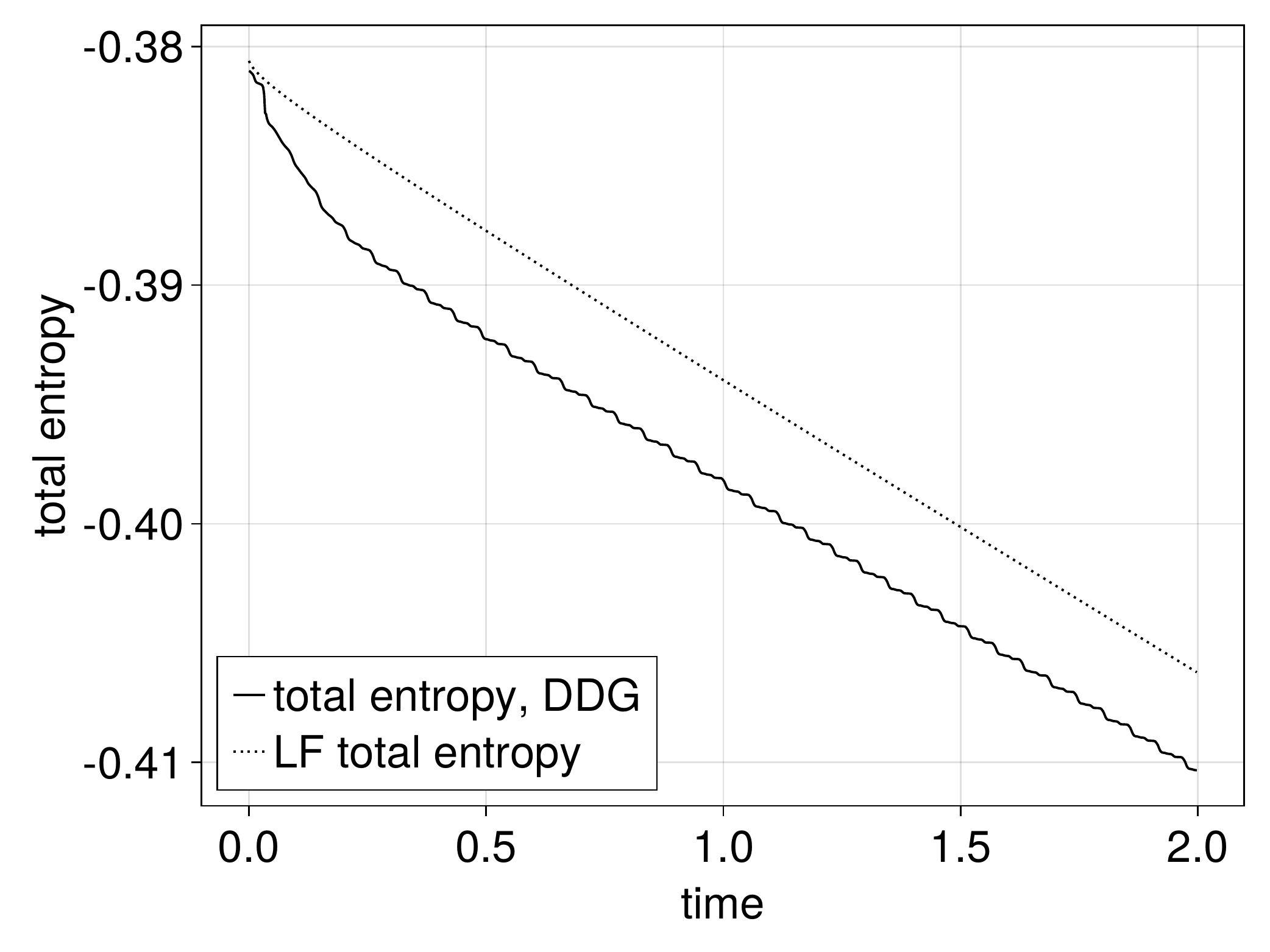}
	\caption{Plot of the total entropy ($p=3$)}
	\end{subfigure}
	\begin{subfigure}{0.49 \textwidth}
	\includegraphics[width= \textwidth]{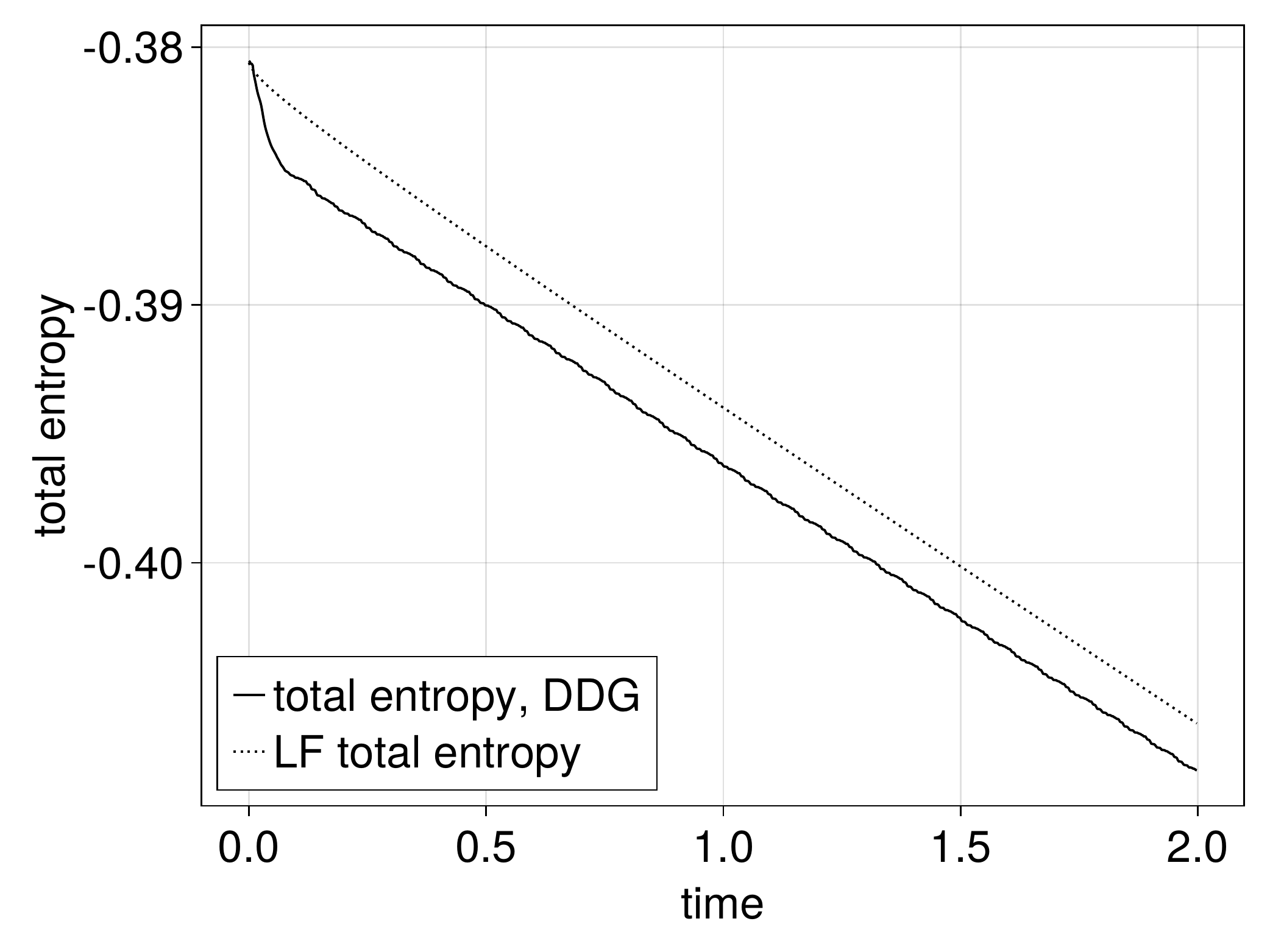}
	\caption{Plot of the total entropy $(p = 7)$}
	\end{subfigure}
	\begin{subfigure}{0.49 \textwidth}
		\includegraphics[width=\textwidth]{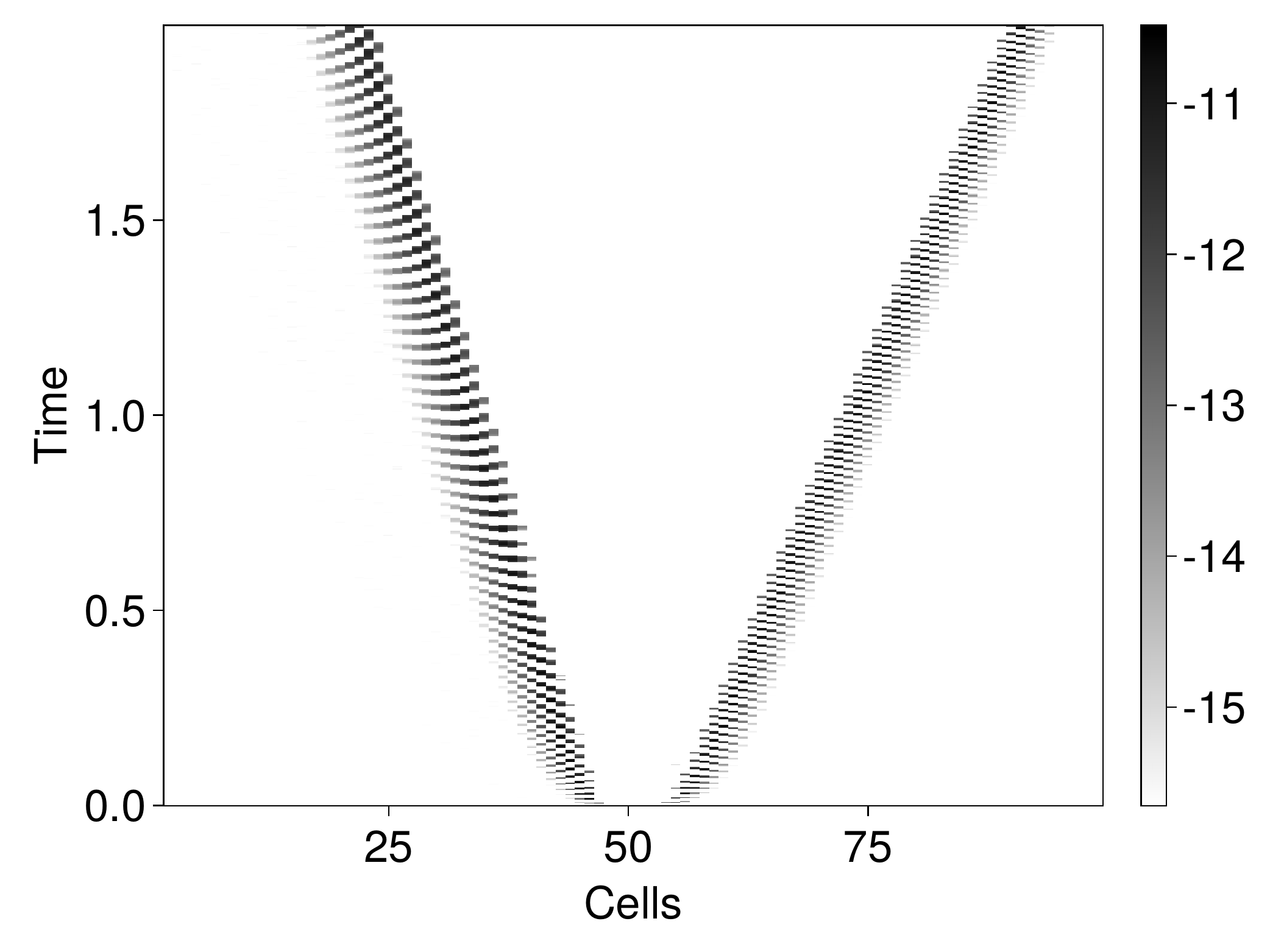}
		\caption{Logarithm of the positive violation of the entropy equality ($p = 3$)}
	\end{subfigure}
	\begin{subfigure}{0.49 \textwidth}
		\includegraphics[width=\textwidth]{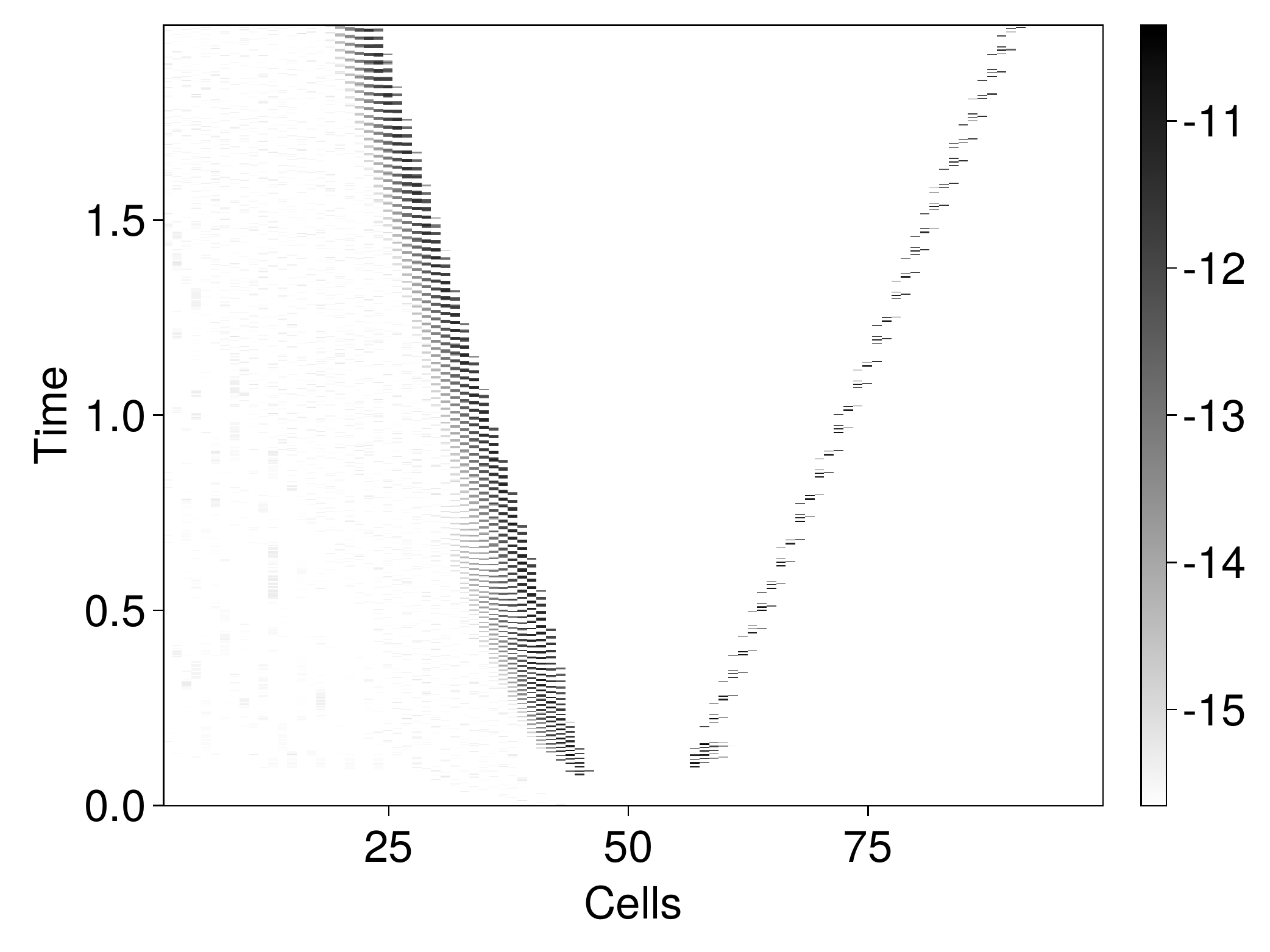}
		\caption{Logarithm of the positive violation of the entropy equality ($p = 7$)}
	\end{subfigure}
	\begin{subfigure}{0.49 \textwidth}
		\includegraphics[width=\textwidth]{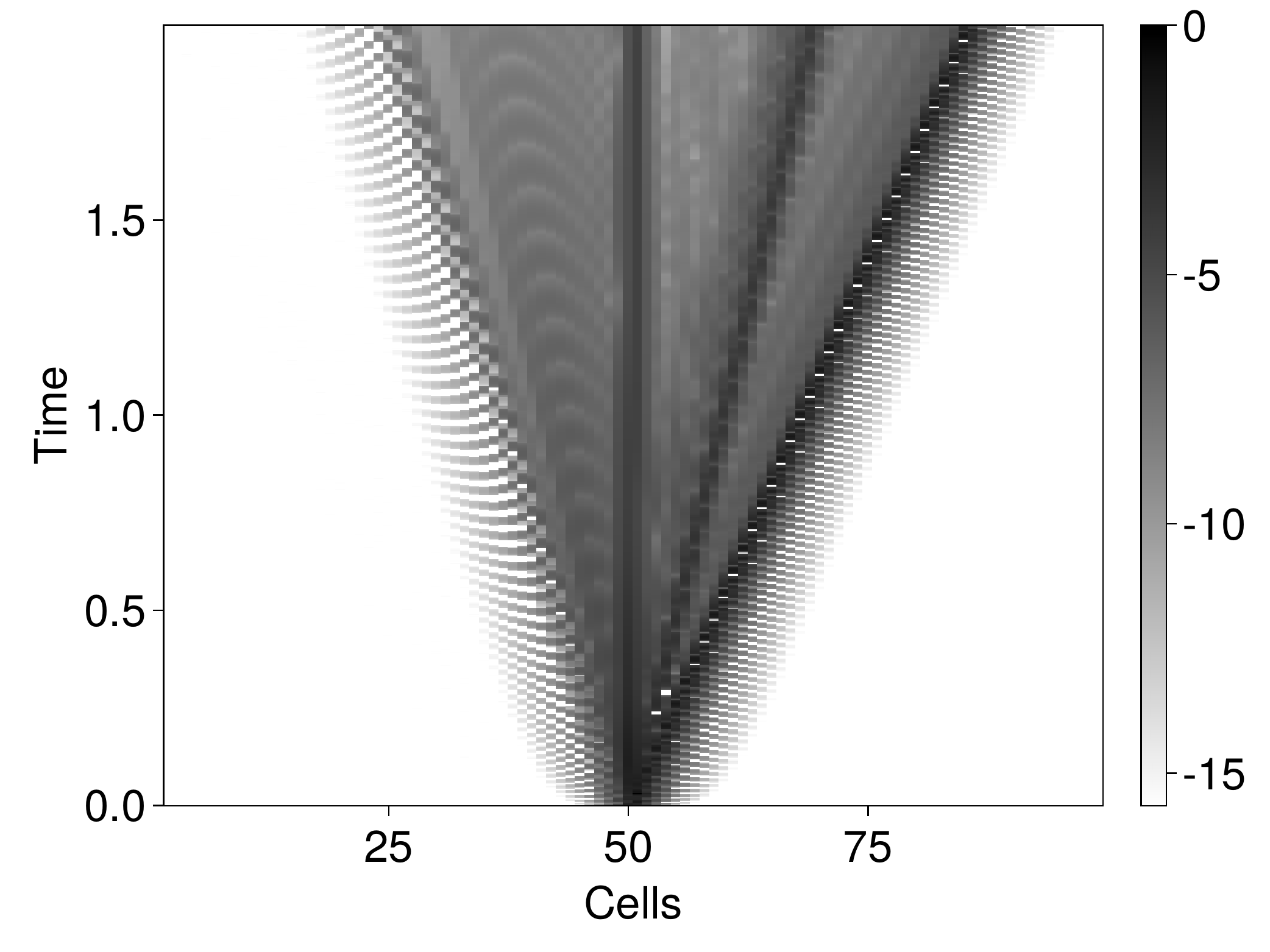}
		\caption{Logarithm of the negative violation of the entropy equality ($p = 3$)}
	\end{subfigure}
	\begin{subfigure}{0.49 \textwidth}
		\includegraphics[width=\textwidth]{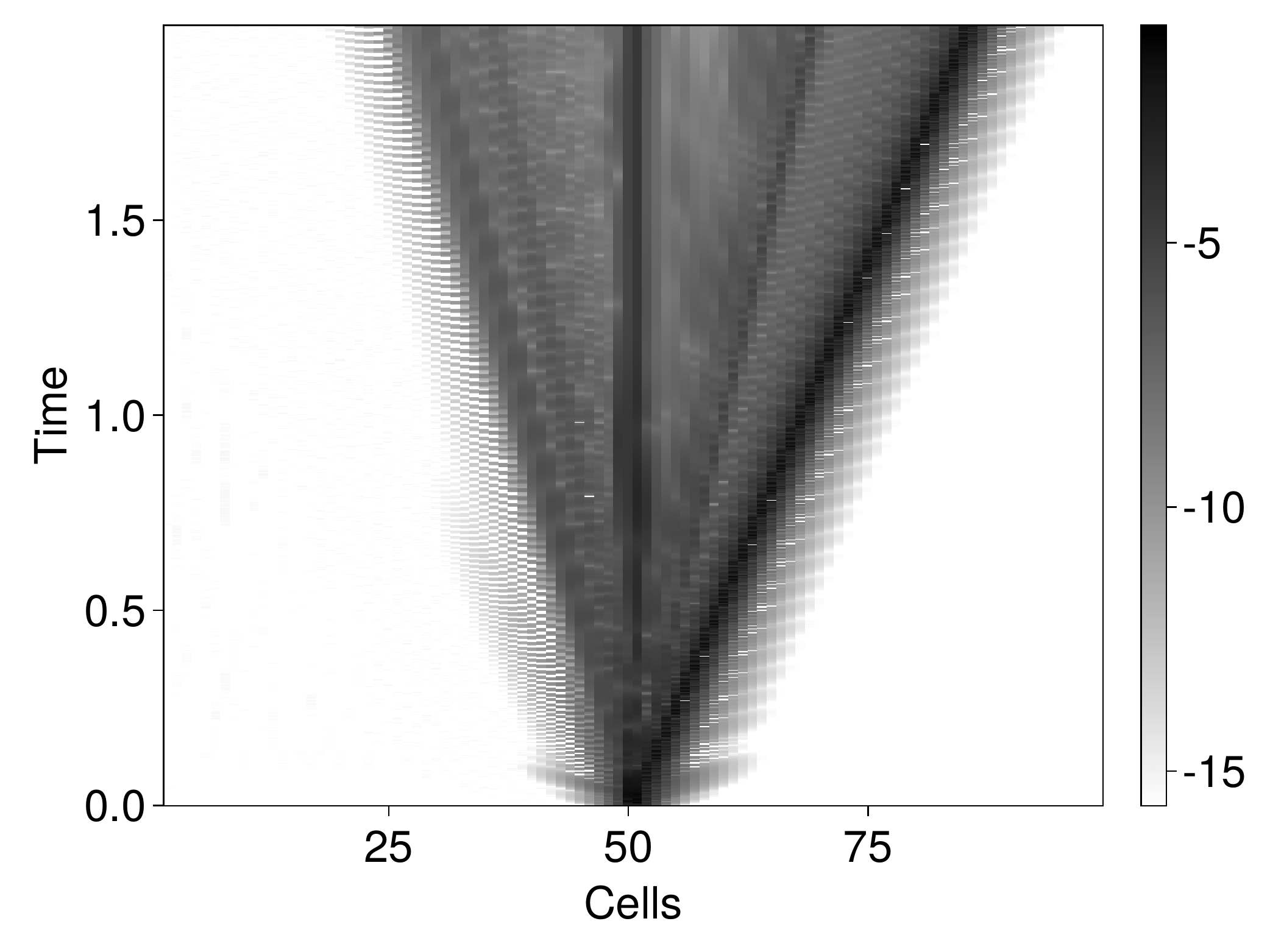}
		\caption{Logarithm of the negative violation of the entropy equality ($p = 7$)}
	\end{subfigure}
		\caption{Entropy tests for the first shock tube.}
		\label{fig:Etest}
\end{figure}

\subsection{Shu-Osher test}
To showcase a combination of shocks and smooth areas the well established shock-sine interaction problem from \cite[Problem 8]{SO1989} was tested. The initial conditions are given by
	\begin{align*}
	\rho_0(x, 0) = \begin{cases}3.857153  \\ 1 + \epsilon \sin(5 x)  \end{cases} 
	\quad 
	v_0(x, 0) = \begin{cases} 2.629  \\ 0  \end{cases}
	p_0(x, 0) = \begin{cases} 10.333 & x < 1 \\ 1 & x \geq 1 \end{cases}.
\end{align*}
The parameter $\epsilon$ was set to the canonical value of $\epsilon = 0.2$.
	\begin{figure}
		\begin{subfigure}{0.49\textwidth}
		\includegraphics[width=\textwidth]{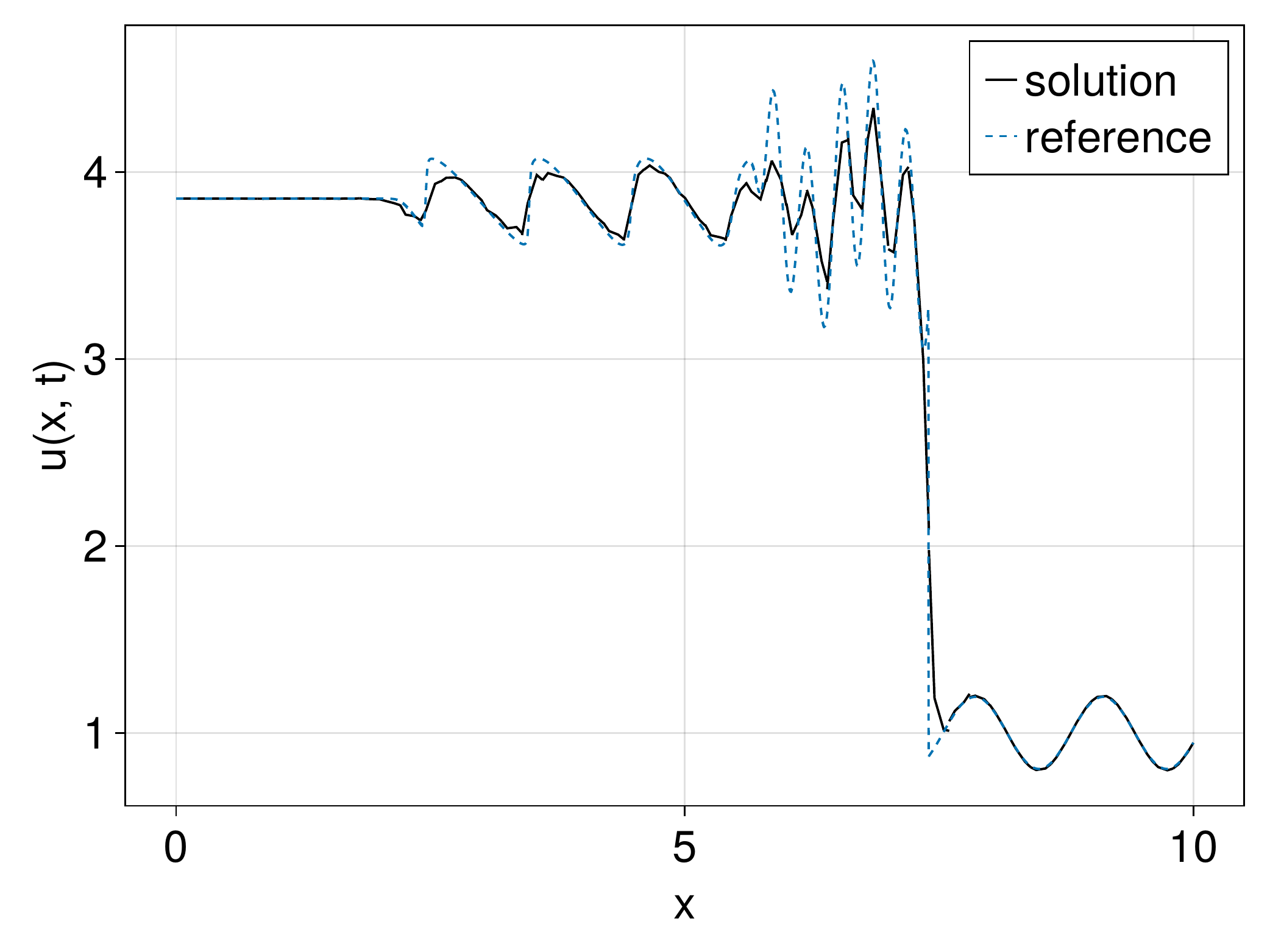}
		\caption{$p = 3$, $50$ cells.}
		\end{subfigure}
		\begin{subfigure}{0.49\textwidth}
		\includegraphics[width=\textwidth]{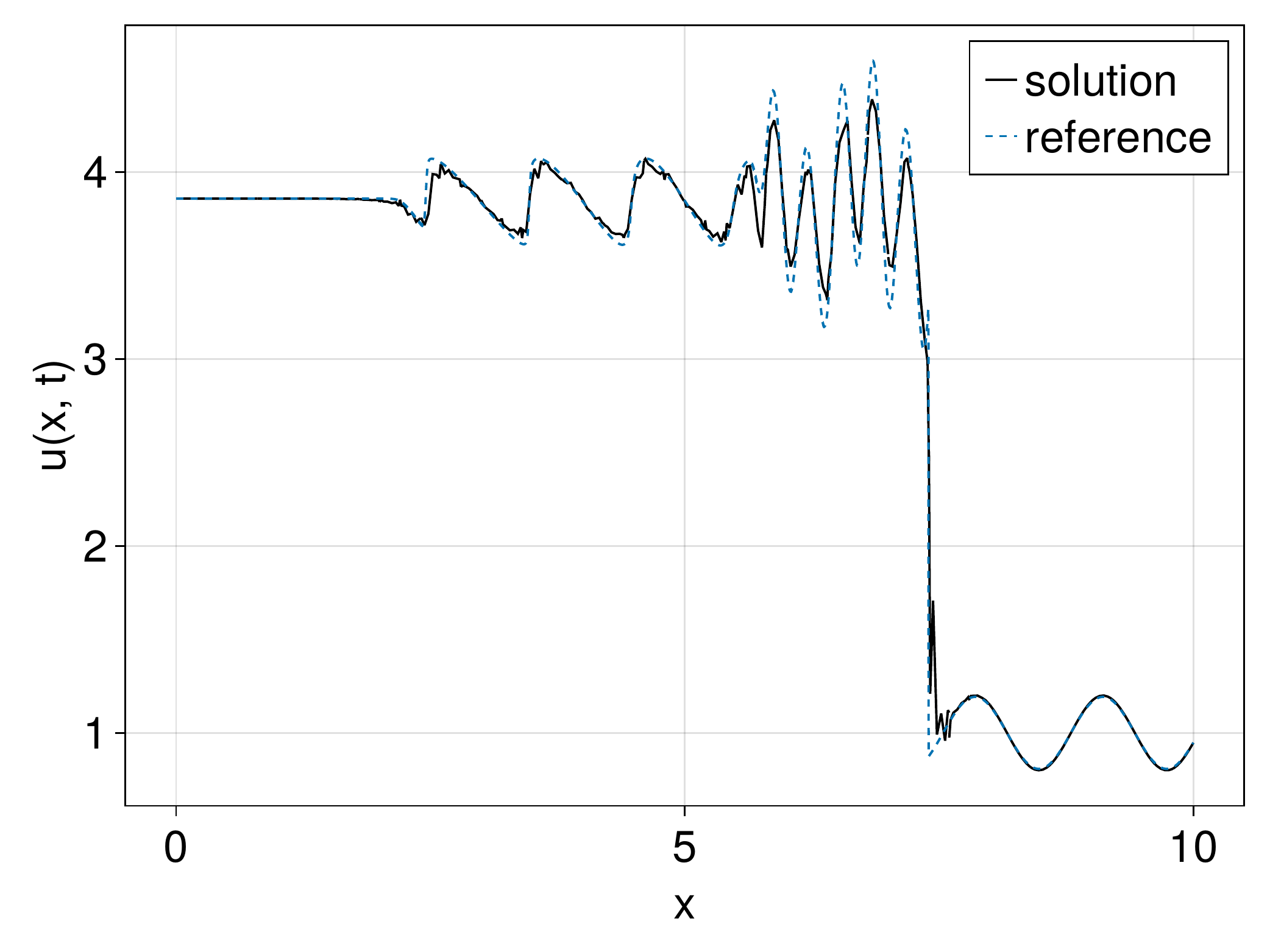}
		\caption{$p = 7$, $50$ cells.}
	\end{subfigure}
	\begin{subfigure}{0.49\textwidth}
	\includegraphics[width=\textwidth]{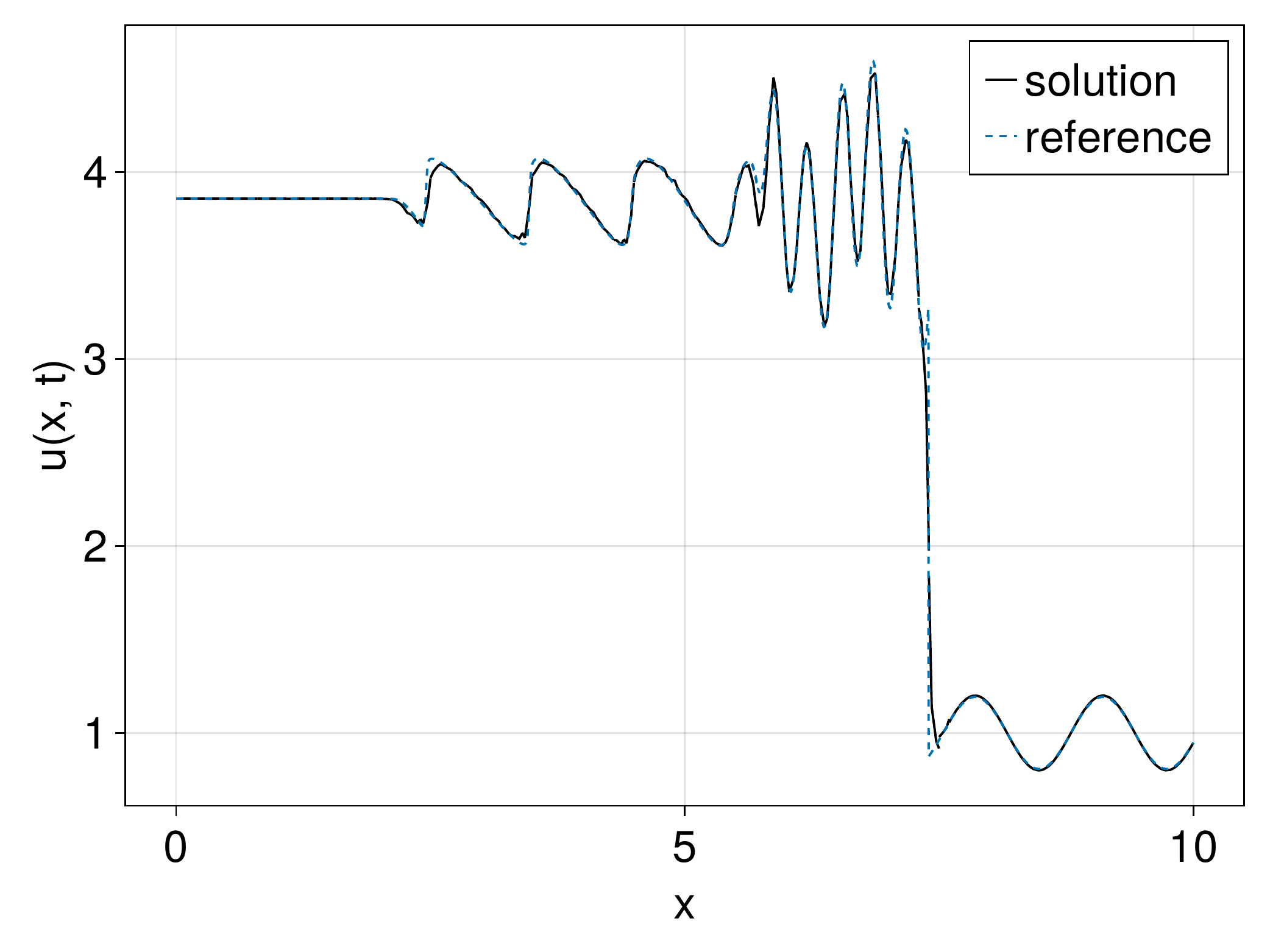}
	\caption{$p = 3$, $100$ cells.}
	\end{subfigure}
		\begin{subfigure}{0.49\textwidth}
		\includegraphics[width=\textwidth]{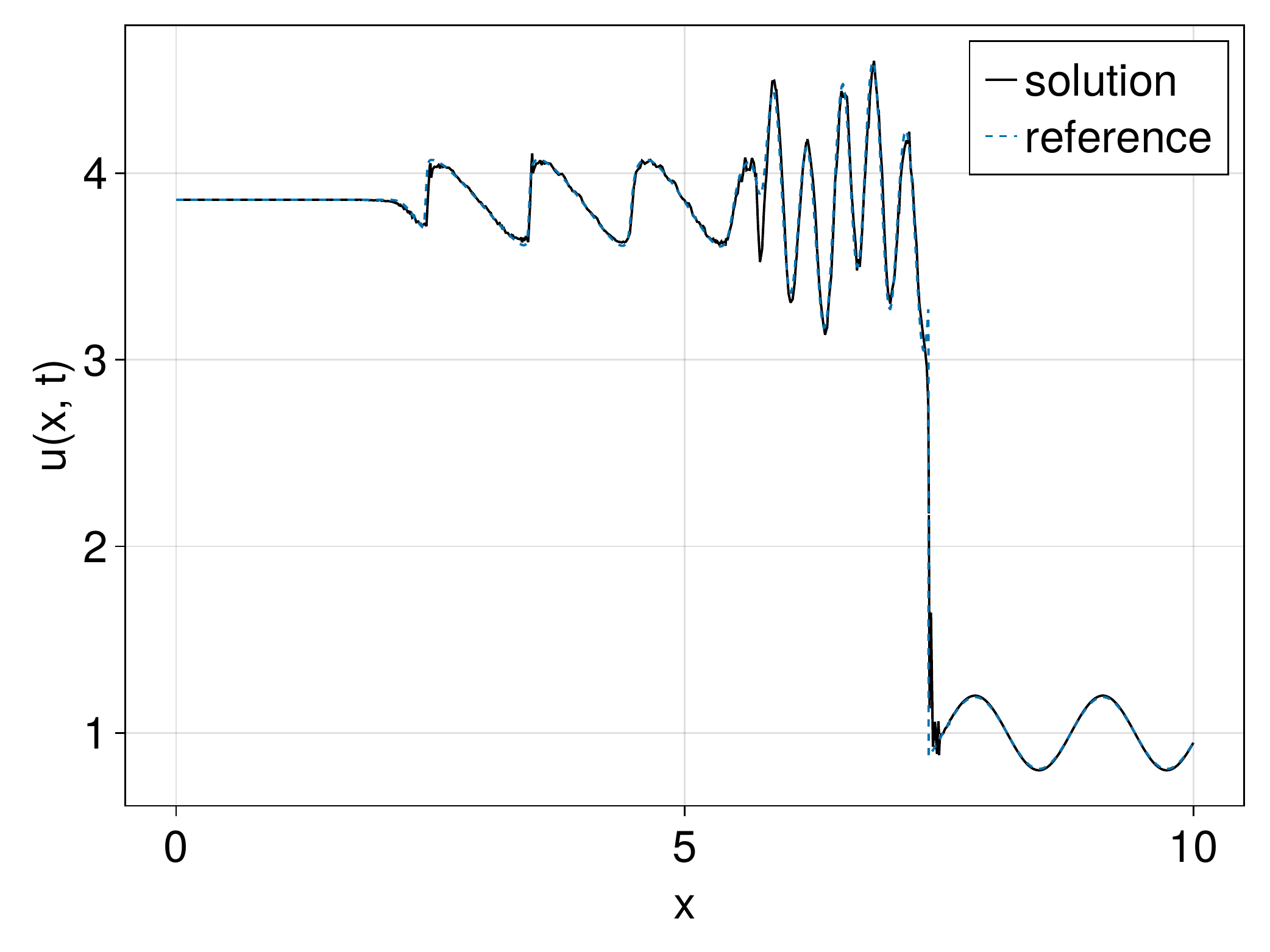}
		\caption{$p = 7$, $100$ cells.}
	\end{subfigure}
	\begin{subfigure}{0.49\textwidth}
	\includegraphics[width=\textwidth]{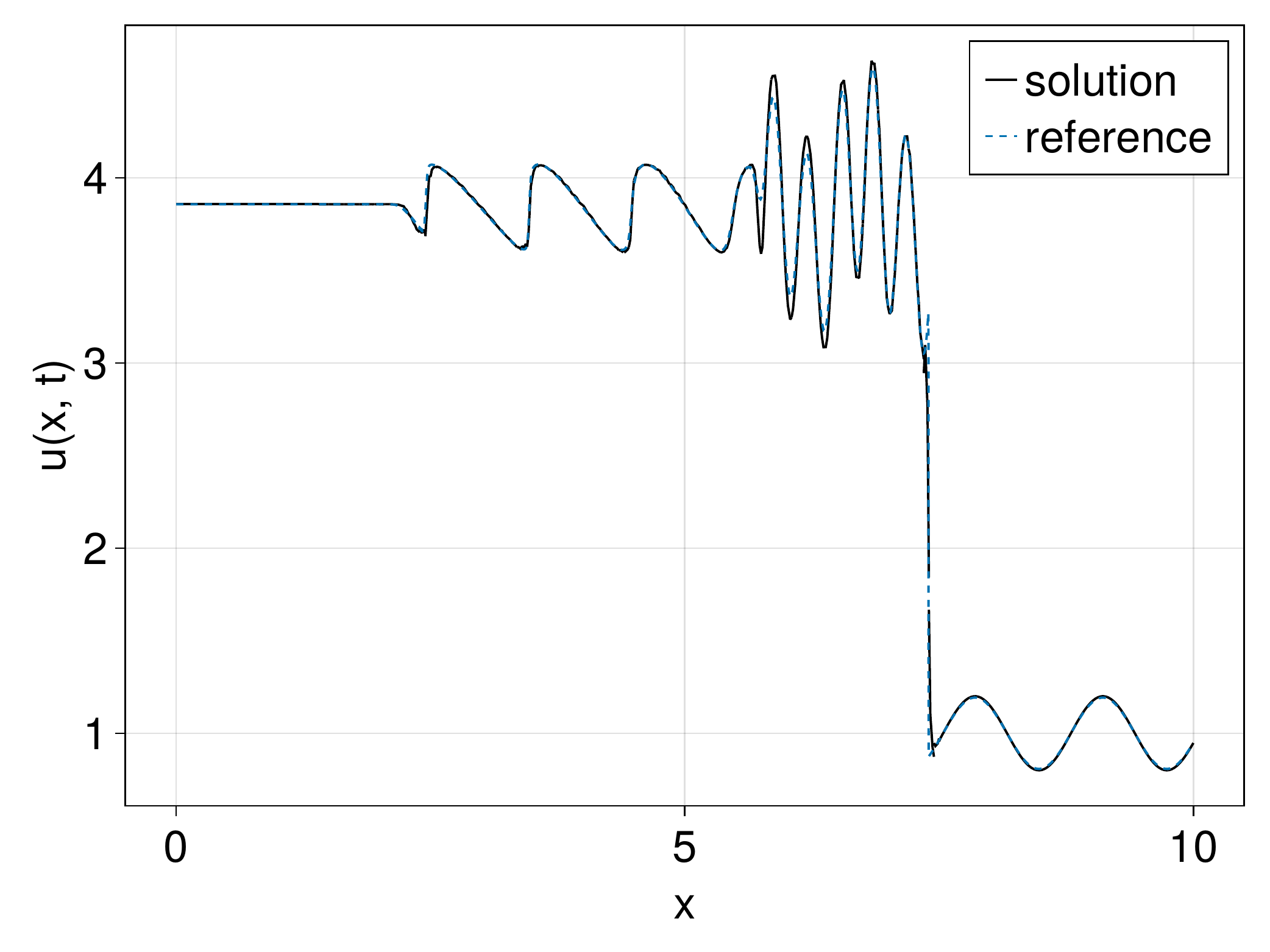}
	\caption{$p = 3$, $200$ cells.}
\end{subfigure}
	\begin{subfigure}{0.49\textwidth}
	\includegraphics[width=\textwidth]{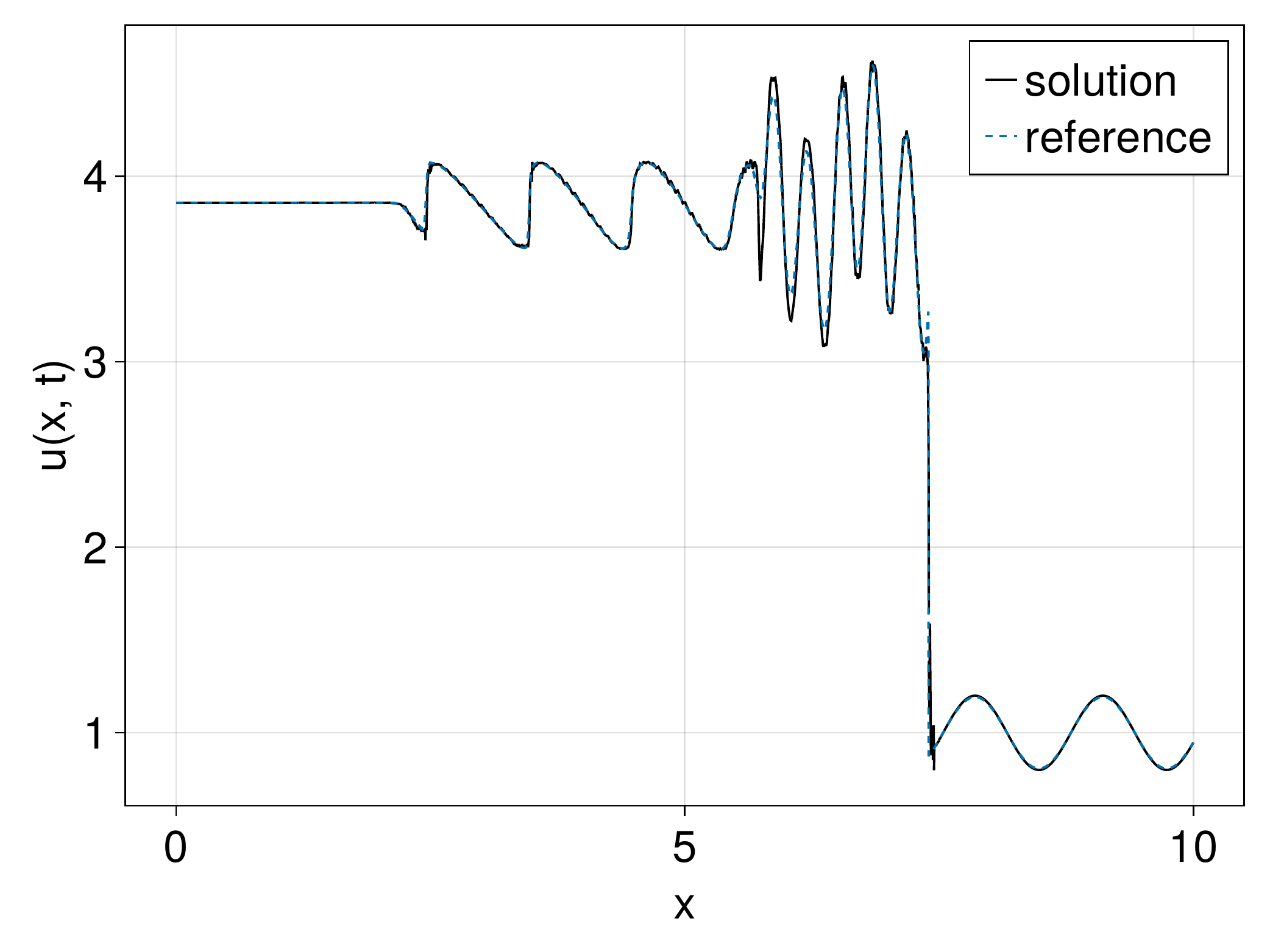}
	\caption{$p = 7$, $200$ cells.}
\end{subfigure}
\caption{Shu-Osher test for $50, 100$ and $200$ cells and order $p = 3$ and $p = 7$ and therefore $200, 400, 800$ and $1600$ degrees of freedom at $t = 1.8$.}
\label{fig:ShuOsh}
		\end{figure}
	The results look satisfactory already when only $N = 100$ cells are used in the calculation. Yet, we note that this already corresponds to $400$ and $800$ degrees of freedom for the selected orders. When $N=200$ cells are used the solution is nearly indistinguishable from the reference solution.

\subsection{Convergence Analysis}
While the main aim of our modification was to devise a new DG scheme usable for shock-capturing calculations the scheme also converges with high order of accuracy for smooth solutions in our experiments. As an example the solution of 
\begin{equation*}\label{eq_density}
	\rho_0(x, 0) = 3.857153 + \epsilon(x) \sin(2 x),  \quad v_0(x, 0) =  2.0, \quad p_0(x, 0) = 10.33333,
\end{equation*}
	with
\begin{equation*}
		 \epsilon(x) = \e^{(x-3)^2}.
\end{equation*}
and periodic boundary conditions was calculated using our modified DG method.
The analytical solution for this test problem is
\[
\rho(x, t) = 3.857153 + \epsilon(x-2t) \sin(2 x-4t),  \quad v(x, t) =  2.0, \quad p(x, t) = 10.33333,
\]
with suitable periodic boundary conditions. 
\begin{figure}
	\begin{subfigure}{0.5 \textwidth}
	\includegraphics[width=\textwidth]{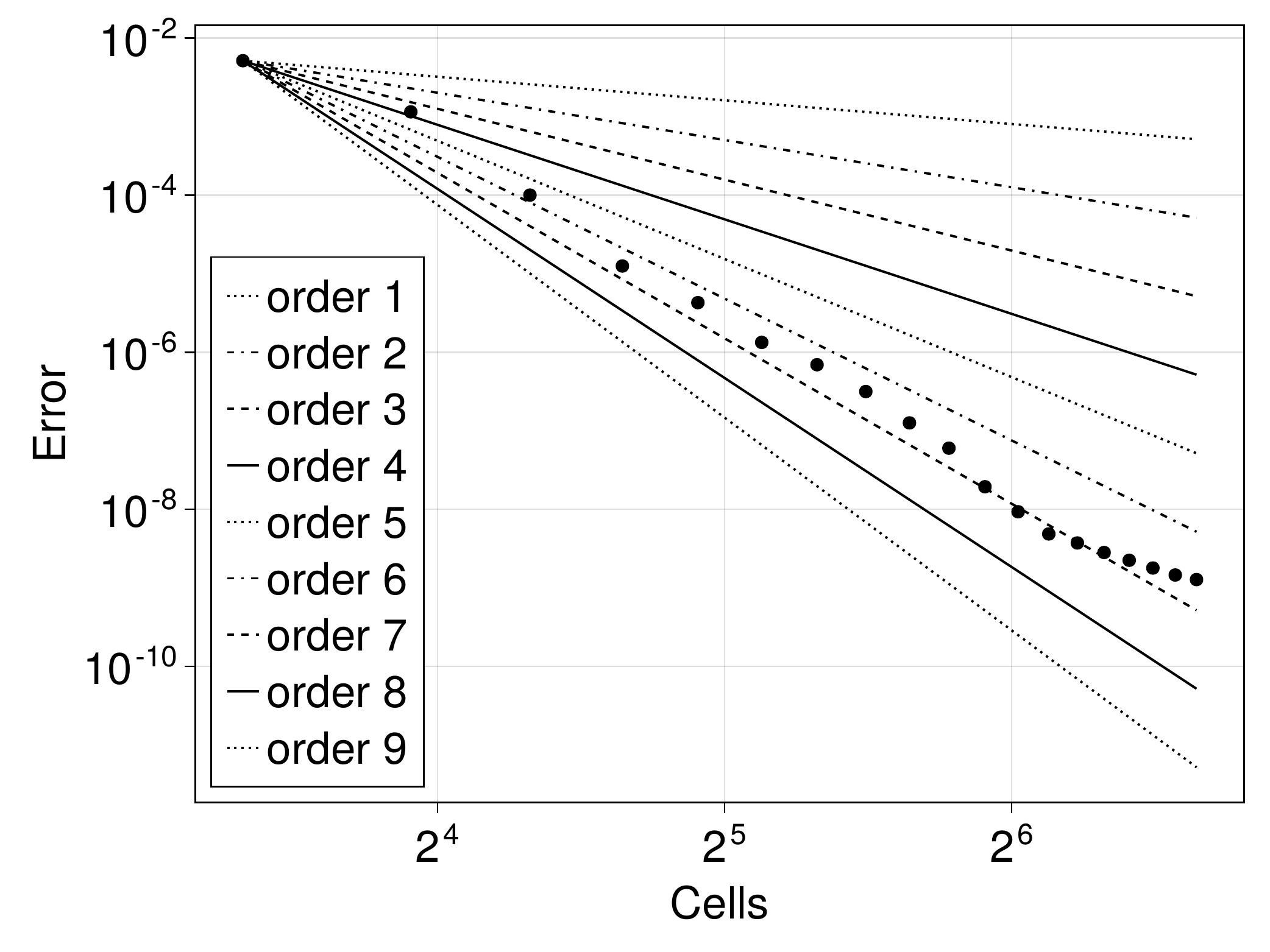}
	\caption{p = 3, $\Leb^1$ norm}
	\end{subfigure}
	\begin{subfigure}{0.5 \textwidth}
		\includegraphics[width=\textwidth]{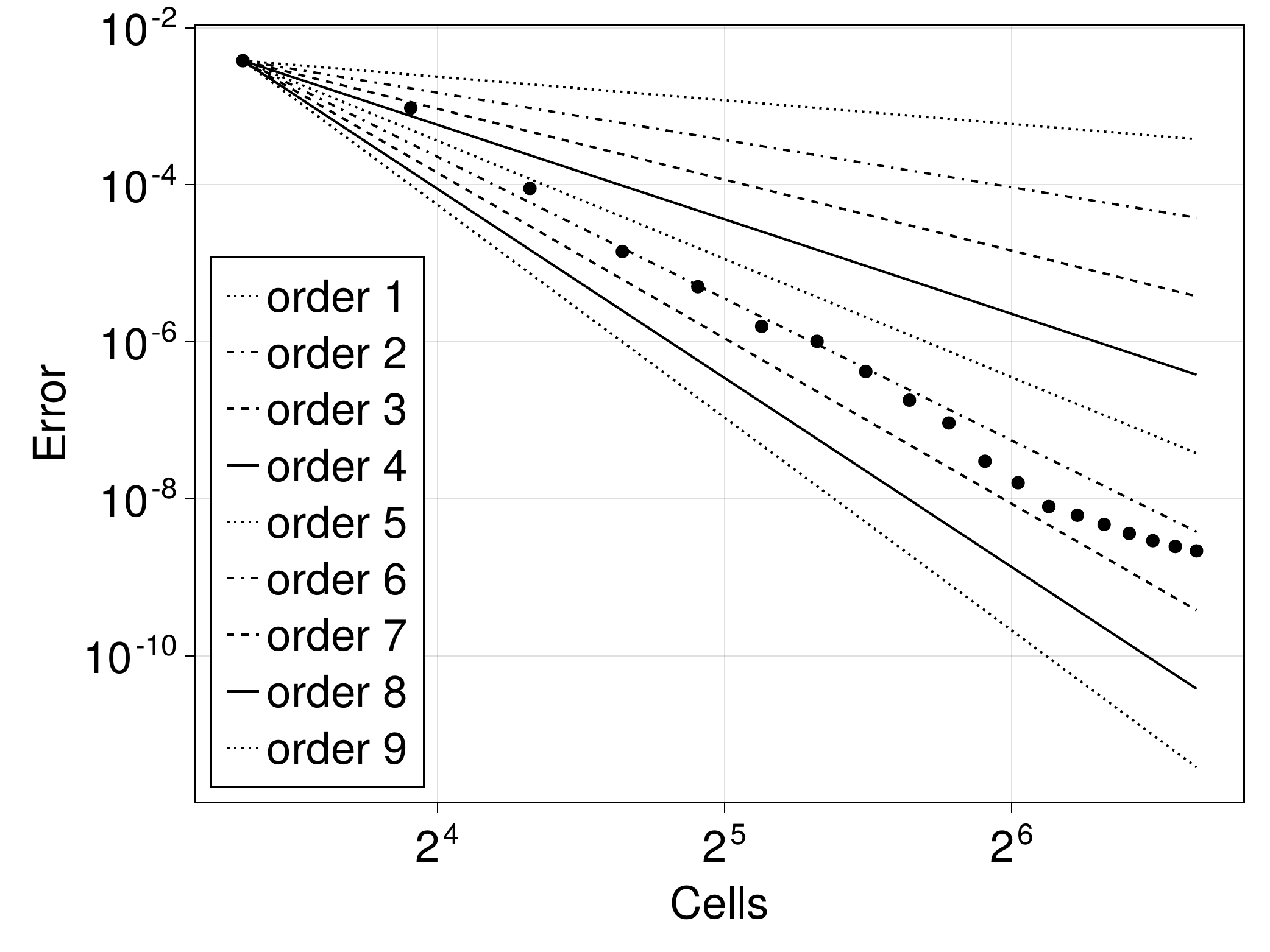}
		\caption{p = 7, $\Leb^2$ norm}
	\end{subfigure}
	\begin{subfigure}{0.5 \textwidth}
			\includegraphics[width=\textwidth]{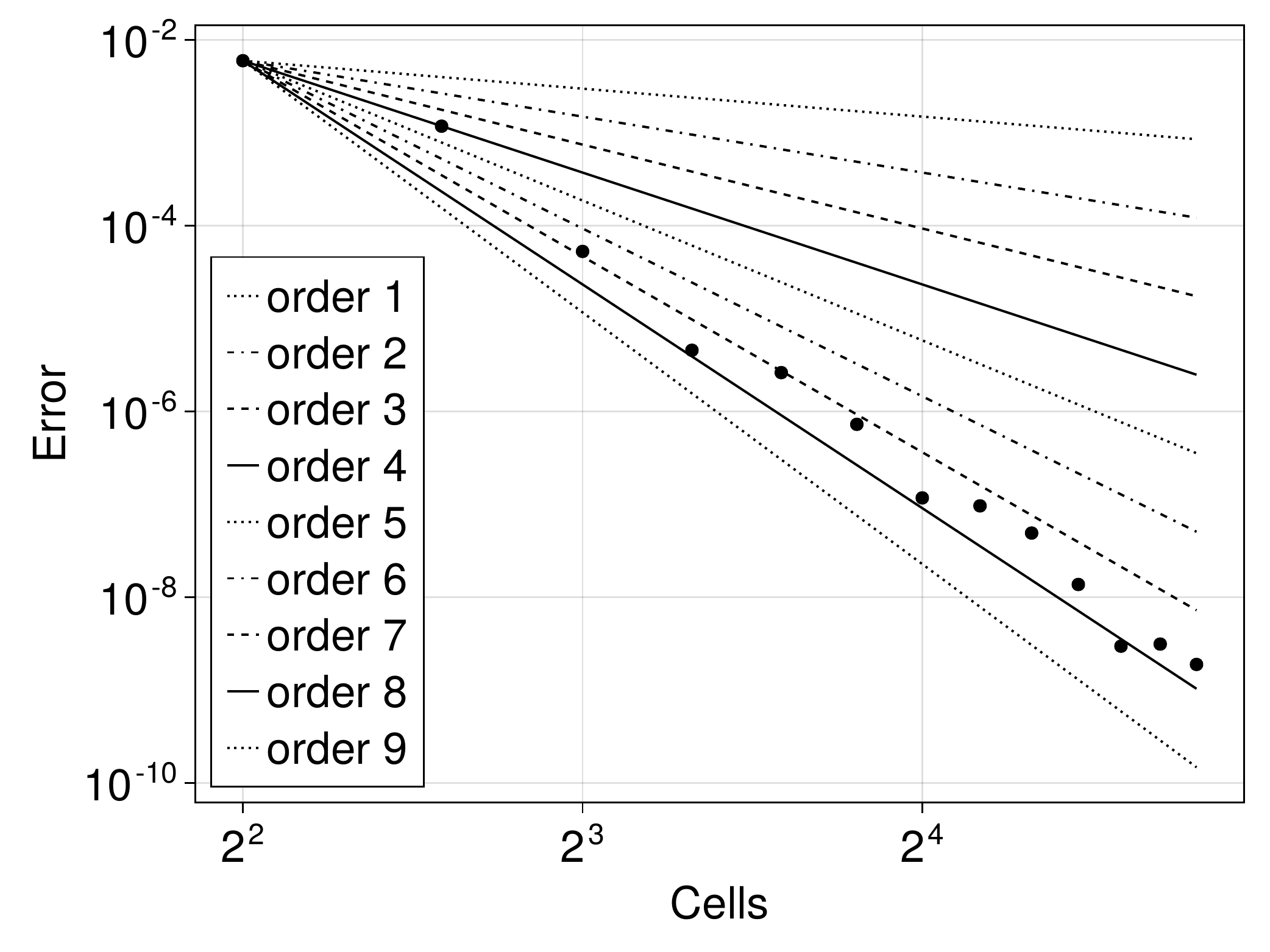}
		\caption{p = 7, $\Leb^1$ norm}
		\end{subfigure}
		\begin{subfigure}{0.5 \textwidth}
		\includegraphics[width=\textwidth]{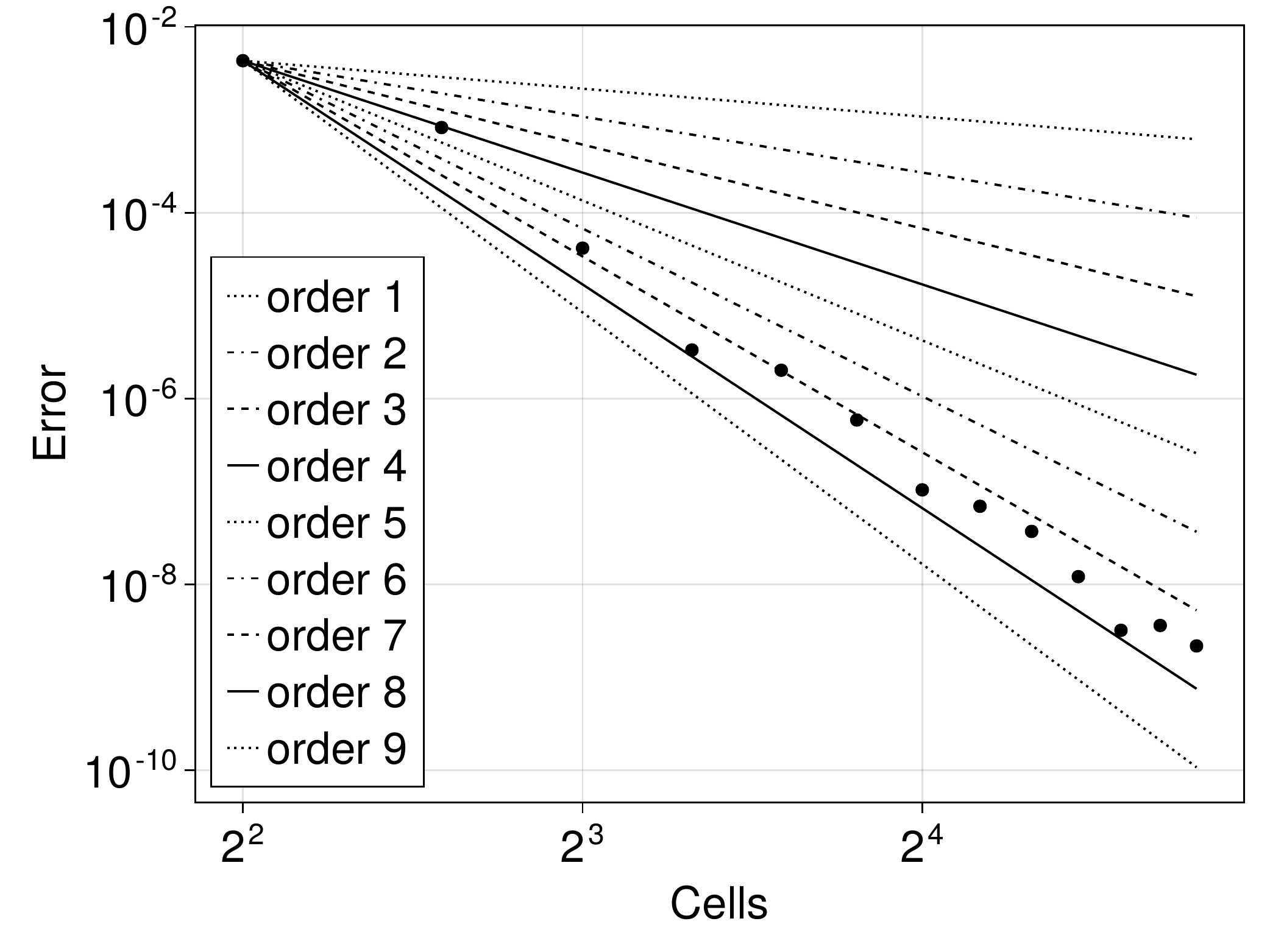}
		\caption{p = 7, $\Leb^2$ norm}
	\end{subfigure}
	\caption{Convergence Analysis for $p = 3$ and $p = 7$.}
	\label{fig:Convana}
\end{figure}
After the solution was calculated for $N = \{10, 15, 20, 25, 30, 40, 50, 60, 70, 80, 90, 100\}$ cells for $p = 3$ and with the same stepping up to $50$ cells for $p = 7$ up to $t = 5$ the $\Leb^1$ and $\Leb^2$ errors were calculated.
The convergence in figure \ref{fig:Convana} seems to take place with too high an order for the ansatz polynomials used. The reason for this could be that the accuracy of the basic scheme is significantly higher for these solutions than the accuracy of the corrected scheme, because the entropy dissipation estimate still falsely reports high amounts of entropy dissipation. When the grid is refined the entropy dissipation estimate converges with a higher speed than the basic scheme following lemma \ref{lem:eds} and because the error introduced to enforce the dissipation dominates a higher convergence speed than expected is observed.

\subsection{Timestep Analysis}
An important result of any modification to a basic scheme can be an impact on the allowed timestep size. In the first part of this publication \cite{klein2023stabilizing} this influence was tested by measuring the maximal timestep possible before a blow-up occurs. This was done once more.
\begin{figure}
	\begin{subfigure}{0.49\textwidth}
		\includegraphics[width=\textwidth]{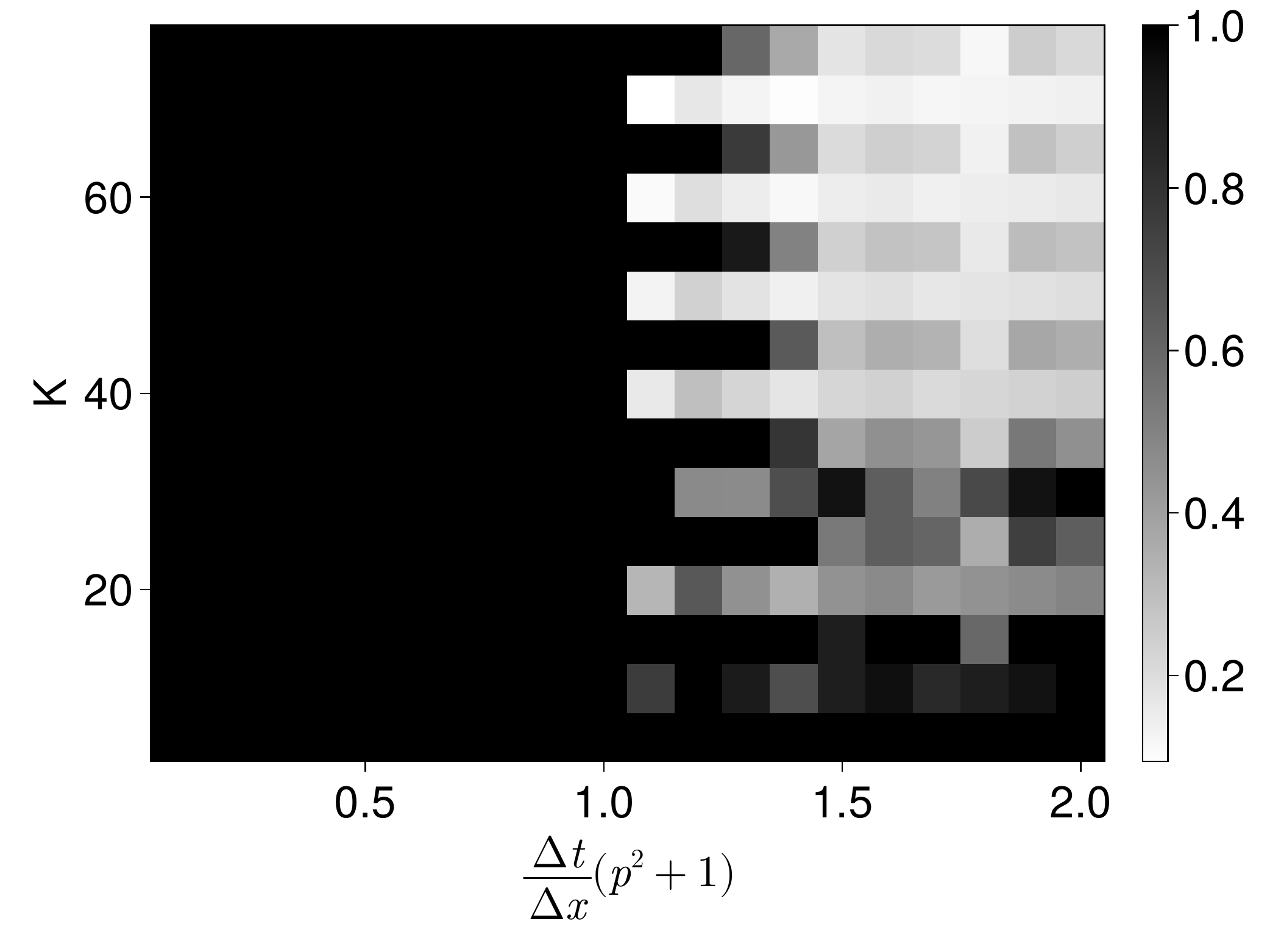}
	\end{subfigure}
	\begin{subfigure}{0.49\textwidth}
		\includegraphics[width=\textwidth]{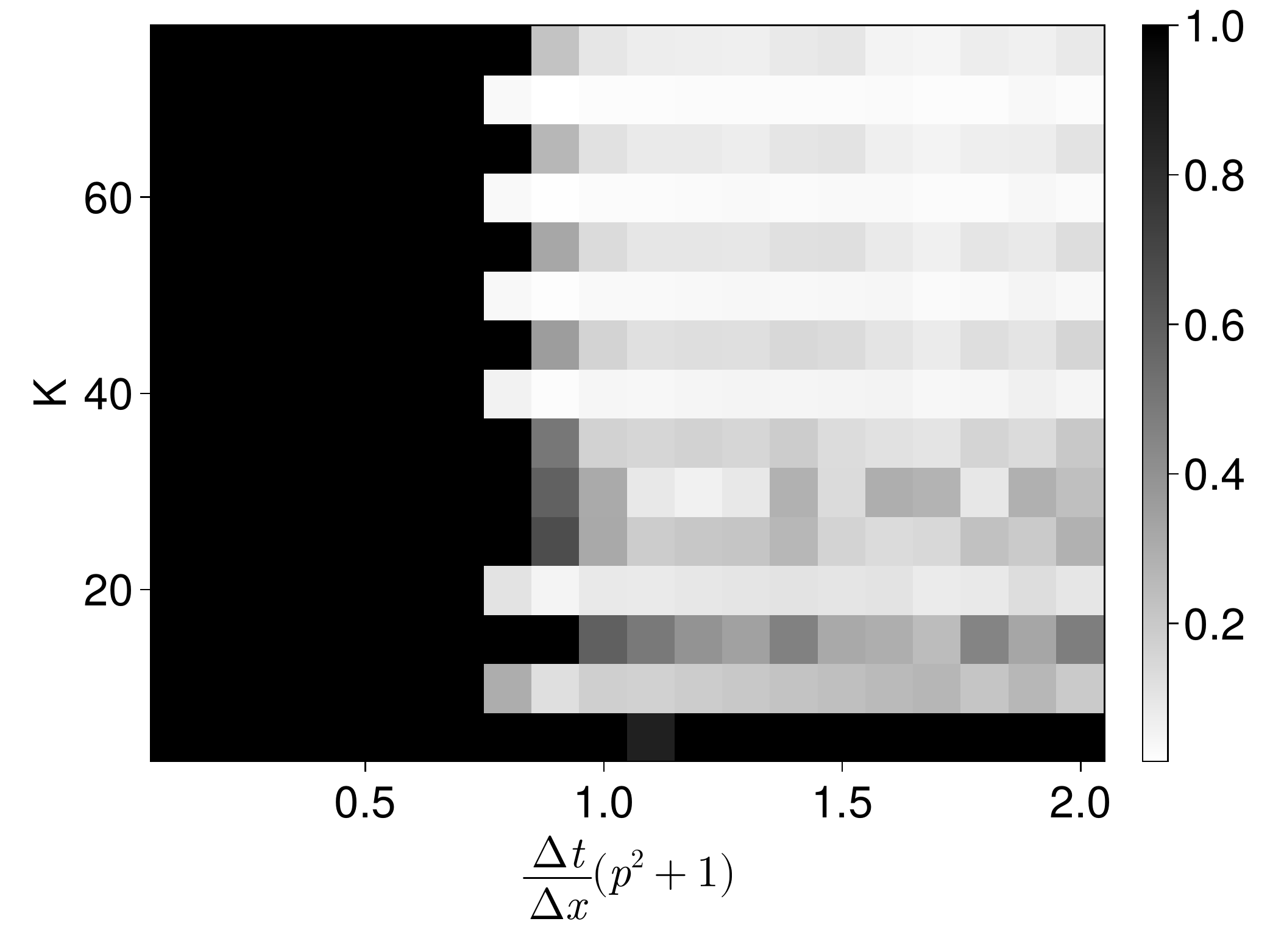}
	\end{subfigure}
	\caption{Maximal timestep sizes for $p = 3$ and $p = 7$ and the shock tube .}
	\label{fig:timestepana}
	\end{figure}
The maximal timestep possible for the first shock tube for orders $p = 3$ and $p=7$ is shown in figure \ref{fig:timestepana}. Obviously this timestep is acceptable and when corrected for the larger maximal wave speed of the Riemann problem used for testing, larger than the timestep reported in the previous part, highlighting the superiority of the new dissipation direction.

%% file: concl.tex
The method described in \cite{klein2023stabilizing} to enforce an entropy rate criterion for DG methods was improved. By using a direct indicator for the entropy dissipation the error indicator used before could be replaced, resulting in a lower dissipation in situations like contact discontinuities. For smooth solutions this new method to quantify the amount of dissipation needed converges significantly faster to zero than the error estimate used before, and therefore allows us to recover the convergence speed of the basic DG scheme that was reduced by one degree before. Further, the direct quantification of the entropy dissipation needed allowed us to consider different dissipation directions, especially combining smoothing and dissipation and therefore bridging into the field of modal filtering. The effectivity of the refined method was demonstrated for the Euler system of gas dynamics. The method is not only high order accurate but also able to handle shocks, contact discontinuities, and rarefactionwaves. The next logical steps can be the application to two-dimensional problems, the application of the designed entropy inequality predictors to other schemes like continuous Galerkin and Spectral Volume schemes, where several adjustments will be needed, and revisiting the splitting into a fully discrete scheme already explored in \cite{klein2023stabilizing}. The presented method to estimate the entropy dissipation needed could also be used with artificial viscosity shock-capturing as for example described in \cite{GNACS2019Smooth}.

%% file: main.bbl
\begin{thebibliography}{42}
\providecommand{\natexlab}[1]{#1}
\providecommand{\url}[1]{\texttt{#1}}
\expandafter\ifx\csname urlstyle\endcsname\relax
  \providecommand{\doi}[1]{doi: #1}\else
  \providecommand{\doi}{doi: \begingroup \urlstyle{rm}\Url}\fi

\bibitem[Ben-Artzi and Falcovitz(2011)]{BA2011GRP}
Matania Ben-Artzi and Joseph Falcovitz.
\newblock \emph{Generalized {Riemann} problems in computational fluid
  dynamics.}, volume~11 of \emph{Camb. Monogr. Appl. Comput. Math.}
\newblock Cambridge: Cambridge University Press, reprint of the 2003 hardback
  ed. edition, 2011.
\newblock ISBN 978-0-521-17327-8.
\newblock \doi{10.1017/CBO9780511546785}.

\bibitem[{Chavent} and {Cockburn}(1989)]{Cockburn1989DGI}
Guy {Chavent} and Bernardo {Cockburn}.
\newblock {The local projection \(P^ 0-P^ 1\)-discontinuous-Galerkin finite
  element method for scalar conservation laws}.
\newblock \emph{{RAIRO, Mod\'elisation Math. Anal. Num\'er.}}, 23\penalty0
  (4):\penalty0 565--592, 1989.
\newblock ISSN 0764-583X.
\newblock \doi{10.1051/m2an/1989230405651}.

\bibitem[Chen and Shu(2017)]{ShuDGReview}
Tianheng Chen and Chi-Wang Shu.
\newblock Entropy stable high order discontinuous {Galerkin} methods with
  suitable quadrature rules for hyperbolic conservation laws.
\newblock \emph{J. Comput. Phys.}, 345:\penalty0 427--461, 2017.
\newblock ISSN 0021-9991.
\newblock \doi{10.1016/j.jcp.2017.05.025}.

\bibitem[{Cockburn} and {Shu}(1989)]{CockburnShu1989DGI}
Bernardo {Cockburn} and Chi-Wang {Shu}.
\newblock {TVB Runge-Kutta local projection discontinuous Galerkin finite
  element method for conservation laws. II: General framework}.
\newblock \emph{{Math. Comput.}}, 52\penalty0 (186):\penalty0 411--435, 1989.
\newblock ISSN 0025-5718.
\newblock \doi{10.2307/2008474}.

\bibitem[Cockburn and Shu(2001)]{CS2001RKDG}
Bernardo Cockburn and Chi-Wang Shu.
\newblock Runge--{Kutta} discontinuous {Galerkin} methods for
  convection-dominated problems.
\newblock \emph{J. Sci. Comput.}, 16\penalty0 (3):\penalty0 173--261, 2001.
\newblock ISSN 0885-7474.
\newblock \doi{10.1023/A:1012873910884}.

\bibitem[Dafermos(1972)]{Dafermos72}
Constantine~M. Dafermos.
\newblock The entropy rate admissibility criterion for solutions of hyperbolic
  conservation laws.
\newblock \emph{Journal Of Differential Equations}, pages 202--212, 1972.

\bibitem[Dafermos(2009)]{Dafermos2009MDR}
Constantine~M. Dafermos.
\newblock A variational approach to the {Riemann} problem for hyperbolic
  conservation laws.
\newblock \emph{Discrete Contin. Dyn. Syst.}, 23\penalty0 (1-2):\penalty0
  185--195, 2009.
\newblock ISSN 1078-0947.
\newblock \doi{10.3934/dcds.2009.23.185}.

\bibitem[{Dafermos}(2016)]{Dafermos2016Hyper}
Constantine~M. {Dafermos}.
\newblock \emph{{Hyperbolic conservation laws in continuum physics}}, volume
  325.
\newblock Berlin: Springer, 2016.
\newblock ISBN 978-3-662-49449-3; 978-3-662-49451-6.
\newblock \doi{10.1007/978-3-662-49451-6}.

\bibitem[Dahlquist(1963)]{Dahlquist1963Special}
Germund~G. Dahlquist.
\newblock A special stability problem for linear multistep methods.
\newblock \emph{BIT, Nord. Tidskr. Inf.-behandl.}, 3:\penalty0 27--43, 1963.
\newblock ISSN 0006-3835.
\newblock \doi{10.1007/BF01963532}.

\bibitem[Evans(2010)]{Evans2010PDE}
Lawrence~C. Evans.
\newblock \emph{Partial differential equations}, volume~19 of \emph{Grad. Stud.
  Math.}
\newblock Providence, RI: American Mathematical Society (AMS), 2nd ed. edition,
  2010.
\newblock ISBN 978-0-8218-4974-3.

\bibitem[Feireisl(2014)]{Feireisl2014MD}
Eduard Feireisl.
\newblock Maximal dissipation and well-posedness for the compressible {Euler}
  system.
\newblock \emph{J. Math. Fluid Mech.}, 16\penalty0 (3):\penalty0 447--461,
  2014.
\newblock ISSN 1422-6928.
\newblock \doi{10.1007/s00021-014-0163-8}.

\bibitem[Gassner(2013)]{Gassner}
Gregor~J. Gassner.
\newblock A skew-symmetric discontinuous {Galerkin} spectral element
  discretization and its relation to {SBP}-{SAT} finite difference methods.
\newblock \emph{SIAM J. Sci. Comput.}, 35\penalty0 (3):\penalty0 a1233--a1253,
  2013.
\newblock ISSN 1064-8275.
\newblock \doi{10.1137/120890144}.

\bibitem[Glaubitz et~al.(2019)Glaubitz, Nogueira, Almeida, Cant{\~a}o, and
  Silva]{GNACS2019Smooth}
J.~Glaubitz, A.~C.~jun. Nogueira, J.~L.~S. Almeida, R.~F. Cant{\~a}o, and
  C.~A.~C. Silva.
\newblock Smooth and compactly supported viscous sub-cell shock capturing for
  discontinuous {Galerkin} methods.
\newblock \emph{J. Sci. Comput.}, 79\penalty0 (1):\penalty0 249--272, 2019.
\newblock ISSN 0885-7474.
\newblock \doi{10.1007/s10915-018-0850-3}.

\bibitem[Glaubitz et~al.(2018)Glaubitz, {\"O}ffner, and Sonar]{GOS2018Modal}
Jan Glaubitz, Philipp {\"O}ffner, and Thomas Sonar.
\newblock Application of modal filtering to a spectral difference method.
\newblock \emph{Math. Comput.}, 87\penalty0 (309):\penalty0 175--207, 2018.
\newblock ISSN 0025-5718.
\newblock \doi{10.1090/mcom/3257}.

\bibitem[Godlewski and Raviart(1991)]{GR91}
Edwige Godlewski and Pierre-Arnaud Raviart.
\newblock \emph{{H}yperbolic {S}ytems of {C}onservation {L}aws}.
\newblock ellipses, 1991.

\bibitem[Gottlieb et~al.(2001)Gottlieb, Shu, and Tadmor]{SSPRK}
Sigal Gottlieb, Chi-Wang Shu, and Eitan Tadmor.
\newblock Strong stability-preserving high-order time discretization methods.
\newblock \emph{SIAM Review}, 43, 05 2001.
\newblock \doi{10.1137/S003614450036757X}.

\bibitem[Harten(1983)]{Harten83b}
Amiram Harten.
\newblock On the symmetric form of systems of conservation laws with entropy.
\newblock \emph{Journal of Computational Physics}, 49:\penalty0 151--164, 1983.

\bibitem[Harten et~al.(1983)Harten, Lax, and van Leer]{HLL1984HLL}
Amiram Harten, Peter~D. Lax, and Bram van Leer.
\newblock On upstream differencing and {Godunov}-type schemes for hyperbolic
  conservation laws.
\newblock \emph{SIAM Rev.}, 25:\penalty0 35--61, 1983.
\newblock ISSN 0036-1445.
\newblock \doi{10.1137/1025002}.

\bibitem[Johnson(2009)]{Johnson2009FEM}
Claes Johnson.
\newblock \emph{Numerical solution of partial differential equations by finite
  element method.}
\newblock Mineola, NY: Dover Publications, reprint of the 1987 {English} ed.
  edition, 2009.
\newblock ISBN 978-0-486-46900-3.

\bibitem[Klein(2022)]{Klein2022Using}
Simon-Christian Klein.
\newblock Using the {Dafermos} entropy rate criterion in numerical schemes.
\newblock \emph{BIT}, 62\penalty0 (4):\penalty0 1673--1701, 2022.
\newblock ISSN 0006-3835.
\newblock \doi{10.1007/s10543-022-00927-x}.

\bibitem[Klein(2023)]{klein2023stabilizing}
Simon-Christian Klein.
\newblock Stabilizing discontinuous galerkin methods using dafermos’ entropy
  rate criterion: I—one-dimensional conservation laws.
\newblock \emph{Journal of Scientific Computing}, 95\penalty0 (2):\penalty0 55,
  2023.

\bibitem[Klein and Sonar(2023)]{KS2023EAR}
Simon-Christian Klein and Thomas Sonar.
\newblock Entropy-aware non-oscillatory high-order finite volume methods using
  the dafermos entropy rate criterion, 2023.
\newblock URL \url{https://arxiv.org/abs/2302.08971}.

\bibitem[Kress(1998)]{Kress1998Tikhonov}
Rainer Kress.
\newblock \emph{Ill-Conditioned Linear Systems}, pages 77--92.
\newblock Springer New York, New York, NY, 1998.
\newblock ISBN 978-1-4612-0599-9.
\newblock \doi{10.1007/978-1-4612-0599-9_5}.
\newblock URL \url{https://doi.org/10.1007/978-1-4612-0599-9_5}.

\bibitem[Lax(1971)]{Lax71}
Peter~D. Lax.
\newblock Shock waves and entropy.
\newblock \emph{Contributions to Nonlinear Functional Analysis}, pages
  603--634, 1971.

\bibitem[Lax(2002)]{LaxFun}
Peter~D. Lax.
\newblock \emph{{F}unctional {A}nalysis}.
\newblock Wiley Interscience, 2002.

\bibitem[Lax et~al.(1976)Lax, Burstein, and Lax]{Lax1976Calculus}
Peter~D. Lax, Samuel Burstein, and Anneli Lax.
\newblock \emph{Calculus with applications and computing. {Vol}. {I}}.
\newblock Undergraduate Texts Math. Springer, Cham, 1976.

\bibitem[Luo et~al.(2007)Luo, Baum, and Löhner]{LUO2007686}
Hong Luo, Joseph~D. Baum, and Rainald Löhner.
\newblock A hermite weno-based limiter for discontinuous galerkin method on
  unstructured grids.
\newblock \emph{Journal of Computational Physics}, 225\penalty0 (1):\penalty0
  686--713, 2007.
\newblock ISSN 0021-9991.
\newblock \doi{https://doi.org/10.1016/j.jcp.2006.12.017}.
\newblock URL
  \url{https://www.sciencedirect.com/science/article/pii/S0021999106006164}.

\bibitem[Ranocha(2019)]{ranocha2019mimetic}
Hendrik Ranocha.
\newblock Mimetic properties of difference operators: Product and chain rules
  as for functions of bounded variation and entropy stability of second
  derivatives.
\newblock \emph{BIT Numerical Mathematics}, 59\penalty0 (2):\penalty0 547--563,
  06 2019.
\newblock \doi{10.1007/s10543-018-0736-7}.

\bibitem[Ranocha et~al.(2018)Ranocha, Glaubitz, {\"O}ffner, and
  Sonar]{RGOS2018Stab}
Hendrik Ranocha, Jan Glaubitz, Philipp {\"O}ffner, and Thomas Sonar.
\newblock Stability of artificial dissipation and modal filtering for flux
  reconstruction schemes using summation-by-parts operators.
\newblock \emph{Appl. Numer. Math.}, 128:\penalty0 1--23, 2018.
\newblock ISSN 0168-9274.
\newblock \doi{10.1016/j.apnum.2018.01.019}.

\bibitem[Rudin(1966)]{Rudin1966RCA}
Walter Rudin.
\newblock Real and complex analysis.
\newblock {McGraw}-{Hill} {Series} in {Higher} {Mathematics}. {New} {York}
  etc.: {McGraw}-{Hill} {Book} {Company}. xi, 412 p. (1966)., 1966.

\bibitem[Shu and Osher(1988)]{SO1988}
Chi-Wang Shu and Stanley Osher.
\newblock Efficient implementation of essentially non-oscillatory
  shock-capturing schemes.
\newblock \emph{Journal of Computational Physics}, 77:\penalty0 439--471, 1988.

\bibitem[Shu and Osher(1989)]{SO1989}
Chi-Wang Shu and Stanley Osher.
\newblock Efficient implementation of essentially non-oscillatory
  shock-capturing schemes {II}.
\newblock \emph{Journal of Computational Physics}, 83:\penalty0 439--471, 1989.

\bibitem[Sonar(2016)]{SONAR201655}
T.~Sonar.
\newblock Chapter 3 - classical finite volume methods.
\newblock In Rémi Abgrall and Chi-Wang Shu, editors, \emph{Handbook of
  Numerical Methods for Hyperbolic Problems}, volume~17 of \emph{Handbook of
  Numerical Analysis}, pages 55--76. Elsevier, 2016.
\newblock \doi{https://doi.org/10.1016/bs.hna.2016.09.005}.
\newblock URL
  \url{https://www.sciencedirect.com/science/article/pii/S157086591630014X}.

\bibitem[Sonntag and Munz(2014)]{MS2014FVS}
Matthias Sonntag and Claus-Dieter Munz.
\newblock Shock capturing for discontinuous {Galerkin} methods using finite
  volume subcells.
\newblock In J{\"u}rgen Fuhrmann, Mario Ohlberger, and Christian Rohde,
  editors, \emph{Finite Volumes for Complex Applications VII-Elliptic,
  Parabolic and Hyperbolic Problems}, pages 945--953, Cham, 2014. Springer
  International Publishing.
\newblock ISBN 978-3-319-05591-6.

\bibitem[Tadmor(1984{\natexlab{a}})]{Tadmor1984I}
Eitan Tadmor.
\newblock The large-time behavior of the scalar, genuinely nonlinear
  {Lax}-{Friedrichs} scheme.
\newblock \emph{Math. Comput.}, 43:\penalty0 353--368, 1984{\natexlab{a}}.
\newblock ISSN 0025-5718.
\newblock \doi{10.2307/2008281}.

\bibitem[Tadmor(1984{\natexlab{b}})]{Tadmor1984II}
Eitan Tadmor.
\newblock Numerical viscosity and the entropy condition for conservative
  difference schemes.
\newblock \emph{Math. Comput.}, 43:\penalty0 369--381, 1984{\natexlab{b}}.
\newblock ISSN 0025-5718.
\newblock \doi{10.2307/2008282}.

\bibitem[Tadmor(1987)]{Tadmor1987}
Eitan Tadmor.
\newblock The numerical viscosity of entropy stable schemes for systems of
  conservation laws.
\newblock \emph{Mathematics of Computation}, 49:\penalty0 91--103, 1987.

\bibitem[Tadmor(2003)]{Tadmor2003}
Eitan Tadmor.
\newblock Entropy stability theory for difference approximations of nonlinear
  conservation laws and related time dependent problems.
\newblock \emph{Acta Numerica}, pages 451--512, 2003.

\bibitem[{Toro}(2009)]{Toro2009Riemann}
Eleuterio~F. {Toro}.
\newblock \emph{{Riemann solvers and numerical methods for fluid dynamics. A
  practical introduction}}.
\newblock Berlin: Springer, 2009.
\newblock ISBN 978-3-540-25202-3; 978-3-540-49834-6.
\newblock \doi{10.1007/b79761}.

\bibitem[Wang(2002)]{Wang2002SV}
Z.~J. Wang.
\newblock Spectral (finite) volume method for conservation laws on unstructured
  grids. {Basic} formulation.
\newblock \emph{J. Comput. Phys.}, 178\penalty0 (1):\penalty0 210--251, 2002.
\newblock ISSN 0021-9991.
\newblock \doi{10.1006/jcph.2002.7041}.
\newblock URL
  \url{semanticscholar.org/paper/4e9aef2784954f33505124a59c0e69d2248a2c96}.

\bibitem[Wanner et~al.(1978)Wanner, Hairer, and N{\o}rsett]{WHN1978Order}
Gerhard Wanner, Ernst Hairer, and Syvert~P. N{\o}rsett.
\newblock Order stars and stability theorems.
\newblock \emph{BIT, Nord. Tidskr. Inf.-behandl.}, 18:\penalty0 475--489, 1978.
\newblock ISSN 0006-3835.
\newblock \doi{10.1007/BF01932026}.

\bibitem[Zhu and Qiu(2011)]{ZHU20114353}
Jun Zhu and Jianxian Qiu.
\newblock Local {DG} method using {WENO} type limiters for
  convection–diffusion problems.
\newblock \emph{Journal of Computational Physics}, 230\penalty0 (11):\penalty0
  4353--4375, 2011.
\newblock ISSN 0021-9991.
\newblock \doi{https://doi.org/10.1016/j.jcp.2010.03.023}.
\newblock URL
  \url{https://www.sciencedirect.com/science/article/pii/S0021999110001336}.
\newblock Special issue High Order Methods for CFD Problems.

\end{thebibliography}
